\documentclass[12pt]{amsart}
\usepackage{amssymb, amsmath}
\usepackage{enumitem}
\usepackage{mathtools}
\usepackage{verbatim}
\usepackage{amsthm}
\usepackage{caption}
\usepackage{soul}
\usepackage{bbm}
\usepackage{hyperref}
\usepackage{fp}
\usepackage{environ}
\usepackage{bm}
\usepackage[usenames,dvipsnames]{xcolor}
\usepackage{unixode}
\usepackage{etoolbox}

\usepackage[normalem]{ulem}

\definecolor{kwred}{RGB}{0,0,0}%{160,0,0}
\newtoggle{newenv}
\newcommand{\newcolor}{\color{kwred}}
\newcommand{\oldcolor}{\color{black}}
\newenvironment{enew}{\newcolor\toggletrue{newenv}}{\oldcolor\togglefalse{newenv}}
\newcommand{\new}[1]{\iftoggle{newenv}{#1}{\color{kwred}#1\color{black}}}
\newcommand{\removed}[1]{\new{\sout{{ #1}}}}

\renewcommand{\removed}[1]{}

\def\tempmake#1{\def\tempmaker##1{\def\temp{{\csname #1\endcsname{##1}}}\expandafter\let\csname ##1\endcsname\temp}}
\def\makemake#1#2{\tempmake{#2}\expandafter\let\csname #1\endcsname\tempmaker}
\def\makeme#1{\makemake{make#1}{math#1}}
\makeme{bb}
\makeme{cal}
\makeme{bf}
\makeme{frak}

\makemake{makeop}{operatorname}
\makeop{PSL}
\makeop{Vol}
\makeop{ad}
\makeop{SL}
\makeop{GL}
\makeop{Deck}

\makefrak{sl}

\makebb{R}
\makebb{Q}
\makebb{C}
\makebb{Z}
\makebb{N}
\def\d{\partial}

\makebb{H}
\makebb{T}

\makecal{A}
\makecal{F}
\makecal{U}
\newcommand{\G}{\Isom(\H^3)}
\newcommand{\nbhd}{\mathcal N}
\makebf{P}

\let\from\colon
\def\Re{\operatorname{Re}}
\DeclareMathOperator{\Isom}{Isom}

\DeclareMathOperator{\csch}{csch}
\DeclareMathOperator{\sech}{sech}
\let\Im\Pim

\def\suchThat{:}
\newcommand{\neigh}{\nbhd}
\newcommand{\ddm}[1]{\diam{\d #1}}
\let\following\circ
\newcommand{\alert}[1]{\textbf{#1}}

\newcommand{\cem}{C(\epsilon, M)}
\newcommand{\size}[1]{\# #1}
\newcommand{\sizep}[1]{\size{(#1)}}
\newcommand{\measure}[1]{\left|#1\right|}
\newcommand{\umbrella}{U}

\newcommand{\umb}{\text{umb}}
\newcommand{\randomUmbrella}{\hat{\umbrella}}
\newcommand{\randomUmbrellaBoundarySize}[1]{R^{N_{\new{\umb}}} e^{K\max(0, #1 - h_T) + 2h_T}/|\Gamma|}
\newcommand{\dumb}{\partial _{\operatorname{ext}}U}
\newcommand{\dout}{\d_{\operatorname{out}}}

\newcommand{\goodChains}{C_{\epsilon, R}}
\newcommand{\goodCurves}{\Gamma_{\epsilon, R}}
\newcommand{\ogoodCurves}{\Gamma^*_{\epsilon, R}}

\newcommand{\goodCurvesNumber}{\epsilon^2 e^{4R}/R}
\newcommand{\goodPants}{\Pi_{\epsilon, R}}
\newcommand{\ogoodPants}{\Pi^*_{\epsilon, R}}
\newcommand{\oogoodPants}{\hat\Pi^*_{\epsilon, R}}
\newcommand{\goodCurvesWithHeight}[1]{\goodCurves^{<#1}}
\newcommand{\ogoodCurvesWithHeight}[1]{\goodCurves^{*<#1}}

\newcommand{\goodPantsWithHeight}[1]{\Pi^{<#1}_{\epsilon, R}}
\newcommand{\goodPantsWithHeightMoreThan}[1]{\Pi^{\ge #1}_{\epsilon, R}}
\newcommand{\ogoodPantsWithHeightMoreThan}[1]{\Pi^{*\ge #1}_{\epsilon, R}}
\newcommand{\ogoodPantsWithHeight}[1]{\Pi^{*<#1}_{\epsilon, R}}

\newcommand{\Group}{\G}
\newcommand{\Pants}{\Pi}
\newcommand{\pants}{\pi} % may have to change
\newcommand{\Curves}{\Gamma}
\newcommand{\dc}{d_\C}
\newcommand{\isomHThree}{\G}
\let\IsomHThree\isomHThree
\newcommand{\isomHTwo}{\Isom(\H^2)}
\newcommand{\HThree}{\H^3}
\newcommand{\FrameHThree}{\mathcal{F}\HThree}
\newcommand{\theManifold}{\isomHThree/\Lattice}

\newcommand{\transportAngle}{\mathbf{\theta}} % Doesn't work!
\newcommand{\transportCL}{w}

\newcommand{\squareRootUnitNormalBundle}[1]{N^1(\sqrt{#1})}
\newcommand{\df}{\d^{\F}}
\newcommand{\Mobius}{M\"obius}

\newcommand{\normalGroup}[1]{\C/(\length(#1)\Z + 2 \pi i \Z)}
\newcommand{\normalHalfGroup}[1]{\C/(2 \pi i \Z + \hll(\gamma) \Z)}
\newcommand{\invariants}{\mathbf I}
\newcommand{\goodInvariants}[1]{\mathbb I_{\epsilon, R}(#1)}
\newcommand{\goodInvariantsZ}[1]{\mathbb I_{\epsilon, R}^0(#1)}
\newcommand{\spaceOfInvariants}{\mathbb I}
\newcommand{\numberOfConnections}{\mathbf n}

\newcommand{\crunchedGoodPantsPerCurve}{{\numberOfGoodPants}/{\numberOfGoodCurves}}
\newcommand{\goodPantsPerCurve}{\frac{\numberOfGoodPants}{\numberOfGoodCurves}}
\newcommand{\numberOfGoodPants}{\abs{\Pi}}
\newcommand{\numberOfGoodCurves}{\abs{\Gamma}}

\newcommand{\umbrellaBoundarySize}[1]{R^2 e^{K\pnnpart{h(#1) - h_T}}}
\newcommand{\umbrellaNumberOfComponents}[1]{R e^{K(\pnnpart{h(#1) - h_T)}}}
\newcommand{\baldSpotSize}{e^{2(h(\gamma) - h_C)}}

\newcommand{\HWUnnecessaryExcursions}{4\log R}
\newcommand{\maxStepThreeUmbrellaBend}{30\epsilon} 
\newcommand{\minStepThreeUmbrellaBend}{20\epsilon}
\newcommand{\maxVectorFieldBend}{100\epsilon}
\newcommand{\PTiltUp}{10\epsilon}

\newcommand{\gp}{\Pi_{\epsilon, R}}

\newcommand{\lattice}{\Gamma}
\newcommand{\rootnormal}{N^1(\sqrt \gamma)}
\newcommand{\hll}{\mathbf{hl}}
\newcommand{\length}{\mathbf{l}}
\newcommand{\foot}{\mathbf{foot}}

\def\norm#1{\left\Vert {#1} \right \Vert}
\def\abs#1{\left| {#1} \right|}
\def\inv{^{-1}}
\def\invof#1{\inv(#1)}
\newcommand\nnpart[1]{\max(0, #1)}
\newcommand\pnnpart[1]{\nnpart{#1}}
\newcommand{\itutu}[4]{#1 & #2 \\ #3 & #4}
\newcommand{\tutu}[4]{\left(\begin{smallmatrix} \itutu{#1}{#2}{#3}{#4} \end{smallmatrix} \right)}

\newcommand{\cheesef}[1]{\footnote{#1}}

\def\<{\left<}
\def\>{\right>}

\newcommand{\e}{\epsilon}
\renewcommand{\bold}[1]{\medskip \noindent {\bf #1 }\nopagebreak}
\newcommand{\ann}[1]{\marginpar{\tiny{#1}}}
\newcounter{annc}
\newcommand{\annf}[1]{\stepcounter{annc}$^{\theannc}$\ann{\theannc: #1}}

\newcommand{\MargulisConst}{\epsilon_0}
\newcommand{\Lattice}{\Gamma}
\newcommand{\diam}[1]{\operatorname{diam}(#1)}

\newenvironment{tododo}{\subsection*{To Do}\begin{enumerate}}{\end{enumerate}\medskip}

\let\draftnewpage\newpage

\renewcommand{\bold}[1]{\medskip \noindent {\bf #1 }\nopagebreak}

\let\nt\newtheorem
\nt{theorem}{Theorem}[section]
\nt{lemma}[theorem]{Lemma}
\nt{corollary}[theorem]{Corollary}
\theoremstyle{definition}
\nt{definition}[theorem]{Definition}
\nt{remark}[theorem]{Remark}
\nt{question}[theorem]{Question}
\nt{pthm}[theorem]{Proposed Theorem}

\newcommand{\ntentry}[3]{$#1$& \texttt{#2} & #3\\}
\newenvironment{notationTable}{\noindent\begin{tabular}{cll}}{\end{tabular}}

\numberwithin{equation}{subsection}
\renewcommand{\draftnewpage}{}
\renewcommand{\annf}[1]{}
\RenewEnviron{todo}{}
\RenewEnviron{tododo}{}
\renewcommand{\alert}[1]{}
\renewcommand{\cheesef}[1]{}

\title[Surface subgroups of finite covolume Kleinian groups]{Nearly Fuchsian surface subgroups of finite covolume Kleinian groups}
\author{Jeremy Kahn and Alex Wright}
\begin{document}
\begin{abstract}
Let $\Lattice < \PSL_2(\C)$ be discrete, cofinite volume, and noncocompact. We prove that for all $K > 1$, there is a subgroup $H < \Lattice$ that is $K$-quasiconformally conjugate to a discrete cocompact subgroup of $\PSL_2(\R)$. Along with previous work of Kahn and Markovic \cite{KM:Immersing}, this proves that every finite covolume Kleinian group has a nearly Fuchsian surface subgroup. 
\end{abstract}
\maketitle
\thispagestyle{empty}
%\begin{todo}
%%\item
%%Fix up Section 6 according to AW's comments.
%\item
%See if there are any marginal comments or alerts (boldface in text) that have to be dealt with.
%\item 
%Read the whole paper and see how you feel about it. 
%\item
%Decide what to do with the counting connections paper that we're citing. 
%%\item
%%Read the paper as a whole?
%\end{todo}

\tableofcontents
\begin{comment}
%%%%%%%%%%%%%%%%%%%%%%%%%%%%%%%%%%%%%%%%%%%
%\section*{Recent changes}
%%%%%%%%%%%%%%%%%%%%%%%%%%%%%%%%%%%%%%%%%%%
\newpage
%%%%%%%%%%%
\section*{Table of Notation}
%%%%%%%%%%%
\begin{notationTable}
\ntentry{\Gamma}{Lattice}{nonuniform lattice of isometries}
\ntentry{\isomHThree}{isomHThree}{ group of all hyperbolic isometries}
\ntentry{\MargulisConst}{MargulisConst}{horoballs with injectivity radius $\leq \MargulisConst$ are disjoint}
\ntentry{\Curves_{\epsilon, R}^h}{Curves}{$(R, \e)$-good closed geodesics  with height  $\leq h$}
\ntentry{\Pants_{\epsilon, R}^{h}}{Pants}{$(R, \e)$-good pants with height  $\leq h$}
\ntentry{h_C}{}{cutoff height, a multiple of $\log R$}
\ntentry{A_0}{}{$\Pants_{\epsilon, R}^{h_C}$}
\ntentry{A_0'}{}{set of $(P, \gamma)$,  $P\in A_0$, initially unpaired}
\ntentry{h_T}{}{target height, a multiple of $\log R$,  $h_T\ll h_C$}
\ntentry{Q(P, \gamma)}{}{average of umbrellas}
\ntentry{A_1}{}{$A_0$ plus $Q$ of $A_0'$}
\ntentry{g}{}{correction operator $\Q\Curves_{\epsilon, R}^{h_T}\to\Q\Curves_{\epsilon, R}^{h_T+C_{gpc}}$}
\ntentry{C_{gpc}}{}{height increase of correction operator}
\ntentry{\FrameHThree}{FrameHThree}{frame bundle}
\ntentry{A}{}{geodesic flow $\left<g_t\right>$ in direction of first vector}
\ntentry{C_A}{}{centralizer of $A$}
\ntentry{Y}{}{}
\ntentry{Q}{}{}
\ntentry{h_h}{}{$h_T + C_{gpc}$, $A_0$ well distributed above this}
\ntentry{}{}{}
\ntentry{}{}{}
\ntentry{}{}{}
\end{notationTable} 
\end{comment}

%%%%%%%%%%%%%%%%%%%%%%%%%%%%%%%%%%%%%%%%%%%
\section{Introduction}
%%%%%%%%%%%%%%%%%%%%%%%%%%%%%%%%%%%%%%%%%%%
%\subsection*{To Do}
%\begin{enumerate}
%\item 
%\end{enumerate}

\bold{Main result.} In this paper, we study the geometry of finite volume, complete, non-compact hyperbolic three manifolds $M=\H^3/\Gamma$, or equivalently their associated Kleinian groups $\Lattice < \PSL_2(\C)$.  We say that a surface subgroup $H < \Lattice$ is \emph{$K$-quasi-Fuchsian} if its action on $\partial \H^3$ is $K$-quasiconformally conjugate to the action of a cocompact Fuchsian group of $\PSL_2(\R)$. We say that a collection of quasi-Fuchsian surface subgroups is \emph{ubiquitous} if, for any pair of  hyperbolic planes $\Pi, \Pi'$ in $\H^3$ with distance $d(\Pi, \Pi')>0$, there is a surface subgroup whose boundary circle lies between $\partial \Pi$ and $\partial \Pi'$. Our main result is the following.  

\begin{theorem}\label{main theorem}
Assume $\H^3/\Gamma$ has finite volume, and is not compact. For all $K > 1$, the collection of $K$-quasi-Fuchsian surface subgroups of $\Gamma$ is 
ubiquitous. 
\end{theorem}

Informally speaking, this result says that every cusped hyperbolic three manifold contains a great many almost isometrically immersed closed hyperbolic surfaces.  

\bold{Related results.}  Kahn and Markovic proved the analogous statement when $\Gamma$ is cocompact \cite{KM:Immersing}. This was a key tool in Agol's resolution of the virtual Haken conjecture \cite{agol}. Kahn and Markovic went on to prove the Ehrenpreis Conjecture using related methods \cite{KM:Ehrenpreis}.  

Prior to Theorem \ref{main theorem}, there were no known examples of $\Gamma$ as in Theorem \ref{main theorem} that contained a $K$-quasi-Fuchsian surface subgroup for all $K>1$, except for trivial examples containing a Fuchsian surface subgroup. However, Masters and Zhang showed that every such $\Gamma$ has a $K$-quasi-Fuchsian surface subgroup for some $K>1$ \cite{masters-zhang-2008,  masters-zhang-2009}. See also Baker and Cooper \cite{baker-cooper}.

Independently of the current work, and using different methods, Cooper and Futer proved that the collection of $K$-quasi-Fuchsian surface subgroups is ubiquitous for \emph{some} $K>1$ \cite{cooper-futer}. This result, and  also Theorem \ref{main theorem}, answers a question of Agol \cite[Question 3.5]{delp-hoffoss-manning}. Cooper and Futer use their results to give a new proof \cite[Corollary 1.3]{cooper-futer} of the result of Wise \cite{wise} that $\Gamma$ acts freely and cocompactly on a $CAT(0)$ cube complex. Later Groves and Manning \cite{groves-manning} used Cooper and Futer's results to give a new proof of the result of Wise \cite{wise} that $\Gamma$ is virtually compact special.

Hamenst\"adt showed that  cocompact lattices in rank one simple Lie groups of non-compact type distinct from $SO(2m, 1)$ contain surface subgroups \cite{hamenstadt}. Kahn, Labourie, and Mozes have proven that cocompact lattices in complex simple Lie groups of non-compact type contain surface subgroups that are ``closely aligned'' with an $\SL_2(\R)$ subgroup \cite{kahn-labourie-mozes}. 

\bold{Motivation and hopes for the future.} One may view Theorem \ref{main theorem}, as well as \cite{KM:Immersing, KM:Ehrenpreis, hamenstadt}, as special cases of the general question of whether lattices in Lie groups $\mathcal{G}$ contain surface subgroups that are ``close" to lying in a given subgroup isomorphic to $PSL(2,\R)$. The solution to the original surface subgroup conjecture is the case where $\mathcal{G}=PSL(2,\C)$ and $\Gamma$ is cocompact; the solution to the Ehrenpreis conjecture is the case where $\mathcal{G}=PSL(2,\R)\times PSL(2,\R)$ and $\Gamma=\Gamma_0\times\Gamma_0$ is cocompact; and Hamenst\"adt's work concerns the case where $\mathcal{G}$ is a rank one Lie group, $\mathcal{G}\neq SO(2m, 1)$,  and $\Gamma$ is cocompact. 

Other cases of special interest include the punctured Ehrenpreis Conjecture, where $\mathcal{G}=PSL(2,\R)\times PSL(2,\R)$ and $\Gamma=\Gamma_0\times\Gamma_0$ is not cocompact, and higher rank Lie groups such as $\mathcal{G}=SL(n,\R)$. 

We hope that the methods in this paper will be applied to other cases where $\Gamma$ is not cocompact. In many cases, significant additional work is required.

%The methods in this paper are already being applied to the punctured Ehrenpreis Conjecture (forthcoming work of Kahn and Kontorovich), and we hope they will be similarly helpful in all cases when $\Gamma$ is not cocompact. (In each new case, including the punctured Ehrenpreis Conjecture, significant additional work is required.) 

Farb-Mosher  asked if there is a convex cocompact surface subgroup of the mapping class group \cite[Question 1.9]{MCG}. By work of Farb-Mosher and Hamenst\"adt, this is equivalent to asking of there is a surface bundle over a surface whose fundamental group is Gromov hyperbolic \cite{hamenstadt:extensions}, a question previously asked by Reid \cite[Question 4.7]{reid}.

There are several difficulties in applying the approach of  \cite{KM:Immersing} to the mapping class group, including that the moduli space of Riemann surfaces is not compact, and it is not homogeneous. %(Furthermore, the Teichm\"uller metric is not Riemannian and is not non-positively curved, and the Weil-Petersson metric is not complete.)  
Our primary goal in proving Theorem \ref{main theorem} was to attempt to isolate and resolve some of the difficulties of working with a non-compact space. 

\bold{Finding surface subgroups in the compact case.} Kahn and Markovic's strategy is to build quasi-Fuchsian subgroups of the fundamental group of  a \emph{compact} hyperbolic  3-manifold out of good pants. Each pants can be viewed as a map from a three-holed sphere into the 3-manifold, and one says the pants is good if the complex lengths of the geodesic representatives of the three peripheral curves (``cuffs") are within a small constant $\epsilon$ of a given large constant $2R$. Each such geodesic is a cuff for a great many pants, and these pants are equidistributed about the cuff. Kahn and Markovic very carefully glue pairs of pants with common cuffs to obtain a  quasi-Fuchsian surface. This requires the  equidistribution of pants around each possible cuff, which in turn depends on uniform and exponential mixing of the geodesic (and frame) flow.  

\bold{The need for a cutoff in the noncompact case.}
When the 3-manifold is not compact, 
mixing does not imply the same equidistribution estimates in regions where the injectivity radius is small.
In order to obtain the required equidistribution that was used in the cocompact case, 
we have to restrict ourselves to good pants that do not go too far out the cusp. 

\bold{The appearance of ``bald spots".}
By restricting the pants that we can use based on height---how far out they go into the cusp---we create ``bald spots" around some good curves: there are regions in the unit normal bundle to the curve that have no feet of the pants that we want to use. (The pants that have feet in these regions go too far out the cusp to be included in our construction).  The resulting imbalance makes it impossible to construct a nearly Fuchsian surface using one of every good pants that does not go too far out the cusp, and far more complicated, if not impossible, to construct a quasi-Fuchsian one.  

\bold{Umbrellas.} The main new idea of this paper is to first build, by hand, quasi-Fuchsian surfaces with one boundary component that goes into the cusp, and all remaining boundary components in the thick part. These umbrellas, as we call them, are used to correct the failure of equidistribution, and we wish to assemble a collection of umbrellas and good pants into a quasi-Fuchsian surface.   

\begin{figure}[ht!]
\includegraphics[width=\linewidth]{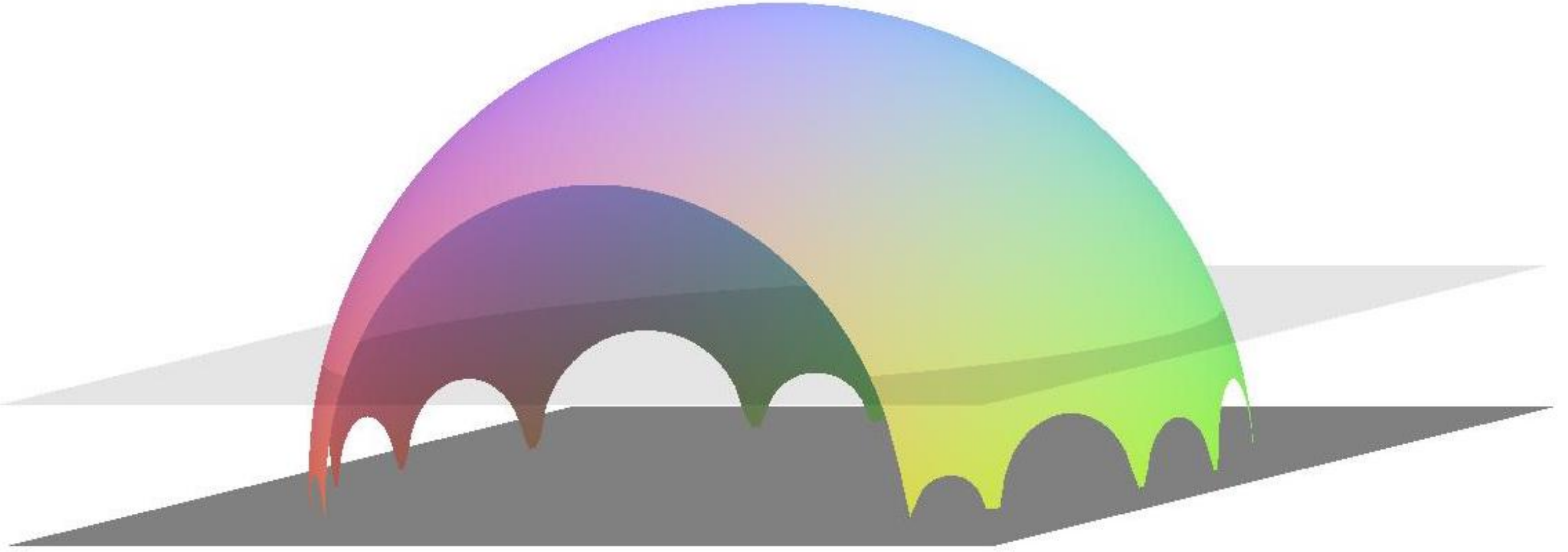}
\caption{A picture of a perfect umbrella, visualized in $\H^3$ with the vertical direction corresponding to a cusp, and one horoball shown.}
\label{F:Umbrella}
\end{figure}

\bold{The construction of the umbrella.} 
We first tried to construct the umbrellas out of good pants. The idea is to glue together many good pants, gradually and continually tilting to move away from the cusp. A major difficulty is that a cuff might enter the cusp several times, and we might require different, incompatible tilting directions for each. Furthermore, as the umbrella is constructed, new cusp excursions might be introduced. We solve these problems simultaneously by using a component we call a good hamster wheel. The advantage of the hamster wheel is that each boundary component created by adding a hamster has at most one cusp excursion, and no unnecessary cusp excursions are created. 

Each hamster wheel can be viewed as a map from a many times punctured sphere into the 3-manifold. As is the case for good pants, the complex length of each boundary circle will be required to be within a small constant $\epsilon$ of a large constant $2R$, however the boundary circles do not all play the same role. There are two outer boundaries, and $R$ inner boundaries; see Figure \ref{F:HamsterWheel}. We use only hamster wheels one of whose outer boundaries is in the compact part of the three manifold, and we glue the other outer boundary to one of the boundary components of umbrella that is under construction.  

\begin{figure}[ht!]
\makebox[\textwidth][c]{
\includegraphics[width=1.02\linewidth]{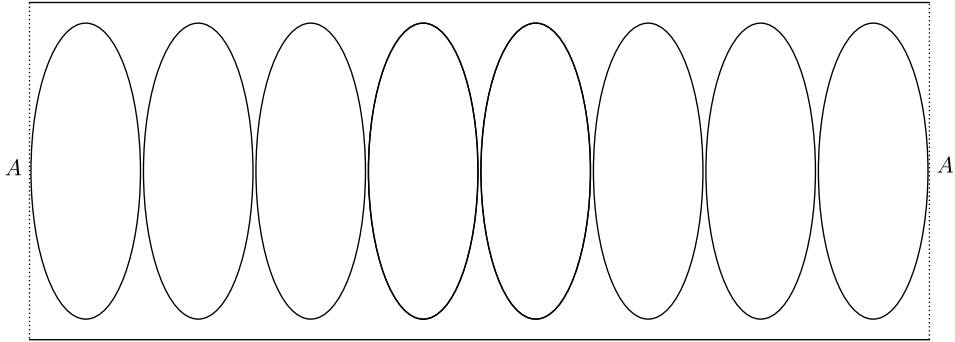}}
\caption{Identifying the left and right edges gives a schematic picture of a hamster wheel. The top boundary circle, the bottom boundary circle, and each of the $R$ many inner boundary circles each have length close to $2R\gg1$. Each inner boundary comes extremely close to each of its neighbors, but stays definite distance away from the two outer boundary circles. }
\label{F:HamsterWheel}
\end{figure}

\bold{More details and the key estimate.} We begin the proof of Theorem \ref{main theorem} by taking one pair of every pair of $(R, \epsilon)$-good pants of height at most some cutoff $h_C$, where $h_C$ is \color{kwred}a large \color{black} multiple of $\log R$. We observe that if we allow ourselves to use all good pants, rather than only those of height at most $h_C$, then the pants are sufficiently well equidistributed around each cuff of height below $h_C$ to allow us to match them while requiring that the bending between matched pants is very small. Then, for each pants we use above height $h_C$ (with \color{kwred}a \color{black} boundary \color{kwred}component \color{black} below $h_C$), we build an umbrella to replace it. The boundary of the umbrella is below a target height $h_T$,  a \color{kwred}much smaller \color{black} multiple of $\log R$.  

We have only somewhat crude control on the distribution of the boundaries of the umbrellas, but we wish to match the remaining pants and boundaries of umbrellas. So we need that the total size of the boundary of all umbrellas is relatively small, so that the umbrellas don't ruin the equidistribution of good components about each cuff. 

If we keep the target height $h_T$ constant, there are two effects to increasing the cutoff height $h_C$. First, this increases the size of the umbrellas, because the umbrellas must span the  distance $h_T-h_C$ down to  the target height. Second,  fewer cuffs of pants will come close to the cutoff height $h_C$, so fewer umbrellas are required. The key quantitative comparison that makes our construction viable is that the second effect is more significant that the first. 

\color{kwred}Formally speaking, we first fix the constant $h_C/\log R$, then pick the large constant $h_T/\log R$, and then pick $R$ to be very large depending on these other constants.\color{black}

\bold{Extending existing technology.} Our proof makes heavy use of the technology developed in \cite{KM:Immersing}, in a context more general than it was originally developed. Namely, we use the ``local to global" result giving that a good assembly is nearly Fuchsian, but unlike in \cite{KM:Immersing} we use both pants and hamster wheels, instead of just pants. 
In the course of generalizing the results of \cite{KM:Immersing} to the setting of hamster wheels,
we found that it would help to completely rework the corresponding section of that paper;
the result of this reworking is found in the appendix. 
This new approach the local to global theorem is simpler, much more constructive, and generalizes more readily to other semi-simple Lie groups.

\bold{Organization.}  Section \ref{S:ComponentsAndAssemblies} introduces the basics of good components and assemblies, and ends by stating the local to global theorem that is proved in the appendix. 

In Section \ref{mixing} we collect the necessary results on the construction, counting, and equidistribution of good pants, and the construction of good hamster wheels. Our forthcoming paper \cite{KW:counting} details how these are derived from mixing and the Margulis argument. Our perspective focuses on orthogeodesic connections as in \cite{KM:Ehrenpreis}, rather than tripods as in \cite{KM:Immersing}.  Some readers may wish to only skim this section on first reading. 

Section \ref{section umbrella} is in a sense the heart of the paper. It constructs and proves the necessary estimates on the umbrellas. All required facts about the umbrellas are specified in Theorem \ref{theorem umbrella}, and its ``randomized" version Theorem \ref{random umbrella}. These theorems are used in Section \ref{sec:matching}, which concludes our arguments and constructs the nearly Fuchsian surface subgroup. Of particular importance are subsection \ref{bald}, which calculates the size of the bald spot, and the final subsection \ref{endgame}, which contains the key computation that, when the cutoff height is chosen large enough, the total number of boundaries of umbrellas below the target height is insignificant. 

\bold{Acknowledgements.} 
Much of the work on this project took place during the MSRI program in spring 2015 and the IAS program in fall 2015, and we thank MSRI and IAS. 
This research was conducted during the period when AW served as a Clay Research Fellow.
The authors acknowledge support from U.S. National Science Foundation grants DMS 1107452, 1107263, 1107367 ``RNMS: GEometric structures And Representation varieties" (the GEAR Network).
This material is based upon work supported by the National Science Foundation under Grant No. DMS 1352721.
This work was supported by a grant from the Simons Foundation/SFARI (500275, JK).

We thank Darryl Cooper and David Futer for explaining their work, and Vladimir Markovic for valuable suggestions including the idea to use some sort of wheel component. 

%%%%%%%%%%%%%%%%%%%%%%%%%%%%%%%%%%%%%%%%%%
\section{Good and perfect components and assemblies}\label{S:ComponentsAndAssemblies} 
%%%%%%%%%%%%%%%%%%%%%%%%%%%%%%%%%%%%%%%%%%
%\begin{todo}
%\item
%Write some kind of proof of Lemma \ref{distance change},
%or at least explain how it is related to \cite{KM:Immersing}.
%\item
%Prove Theorem \ref{unflippable}.
%\item
%Define an assembly of good components. 
%\item
%Define the interpolation $A_t$ between the good assembly and its perfect model. 
%\item
%Show that every good assembly lies in a closed good assembly. 
%\end{todo}
%In \cite{KM:Immersing} we construct quasi-Fuchsian groups out of \emph{good pants} groups,
%which are perturbations of Fuchsian \emph{perfect} pants groups. 
%In this paper we will require an additional form of component, 
%the \emph{hamster wheel}.
%In seems best then to introduce a general concept of a \emph{good component},
%and a good assembly thereof. 
\subsection{Geodesics and normal bundles}
A \emph{geodesic} in $M$ is a unit speed constant velocity\footnote{zero acceleration}
map $\gamma\from T \to M$, 
where $T$ is a Riemannian 1-manifold. 
We say that $\gamma$ is a closed geodesic if $T$ is a closed 1-manifold (in which case it is necessarily a 1-torus). 
We say that $\gamma$ is \emph{subprime} if $\gamma_*(\pi_1(T))$ is a proper subgroup of a cyclic subgroup of $\pi_1(M)$; otherwise we say that $\gamma$ is \emph{prime}. 
If $\gamma$ is subprime it has an index $n > 1$
such that there is a prime geodesic $\hat \gamma \from \hat T \to M$ 
 and an $n$ to 1 covering map $\pi \from T \to \hat T$ 
such that $\gamma = \hat \gamma \circ \pi$. 

If $\gamma\from T\to M$ is a geodesic, we can form the unit normal bundle $N^1(\gamma)$ as
$$
N^1(\gamma) = \{ (t, v)  \mid t \in T, \,v \in T_{\gamma(t)}M,\, \norm{v} = 1, \,\< v, \gamma'(t) \> = 0 \}.
$$

For the most part, 
we will think of two geodesics $\gamma\from T \to M$ and $\eta\from U \to M$ as equivalent whenever there is an isometry $h\from T \to U$ such that $\gamma = \eta \circ h$. But we will find it helpful in the sequel---especially when we want to deal with (and then forget about) subprime geodesics---selected a representative $\gamma\from T_\gamma \to M$ for each equivalence class of geodesics. We can then think of a geodesic as an equivalence class. 

By an oriented geodesic we mean a geodesic $\gamma\from T \to M$, as defined above, along with an orientation on $T$, up to orientation preserving reparametrization. 
\subsection{Components}
Let $Q_0$ be a compact hyperbolic Riemann surface with geodesic boundary. A map $f\from Q_0 \to M$ up to free homotopy is
 canonically equivalent to a homomorphism $\rho\from \pi_1(Q_0) \to \pi_1(M)$, up to postconjugation. We define a \emph{component of type $Q_0$} to be such a map $f$ up to free homotopy such that the associated homomorphism $\rho$ is injective \color{kwred}and such that every non-trivial element of the image is homotopic to a closed geodesic\color{black}. 
 For any component $[f]$ we can find a representative $\hat f$ for which $\hat f|_{\d Q_0}$ has minimal length;
this will occur if and only if $\hat f(\d_i Q_0)$ is a geodesic, 
for each component $\d_i Q_0$ of $\d Q_0$. 
We call such a representative a \emph{nice} representative. 
From now on, 
when we introduce a representative $f$ of a component, 
we will assume that it is nice. 
We observe the following:
\begin{lemma}
Any two nice representatives for the same component are homotopic through nice representatives. %If two nice representatives agree on $\d Q_0$, then they are homotopic rel $\d Q_0$. 
\end{lemma}
%\ann{A: Jeremy, don't forget to think about whether we used the commented out second statement that is false.} 

\begin{proof}
For any homotopy $F\from [0, 1] \times X \to M$ there is an associated \emph{geodesic homotopy} $\hat{F}$ for which the path of each point $x\in M$ is a constant speed geodesic homotopic rel endpoints to the path under $F$. A geodesic homotopy between nice representatives is a homotopy through nice representatives.  
\end{proof}

Now let $f\from Q \to M$ be a nice component, 
and let $\alpha$ be a component of $\d Q$. When we discuss good components, we will parameterize them by a multiset of hyperbolic surfaces $Q_i$, so that different $f_i\from Q_i \to M$ will be considered to associated to different (but possibly isometric) elements $Q_i$ of the multiset. In particular, when we specify a  $\alpha \in \d Q$, it will be implicit that we have also specified $f:Q\to M$. 

Associated to each $\alpha$ is a geodesic $\gamma$ given by the image of $f|_\alpha$. This geodesic $\gamma$ is called a $\emph{cuff}$ of $f$. Sometimes we may also refer to $\alpha$ as a cuff, but it carries more information than the geodesic $\gamma$, namely $\alpha$ carries the information of the associated component. 

Let $g\from T_\gamma\to M$ be the selected parametrization of $f|_\alpha$. 
If $f|_\alpha$ is prime, 
then there is a unique homeomorphism $h\from \alpha \to T_\gamma$ such that $f|_\alpha = g \circ h$. 
If $f|_\alpha$ is subprime of index $n$,
then there are $n$ such homeomorphisms. 
From now on, 
for each component $f$, 
we will add a choice of homeomorphism $h$ to the information (``package'') for the component. 
As we will see, this will permit us to treat subprime geodesics on an equal footing with the prime ones. 

\subsection{Assemblies in \texorpdfstring{$M$}{M}.}
Suppose now that $\F = \{ f_i\from Q_i \to M\}$ is a multiset of components.
Suppose moreover that we have a fixed-point-free involution $\tau$ on the the multiset $\cup_i \d Q_i$,
and suppose that for each   $ \alpha\in \cup_i \d Q_i$,
both $\alpha$ and $\tau(\alpha)$ determine the same geodesic in $M$. 

Then we can identify $\alpha$ and $\tau(\alpha)$ by means of their identifications with the selected domain for their associated geodesics in $M$, and we obtain the \emph{assembly} $\F/\tau$. 
It is a surface ``made out of" the $Q_i$
(which have been joined, but not isometrically), 
along with a map of this surface into $M$. 
Letting $\A = (\F, \tau)$ be the data for the assembly,
we let $S_\A = \bigcup Q_i /\tau$ be the surface formed by joining the $Q_i$,
and $f_\A\from S_\A \to M$ be the result of joining the $f_i$ by $\tau$.
When $S_\A$ is connected, we let $\rho_\A \from \pi_1(S_\A) \to \Isom(\H^3)$ be the associated surface group representation. 
We observe that when $S_\A$ is not connected, that we partition $\A$ into subassemblies that then determine the components of $S_\A$. 

What we just described is more properly called a complete assembly or a closed assembly. 
If we are given an involution $\tau$ on a proper subset of $\cup_i \d Q_i$, 
then we obtain a surface which still has some boundary, 
which we will call an incomplete assembly. 

\subsection{Oriented components and assemblies}
Let us now assume that every perfect model $Q_i$ is oriented. 
This of course induces orientations on $\cup_i \d Q_i$ 
\new{(where each $Q_i$ lies to the left of each element of $\d Q_i$)}, 
and we now assume that $\tau$ is orientation reversing on each $\alpha$. 
Then $\cup_iQ_i/\tau$ will be oriented, and $\F/\tau$ is an oriented assembly. 

We observe that, given $\F = \{ f_i\from Q_i \to M\}$, 
there may be no such $\tau$ that respects (that is to say, reverses) orientation. 
In fact, there is such a $\tau$ if and only if $\d\F = 0$, 
where $\d\F$ assigns to each oriented geodesic $\gamma$ in $M$,
the difference between the number of $\alpha \in \cup_i \d Q_i$ for which $\F|_\alpha = \gamma$, 
and the number for which $\F|_\alpha = \gamma^{-1}$. 

\subsection{The doubling trick}  \label{components:doubling}
We now introduce the \emph{doubling trick},
which is an effective way to get a family $\F$ for which $\d\F = 0$. 
\new{It is simply an application of the observation that for any bijection $\sigma\from S \to S$ (where $S$ is any set),
we can define an involution $\tau\from S_- \coprod S_+\to S_- \coprod S_+$, where $S_-$ and $S_+$ are two copies of $S$,
by $\tau(x_-) = (\sigma(x))_+$ and $\tau(x_+) = (\sigma^{-1}(x))_-$. }

So suppose we have a family $\F = \{f_i\from Q_i \to M\}$, where the $Q_i$ are again oriented. 
We can form the family $\F^{-1} = \{f_i\from Q_i \to M\}$, where the only difference is that we have reversed the orientation of each of the $Q_i$. 
Letting $2\F$ denote $\F \coprod \F^{-1}$,
we observe that $\d(2\F) = 0$.

Now for each geodesic $\gamma$ for $M$,
let $\F(\gamma)$ denote the \new{submultiset of $\cup_i \d Q_i$ comprising those }$\alpha$
 whose associated geodesic in $M$ is $\gamma$ or $\gamma^{-1}$. 
(Of course, $\F(\gamma)$ may be empty). 
Suppose that for each $\gamma$ we have a bijection $\sigma_\gamma\from \F(\gamma) \to \F(\gamma)$. 
Now we can take $2\F(\gamma)$
 and divide it into those $\alpha$ with the same orientation as $\gamma$ 
 (call the set of them $\F_+(\gamma)$), 
and those guys with the opposite orientation ($\F_-(\gamma)$).
Then each $\alpha \in \F(\gamma)$ has an $\alpha_+$ and an $\alpha_-$,
and we define an involution $\tau$ on $2\F(\gamma)$ by $\tau(\alpha_+) = (\sigma(\alpha))_-$, 
and $\tau(\alpha_-) = (\sigma^{-1}(\alpha))_+$. 
This involution will result in a closed oriented assembly. 

In the previous two paragraphs, 
it might be more instructive to think of the original $Q_i$'s as orientable but not oriented. 
Then we take two copies of each $f_i\from Q_i \to M$, one with each orientation of $Q_i$, 
to obtain an oriented family with two components for each original component. 
We then define $\F(\alpha)$ and $2\F(\alpha) = \F_+(\alpha) \coprod \F_-(\alpha)$ as above;
the point is that none of this construction uses the original orientations of the $Q_i$. 

There is also a more subtle version of the doubling trick, which proceeds as follows:
First, suppose that we have an oriented partial assembly $A = (\F, \tau)$. 
We can form the assembly $A\inv$ by reversing the orientations of each $Q_i$ that appears in $\F$,
while leaving the $\tau$ unchanged. 
%It is important to observe that geometry of $A\inv$ (in terms of invariants, for example) can be entirely different from the geometry of $A$. We'll talk more about this after we've introduced ways of measuring the geometry. 
We observe that invariants of $A$ that depend on the orientation of $A$ will be negated in $A\inv$. 
Nonetheless $A\inv$ is a perfectly legitimate partial assembly, and the total boundary of $A\inv$ will be the same as the total boundary of $A$, except with all the orientations reversed. 

Suppose now that we have a set $\A = \{A_i\}$ of oriented partial assemblies\removed{\ (which we maybe think again as orientable partial assemblies)}. Then we can form $\A\inv = \{A_i\inv\}$, and let $2\A = \A \coprod \A\inv$. For each geodesic $\alpha$ in $M$, we can proceed as before, replacing $\F$ with $\A$ and using only the outer (unmatched) boundaries of the $A_i$. The result with be a complete assembly which will have each $A_i$ and $A_i\inv$ as a subassembly. 

\subsection{Pants and hamster wheels}
We now describe the two types of components that will appear in this paper. 
First we will fix a large integer $R$ (how large $R$ must be will be determined later in the paper). 
Now we let the perfect pants $P_R$ be the unique planar compact hyperbolic surface with geodesic boundary, such that the boundary has three components of length $2R$. 
Let us denote the three boundary components of $P_R$ by $C_i$, $i \in \Z/3\Z$.

Let us now define the $R$ perfect hamster wheel $H_R$.
It is a planar hyperbolic Riemann surface with $R+2$ geodesic boundaries, each of length $2R$.
It admits a $\Z/R\Z$ action, 
which fixes two of the boundaries setwise (we call these the outer cuffs),
and cyclically permutes the other $R$ boundaries (the inner cuffs).
% The quotient of $H_R$ by this action is $Q_R$.  
We claim that there is a unique surface
 (with geodesic boundary)
with these properties,
and we call it $H_R$. 

We can also construct $H_R$ as follows. 
Let $Q_{R}$ denote the pants with cuff lengths 2, 2 and $2R$.
If we collapse the cuff of length $2R$ to a point,
 we obtain an annulus;
the map of $Q_{R}$ to this annulus induces a map $q_R$ of $\pi_1(Q_{R})$ to $\Z$. 
We then let $H_R$ be the regular cyclic cover of degree $2R$ of $Q_R$
 induced by $q_R\invof{\left< 2R \right>}$. 
It has $R$ cuffs that map by degree 1 to the cuff of length $2R$ in $Q_R$;
 these are the inner cuffs. 
It has 2 cuffs that map by degree $R$ to one of the cuffs of length 2 in $Q_R$;
 these are the outer cuffs. 

\subsection{Orthogeodesics and feet}
Suppose that $\alpha_i\from T_i \to M$, $i = 0, 1$,
 are two closed geodesics in $M$ with distinct images. 
Suppose that $\eta\from [0, 1]\to M$ and $x_i \in T_i$ are such that $\eta(i) = \alpha_i(x_i)$, $i = 0,1$. 
Then there is a unique $\hat\eta\from [0, 1] \to M$ such that $\hat\eta$ has constant velocity,  
is homotopic to $\eta$ through maps with the $x_i$ condition described above
 (where $x_i \in T_i$ is permitted to vary), 
and is such that $\< \hat\eta'(i) , \alpha'(\hat x_i)\> = 0$ 
 ($\hat\eta$ is orthogonal to the $\alpha_i$). We call $\hat\eta$ the \emph{orthogeodesic} properly homotopic to $\eta$.
 If $\alpha_0$ and $\alpha_1$ intersect, it may be that $\eta$ is the constant function; 
otherwise we can reparameterize $\hat\eta$ to $\widehat\eta\from [0, a] \to M$ with unit speed;
we then let the \emph{feet} of $\eta$ be $\widehat\eta\:'(0)$ and $-\widehat\eta\:'(a)$;
there is one foot of $\widehat\eta$ on each of the $\alpha_i$. 

Now let $f\from Q_0 \to M$ be a nice component. 
For any orthogeodesic $\eta_0$ for $Q_0$ between $\alpha_0$ and $\alpha_1$, 
we can form the orthogeodesic $\eta$ for $f(\eta_0)$ between $f|_{\alpha_0}$ and $f|_{\alpha_1}$. 
With the identification of  $\alpha_i$ with the canonical domain for $f|_{\alpha_i}$,
we can think of the feet of $\eta$ as elements of $N^1(f|_{\alpha_i})$. 

\subsection{Good and perfect pants}  \label{components:good-pants}
\begin{comment}
Let $\Pi_0$ be an oriented topological sphere with three boundary components $C_0, C_1, C_2$. Each boundary $C_i$ inherits an orientation from $\Pi_0$. Consider a representation 
$$\rho\from\pi_1(\Pi_0)\to \isomHThree,$$
and assume the image of $\rho$ is discrete and each peripheral curve maps to a loxodromic. This data determines a homotopy class of maps $\Pi_0\to \HThree/\rho(\pi_1(\Pi_0))$, and each $C_i$ gives rise to an oriented closed geodesic $\gamma_i\in \HThree/\rho(\Gamma)$. Let $a_i$ be the oriented simple arc in $\Pi_0$ joining $C_{i-1}$ to $C_{i+1}$, where the indices in $C_i$ are understood modulo $3$; this is well-defined up to homotopy rel the boundary components. Let $\eta_i$ be the oriented distance minimizing (geodesic) arc joining $\gamma_{i-1}$ to $\gamma_{i+1}$ that is isotopic (through such arcs) to the image of $a_i$, so $\eta_i$ meets $\gamma_{i-1}$ and $\gamma_{i+1}$ at right angles. 
\end{comment}

For each $i$, 
we let $a_i$ be the unique simple orthogeodesic in $P_R$ connecting $C_{i-1}$ and $C_{i+1}$.

Now let $f\from P_R \to M$ be a pants in $M$. 
As described above, 
we can determine the three cuffs of $f$,
as well as the three \emph{short orthogeodesics} $\eta_i$ induced by the $f|_{a_i}$.

We define the half-length of $\gamma_i$, denoted $\hll(\gamma_i)$, to be the complex distance from $\eta_{i-1}$ to $\eta_{i+1}$ along $\gamma_i$, which is defined to be the complex \color{kwred}distance \color{black}  between lifts of $\eta_{i-1}$ and $\eta_{i+1}$ that differ by a lift of the positively oriented segment of $\gamma_i$ joining $\eta_{i-1}$ and $\eta_{i+1}$. It is equivalent to use the positively oriented segment of $\gamma_i$ joining $\eta_{i+1}$ to $\eta_{i-1}$, and hence two times  $\hll(\gamma_i)$ is equal to the complex translation length $\length(\gamma)$ of $\gamma_i$ \cite[Section 2.1]{KM:Immersing}.

We define an $(R, \epsilon)$-good pair of pants in $\isomHThree$ to be a $\isomHThree$ conjugacy class of representations $\rho$ as above for which 
$$|\hll(\gamma_i)-R|<\epsilon.$$
Note that, as a complex distance, $\hll(\gamma_i)$ is an element of $\mathbb{C}/(2\pi i \mathbb{Z})$. Here we mean that there exists a lift to $\mathbb{C}$ for which this inequality holds. Since $\epsilon$ will always be small, this lift is unique if it exists, so for any $(R, \epsilon)$-good pair of pants we may consider $\hll(\gamma_i)$ to be a complex number by considering this lift. 

An $R$ perfect pair of pants in $\isomHThree$ is an $(R,0)$-good pair of pants. In other words, it is the conjugacy class of a Fuchsian representation associated to a Fuchsian group uniformizing a hyperbolic pair of pants with three boundary components of length $2R$. 

We define good and perfect pants in $\Gamma$ to be $\Gamma$ conjugacy classes of representations $\rho: \Pi_0\to \Gamma$ for which the associated $\isomHThree$ conjugacy class of representations to $\isomHThree$ is good or perfect. 

Given $\rho: \Pi_0\to \isomHThree$ as above, the unit normal bundle $N^1(\gamma_i)$ to $\gamma_i$ is a torsor for the group $\C/(2\pi i \Z + \length(\gamma_i)\Z)$. (A torsor is a space with a simply transitive group action.) Each $\eta_i$ determines a point in $N^1(\gamma_{i-1})$ and a point in $N^1(\gamma_{i+1})$, which are the normal directions pointing along the unoriented geodesic arc $\eta_i$ at each endpoint, and which are called the feet of $\eta_i$. The difference of two elements of $N^1(\gamma)$ is a well defined element of $\C/(2\pi i \Z + \length(\gamma)\Z)$, and the condition 
$$|\hll(\gamma_i)-R|<\epsilon$$
above can be rephrased as saying that the difference of the two feet on $\gamma_i$ is within $\e$ of $R$. 

\subsection{Good and perfect hamster wheels} \label{goodHW}

 \label{generating hamsters}

\subsubsection*{Slow and Constant Turning Vector Fields}
{% define \u
\def\u{\mathbf u}

Suppose that $\gamma$ is a closed geodesic in a hyperbolic 3-manifold.
Recall that the unit normal bundle for $\gamma$ is a torsor 
 for $V_\gamma := \C/(2\pi i \Z + l(\gamma) \Z)$;
it fibers over $\gamma$,
and a unit normal field $\u$ for $\gamma$ is a section of this bundle. 
For any such $\u$ we have a curve on $V_\gamma$ up to translation
 and hence $\u$ has a \emph{slope} at each point of $\gamma$. 
A \emph{constant turning normal field} for $\gamma$
 is a smooth unit normal field $\u$ with constant slope. 
This constant slope has value $(\theta + 2\pi k)/b$,
where $l(\gamma) = b + i\theta$,
and $k \in \Z$. 
In the case where $|\theta + 2\pi k| < \pi$,
we say that $\u$ is a \emph{slow and constant turning normal field}. 
(Typically we will have $\theta < \epsilon$, 
 and $b$ close to $2R$, 
 so the slope will indeed be small). 
The space of slow and constant turning vector fields for $\gamma$
 (at least when $\gamma$ is $(R, \epsilon)$-good)
 is a 1-torus (a circle)
 and every slow and constant turning vector field is determined 
  by its value at one point of $\gamma$.  
If $v \in N^1(\gamma)$, and $x \in \gamma$ is the base point of $v$, 
we will use $v - \u$ as a shorthand for $v - \u(x)$, which is an element of $V_\gamma$.  
}% end define \u

The \emph{rims} of the perfect hamster wheel $H_R$ are defined to be the outer boundary geodesics. The \emph{rungs} are defined to be the minimal length orthogeodesics joining the two rims. Each hamster wheel has $2$ rims and $R$ rungs.  

Suppose now we have a hamster wheel $f\from H_R \to M$. 
We let $\gamma_l, \gamma_r$ be the images by $f$ of the rims of $H_R$, 
and we let $\lambda_i, i \in \Z/R\Z$, be the orthogeodesics for $\gamma_l$ and $\gamma_r$ corresponding (by $f$) to the rungs of $H_R$.
(Then $\gamma_l$ and $\gamma_r$ are the rims for $f$, and the $\lambda_i$ are the rungs). 

For each $i$, we let $\eta_{il}$ be the foot of $\lambda_i$ at $\gamma_l$, and likewise define $\eta_{ir}$. 
We say that $f$ is $\epsilon, R$-\emph{good} if there exist slow and constant turning vector fields $v_l$ and $v_r$ on $\gamma_l$ and $\gamma_r$ such that the following statements hold:
\begin{enumerate}
\item
For each $i$, the complex distance between $\gamma_l$ and $\gamma_r$ along $\lambda_i$ satisfies
\begin{equation}
|d_{\lambda_i}(\gamma_l, \gamma_r) - (R - 2 \log \sinh 1) | < \epsilon/R.
\end{equation}
\item
For each $i$,
we have
\begin{equation} \label{eq:dslow}
|\eta_{il} - v_l| < \epsilon /R
\end{equation}
and 
\begin{equation} 
|d_{\gamma_l}(\eta_{il}, \eta_{(i+1)l}) - 2| < \epsilon/R
\end{equation}
and likewise for $l$ replaced by $r$. 
\end{enumerate}

\begin{enew}
We will refer to the geodesics of length $O(e^{-R/2})$ connecting adjacent inner cuffs of the hamster wheel as short orthogeodesics; they have length within a constant multiple of the length of the short orthogeodesic on a perfect pants. We call the orthogeodesics connecting each inner cuff to each outer cuff (in the natural and obvious homotopy class) the \emph{medium} orthogeodesics. They have length universally bounded above and below. We observe that the short and medium orthogeodesics in the perfect hamster wheel are all disjoint and divide it into $2R$ components, each bounded by a length 2 segment of outer cuff, two length $R/2$ segments from two consecutive inner cuffs, the short orthogeodesic between them, and the two medium orthogeodesics from the two consecutive inner cuffs to the outer cuff. 
\end{enew}

For each inner cuff $\gamma$ of $f$,
we will choose two \emph{formal feet} which are unit normal vectors to $\gamma$,
as follows.
For each adjacent inner cuff, 
we consider the foot on $\gamma$ of the short orthogeodesic to the adjacent cuff;
we call the two resulting feet $\alpha_+$ and $\alpha_-$.
We will see that $|\alpha_+ - \alpha_- - R|$ is small;
it follows that we can find unique $\alpha'_+$ and $\alpha'_-$ such that $\alpha'_+ - \alpha_+$ is small and equal to $\alpha_- - \alpha'_-$,
and $\alpha'_+ - \alpha'_- = \hll(\gamma)$. 
We then let $\{ \alpha'_+,\alpha'_-\}$ be the formal feet of $\gamma$.
Because $\alpha'_+ - \alpha'_- = \hll(\gamma)$,
there is a unique slow and constant turning vector field through $\alpha_-'$ and $\alpha_+'$. 

\new{We have thus defined \emph{the} slow and constant turning vector field for each inner cuff of a hamster wheel. When we construct a hamster wheel (for example by Theorem \ref{hamsters for loners}) to be good with respect to given slow and constant turning vector fields on the outer cuffs, we can then think of those fields as \emph{the} slow and constant turning fields for that hamster wheel.}

\subsection{Good assemblies} For the purposes of this paper, 
a \emph{good component} is a good pants or a good hamster wheel. 
(It is possible to provide a much more general definition). 

For consistency of terminology, we now define the formal feet 
 and the slow and constant turning vector field
 of a good pants, so that these terms will be defined for all good components. 
The formal feet for a good pants are just the actual feet (of the short orthogeodesics),
and the slow and constant turning vector field for a good pants is the unique such field that goes through the feet.

Given two good components and a curve $\gamma$ that occurs in both of their boundaries, we say that they are well-matched if \color{kwred}either of \color{black} the following holds.
\begin{enumerate}
 \item \label{ga:formal}
Both components have formal feet on this curve. In this case, let one formal \color{kwred}foot \color{black} be $\alpha_0 \in \color{kwred}N^1(\gamma)\color{black}$, 
and let the other be $\alpha_1$. Then we require that $|\alpha_0 - \alpha_1 - (1 + i \pi)| < \epsilon/R$. 
 \item \label{ga:constant}
Otherwise
 the constant turning fields form a bend of at most $\maxVectorFieldBend$. 
\end{enumerate}
In the first case, the two component must be oriented, since the definition of complex distance uses the orientation, and we require the two components to induce opposite orientations on the curve. In the second case, no orientations are required, but if the components have orientations we typically also require that the two components to induce opposite orientations on the curve.

\subsection{Good is close to Fuchsian}
In the appendix we will prove the following theorem:
\begin{theorem} \label{good close to Fuchsian}
\new{There exists $R_0 > 0$  
such that for all $K > 1$ there exists $\epsilon > 0$ for all $R > R_0$}:
Let $\A$ be an $(R, \epsilon)$-good assembly in $M$ such that $S_\A$ is connected, 
and let $\rho_\A$ be the corresponding surface group representation.
Then $\rho_\A$ is $\new{K}$-\color{kwred}quasiconformally \color{black} conjugate to a Fuchsian group. 
\end{theorem}
This theorem, which we use in Section \ref{endgame}, implies that a good assembly is homotopic to a nearly geodesic immersion of a surface. This is discussed at the end of Section \ref{subsec:conclusions}. 
\begin{remark}
By embedding the good assembly in a holomorphic disk of good assemblies and using the Schwarz Lemma (as in the discussion at the end of Section 2 of \cite{KM:Immersing}), it should be possible to prove that an $(R, \epsilon)$-good assembly is $(1+ C\epsilon)$-quasiconformally conjugate to a Fuchsian group, where $C$ is a universal constant. 
\end{remark}

%%%%%%%%%%%%%%%%%%%%%%%%%%%%%%%%%%%%%%%%%%%
\section{Mixing and counting}
\label{mixing}
%%%%%%%%%%%%%%%%%%%%%%%%%%%%%%%%%%%%%%%%%%%
%\begin{todo}
%\item
%Prove Theorem \ref{counting connections} from Theorem \ref{basic counting}.
%\item
%Prove Theorem \ref{counting pants} from Theorem \ref{counting connections}.
%\end{todo}

We begin by giving some basic results on counting geoodesic connections, which are derived from mixing of the geodesic flow using the Margulis argument. We then give applications to counting good pants and good hamster wheels. 

First we recall that there is a universal constant $\MargulisConst$, independent of $\Lattice$, such that horoballs about different cusps consisting of points with injectivity radius at most $\MargulisConst$ are disjoint; \new{see, for example, \cite[Chapter D]{BenedettiPetronio}}. We  define the \emph{height} of a point in $\H^3/\Lattice$ to be the signed distance to the boundary of these horoballs, so that the height is positive if the point lies in the union of these horoballs and negative otherwise. We  define the \emph{height} of a closed geodesic in $\H^3/\Lattice$ to be the maximum height of a point on the geodesic, and the \emph{height} of a pants to be the maximum of the heights of its cuffs.%\ann{Change so negative height becomes height 0?}

Given a geodesic segment or closed geodesic, we call the cusp excursions the maximal subintervals of the geodesic above height $0$. In the case of a geodesic segment, we say that a cusp excursion is intermediate if it does not include either endpoint of the segment. So, if a geodesic starts or ends in the cusp, it will have an initial or terminal cusp excursion, and  all other cusp excursions  are  intermediate. Hence the non-interemediate cusp excursions are either initial or terminal, and we will sometimes use the term ``non-intermediate'' to mean ``either initial or terminal''.

\subsection{Counting good curves} First we recall the count for closed geodesics, which is a standard application of the Margulis argument. Let $\goodCurves$ denote the set of $(R, \epsilon)$-good curves. Then 
\begin{equation}\label{E:curvecount}
\lim_{R\to\infty} \frac{\#(\goodCurves)}{\goodCurvesNumber}=c_\epsilon,
\end{equation}
where $c_\e$ is a non-zero constant depending on $\epsilon$ \new{\cite[Theorem 8]{ParryPollicott}; see also \cite{SarnakWakayama, MargulisMohammadiOh} and the references therein}. 

Let $\goodCurves^{\geq h}$ denote the \color{kwred}set \color{black} of $(R, \epsilon)$-good curves with height at least $h$. We have the following crude estimate.  

\begin{lemma}\label{high curve count}
\color{kwred}There exists a constant $C(\Gamma, \epsilon)>0$ such that 
$$\frac{\#(\goodCurves^{\geq h})}{\#(\goodCurves)} \leq C(\Gamma, \epsilon) R e^{-2h}$$
for all $R$ sufficiently large and $h\geq 0$. \color{black}
%For all $h\geq 0$, 
%$\#(\goodCurves^{\geq h})$ is $$O(R e^{-2h} \#(\goodCurves)),$$ where the implicit constant may depend  %on $\Gamma$ and $\epsilon$.
\end{lemma}

Note that when $h \leq \frac12\log R$, the lemma has no content. 

\begin{proof}
We will show the estimate holds when the height of $\gamma$ lies in $[h, h+1)$; the estimate for the total number of $\gamma$ with height at least $h$ then follows by summing the obvious geometric series. 

We will call the set of points in $M$ with height at least 0 the 0-horoball, and its boundary the 0-horosphere. 
Given any $\gamma$ with an excursion into the 0-horoball,
we let $\hat\gamma$ be the closed curve freely homotopic to $\gamma$
that minimizes its length outside of (this entry into) the 0-horoball,
and minimizes its length in the 0-horoball subject to the first constraint. 
Then $\hat \gamma$ will comprise two geodesic segments,
one leaving and then returning to the 0-horoball and meeting the 0-horosphere orthogonally (the outer part), 
and the other inside the 0-horoball. We let $v(\hat \gamma)$ and $w(\hat \gamma)$ be the entry and exit points for the 0-horoball, respectively.

\color{kwred} Fix a horosphere in $\H^3$ mapping to the closed horosphere containing the entry and exit points, and identify it with $\C$ in such a way that distance in the horosphere is the usual Euclidean distance in $\C$. Fix a preimage $v_0(\hat \gamma)$ of $v(\hat \gamma)$. Lift $\hat \gamma$ to a path through $v_0(\hat \gamma)$, and let $w_0(\hat \gamma)$ be the point where the lift exits the horoball, so $w_0(\hat \gamma)$ is a preimage of the point $w(\hat \gamma)$. 

There is a universal constant $C>0$, independent of everything, such that the height of $\hat \gamma$ is at most $C$ greater than the height of $\gamma$. Thus we see that the Euclidean distance in $\C$ between $v_0(\hat \gamma)$ and $w_0(\hat \gamma)$ is at most $Ce^h$ for some $C>1$. 

The lifts of $w(\hat \gamma)$ to the given horoball form a lattice in $\C$ (an orbit of a cusp group of $\Gamma)$. Up to translation, only finitely many lattices arise in this way, since there are only finitely many cusps. Thus we see that the number of possibilities for $w_0(\hat \gamma)$ given $v_0(\hat \gamma)$ is at most $C' e^{2h}$ for some $C'>1$. 

Each $\hat \gamma$ can be reconstructed from the outer part of $\gamma$ as well as the choice of $w_0(\hat \gamma)$. The former is an orthogeodesic between the closed horosphere and itself. \color{black} The \emph{a priori} counting estimate \new{(Lemma 2.2) } in \cite{KW:counting} gives an upper bound of $C'' e^{2(2R-2h)}$ for the number of orthogeodesics. Hence the number of $\gamma$ is at most  $C(\Gamma, \epsilon) e^{4R-4h} e^{2h}$ for some $C(\gamma, \epsilon)>0$.
 \end{proof}

\subsection{Counting geodesic connections in \texorpdfstring{$\HThree/\Lattice$}{H3/Gamma}}
\label{sec:counting connections}

Suppose that $\gamma_0$ and $\gamma_1$ are oriented closed geodesics in $\theManifold$. 
 (We could make a similar statement for general geodesics, 
  but we would have no cause to use it).
A \emph{connection} between $\gamma_0$ and $\gamma_1$ is a geodesic segment $\alpha$ 
 that meets $\gamma_0$ and $\gamma_1$ at right angles at its endpoints. 
In the case where $\gamma_0 = \gamma_1$ we will frequently call $\alpha$ a \emph{third connection}.
For each such $\alpha$ we let $n_i(\alpha)$ be the unit vector that points in toward $\alpha$
 at the point where $\alpha$ meets $\gamma_i$.
We also let $\transportAngle(\alpha)$ be the angle between the tangent vector to $\gamma_1$ 
 where it meets $\alpha$
 and the parallel transport along $\alpha$ 
  of the tangent vector to $\gamma_0$ where it meets $\alpha$, and define $\transportCL(\alpha)=l(\alpha)+i\transportAngle(\alpha)$.  
Thus given $\alpha$,
 we have a triple
 $$
 \invariants(\alpha) \equiv (n_0(\alpha), n_1(\alpha), \transportCL(\alpha))
 \in 
 \spaceOfInvariants(\gamma_0, \gamma_1) \equiv N^1(\gamma_0) \times N^1(\gamma_1) \times \C/2\pi i \Z.
 $$
 Viewing $\C/2\pi i \Z$ as $S^1\times \R$, 
we put the measure on $\spaceOfInvariants(\gamma_0, \gamma_1)$
 that is the product of Lebesgue measure on the first three coordinates of $\spaceOfInvariants(\gamma_0, \gamma_1)$
  (normalized to total measure 1)
 times $e^{2t} dt$ on $\R$. 
We also have a metric on $\spaceOfInvariants(\gamma_0, \gamma_1)$ that is just the $L^2$ norm of the distances in each coordinate. 
We then have the following theorem, \color{kwred}where $\neigh_\eta(A)$ denotes the set of  points with distance less than $\eta$ to the set $A$ and $\neigh_{-\eta}(A)$ denotes the set of points with distance greater than $\eta$ to the complement of $A$.\color{black}
\begin{theorem} \label{counting connections}
There exists $q>0$ depending on $\Gamma$ such that the following holds when $R^-$ is sufficiently large. Suppose  $A \subset \spaceOfInvariants(\gamma_0, \gamma_1)$, and let $R^-$ be the  infimum of the fourth coordinate of values in $A$. Assuming the heights of the associated points in $\gamma_0$ and $\gamma_1$ are bounded over $A$ by $qR^-$, and set $\eta = e^{-qR^-}$.  
Then the number of connections $\numberOfConnections(A)$ for $\alpha$ between $\gamma_0$ and $\gamma_1$
 that have $\invariants(\alpha) \in A$ satisfies
\[
(1 - \eta) |\neigh_{-\eta}(A)| \le {32\pi^2}{\numberOfConnections(A) |\HThree/\Gamma|} \le (1 + \eta) |\neigh_\eta(A)|.
\]
Furthermore if $\numberOfConnections^{\geq h}(A)$ denotes the number of connections with at least one intermediate cusp excursion of height at least $h$, and $R^+$ is the supremum of the fourth coordinates of values in $A$, then $\numberOfConnections^{\geq h}(A)$ is \new{at most}
\[
\new{C(\Gamma)} R^+ e^{2R^+-2h} 
\]
when $A$ is contained in a ball of unit diameter.
\newline\removed{The implicit constant can be taken to depend only on $\Gamma$.}  
\end{theorem} 

The first part of the Theorem appears as Theorem 3.11 in  \cite{KW:counting};
the proof of the second part is quite similar to that of Lemma \ref{high curve count}, 
and is omitted. 
%In the definition of $\numberOfConnections^{\geq h}(A)$, multiple cusp excursions are allowed, but if $\alpha$ starts or ends in the cusp then that initial or final time in the cusps does not count towards $h$. So a connection that starts at height above $h$ but then stays below height  $h$ after initially leaving the cusp does not contribute towards $\numberOfConnections^{\geq h}(A)$. 

\subsection{Counting pants} 
By considering Theorem \ref{counting connections} with  $\gamma_0 = \gamma_1 = \gamma$, we can derive a count for the number of good pants that have $\gamma$ as a boundary.  

 There is a two-to-one correspondence between \emph{oriented} third connections $\alpha$
  (from $\gamma$ to itself)
 and unoriented pants with $\gamma$ as boundary. 
 (We need to orient $\alpha$ in order to distinguish $n_0(\alpha)$
  from $n_1(\alpha)$).
  Define $\spaceOfInvariants(\gamma) = \spaceOfInvariants(\gamma, \gamma)$. 
  We first describe how to determine in terms of 
  $$ \invariants(\alpha) = (n_0(\alpha), n_1(\alpha), \transportCL(\alpha))\in \spaceOfInvariants(\gamma)$$
 whether the corresponding pants is $(R, \epsilon)$-good.
 
 Suppose that $\alpha_1, \alpha_2$ are oriented geodesic arcs in a hyperbolic three manifold that meet at right angles at both of their endpoints, forming a bigon. 
 \new{(We assume that the start point of $\alpha_2$ is the end point of $\alpha_1$, and vice versa). }
 Let $e_1$ be the complex distance between two lifts of $\alpha_2$ that are joined by a lift of $\alpha_1$\new{---\emph{where we reverse the orientation of one of the lifts of $\alpha_2$---}}%
and similarly define $e_2$, so $e_1$ and $e_2$ lie in $\C/2\pi i \Z$ and have positive real part.
\new{We reverse the orientation of one of the lifts 
so that if $\alpha_1$ and $\alpha_2$ lie on a geodesic subsurface (so the lifts lie in a hyperbolic plane),
and we make right turns within that subsurface going from $\alpha_1$ to $\alpha_2$ and back,
then $e_1$ and $e_2$ have real representatives in $\C/2\pi i \Z$.}%Unsigned complex distance lies in $(\C/2\pi i \Z)/\{\pm1\}$ \cite[Section 2.1]{KM:Immersing}, but when the two geodesics don't intersect, as is the case here, we can and will take the lift to $\C/2\pi i \Z$ with positive real part. 

 Let $\gamma'$ be the closed geodesic homotopic to the concatenation of $\alpha_1$ and $\alpha_2$ at both meeting points. Then the complex length $\length(\gamma')$ of $\gamma'$ is a function $h$ of $e_1$ and $e_2$. This function can be computed explicitly using hyperbolic geometry, giving the estimate 
$$h(e_1, e_2)=e_1+e_2-\new{2}\log(2)+O(e^{-\min(\Re(e_1), \Re(e_2))}).$$
The same estimate in the two dimensional case was used in \cite[Section 3.3]{KM:Ehrenpreis}. 

\new{Let $\gamma$ be an oriented closed geodesic. 
An oriented connection $\alpha$ from $\gamma$ to itself subdivides $\gamma$ into two  arcs $\gamma_1$ and $\gamma_2$,  
and we then orient both $\gamma_i$ so that they start and the end point of $\alpha$ and end at the start point of $\alpha$. }
Both $\gamma_i$ meet $\alpha$ at each of their endpoints at right angles. 
If we set $\alpha_1=\alpha$ and $\alpha_2=\gamma_1$, then in the notation above we have 
$$e_2\equiv n_1(\alpha)-n_0(\alpha)\in  \normalGroup \gamma;$$
\new{we let $u_1$ be the the lift of $n_1(\alpha) - n_0(\alpha)$ to $\C/(2\pi i \Z)$ with real part between $0$ and  $\Re \length(\gamma)$. }
We also have $e_1=\transportCL(\alpha)\new{-i\pi}$. 
\new{(Thus $e_1$ is real in the case of a third connection for a Fuchsian group.) }
Thus the length of one of the new cuffs of the pants with $\alpha$ as the third connection
 is $$h(e_1, e_2) = h(\transportCL(\alpha)\new{-i\pi}, \new{u_1}).$$
\new{Likewise the length of the other new cuff is $h(\transportCL(\alpha) - i\pi, u_2)$,
where $u_2$ is the appropriate lift of $n_0(\alpha) - n_1(\alpha)$.
We observe that $u_1 + u_2 = \length(\gamma)$.} 
% $$|n_1-n_0+l(\alpha)+i \transportAngle(\alpha) - R| < \e$$
  
% Similarly if we set $\alpha_1=\alpha$ and $\alpha_2=\gamma_2$, then  we have $e_1=n_0(\alpha)-n_1(\alpha)$ and $e_2= l(\alpha)-i \transportAngle(\alpha)$. 
   
We define $ \mathbb{J}_{\epsilon, R}(\gamma) \subset \spaceOfInvariants(\gamma)$ to be the set of parameters for which these computations and the $h$ function show that the pants obtained from $\alpha$ and $\gamma$  have complex lengths of all cuffs within $2\epsilon$ of $2R$: it is the subset of $(n_0, n_1, \transportCL \new{-i\pi})\in \spaceOfInvariants(\gamma)$ for which 
$$|h(\new{u_1}, \transportCL \new{-i\pi}) - 2R| < 2\epsilon \quad\text{and}\quad |h(\new{u_2}, \transportCL \new{-i\pi}) - 2R| < 2\epsilon,$$%
\new{where $u_1$ and $u_2$ are defined in terms of $n_0$ and $n_1$ as above. }
 The set of solutions in $\C/2\pi i\Z$ to 
 $$|\new{u_1 + w \new{-i\pi} -2\log(2)}-2R|<2\epsilon\quad\text{and}\quad |\new{\length(\gamma) - u_1 + w \new{-i\pi}-2\log(2)}-2R|<2\epsilon$$
has two components, \new{one in which $u_1$ has imaginary part close to $0$ and $w$ has imaginary part close to $\pi$, and the other in which the roles are reversed}. Similarly, $\mathbb{J}_{\epsilon, R}$ has two components. One corresponds to good pants, and the other corresponds to pants with a cuff with half length that is $i\pi$ off from good. 
 % Possibly not true: 
 %The two components are interchanged by the involution 
 %$$(n_0, n_1, \transportCL \new{-i\pi})\mapsto (n_0, n_1+i\pi, \transportCL \new{-i\pi}+i\pi).$$
 Define $\goodInvariants \gamma$ to be the component corresponding to good pants.

Since  the unit normal bundle $N^1(\gamma)$ is a  torsor
 for $\normalGroup \gamma$, the map $n \mapsto n + \hll(\gamma)$ is an involution of $N^1(\gamma)$. 
We denote by $\rootnormal$ the quotient of $N^1(\gamma)$ by this involution. We can think of $\rootnormal$ as the ``square root'' of the normal bundle of $\gamma$. It  is a torsor for $\normalHalfGroup \gamma$.

 We let $ \ogoodPants$ be the set of all \emph{oriented} $(R, \epsilon)$-good pants in $\lattice$, and $\ogoodPants(\gamma)$ be the set of pants in $\ogoodPants$ for which $\gamma$ is a cuff.  We let $\oogoodPants(\gamma)$ be the set of pants $P$ in $\ogoodPants(\gamma)$, along with a choice of orientation of the third connection for $P$. For any $P \in \oogoodPants(\gamma)$, we define $ \invariants(P)\in \spaceOfInvariants_{\new{\epsilon, R}}(\gamma)$ to be the invariants of the associated third connection for $\gamma$.   
   
\new{Given two normal vectors $n_0, n_1 \in N^1(\gamma)$, 
there are four possible values in $N^1(\gamma)$ for $(n_0 + n_1)/2 \equiv n_0 + (n_1 - n_0)/2$,
and therefore two possible values for the same expression in $\rootnormal$. 
When $\alpha \in \goodInvariants\gamma$,
the elements $n_1 -n_0$ and $n_0 - n_1$ of  $\normalGroup \gamma$ both have lifts to $\C$ that are close to $R$.
Then of the four possible values for $(n_0 + n_1)/2$,
we can find two values $p$ and $q$ such that $p - n_0$ and $q - n_1$ have lifts to $\C$ that are close to $R/2$.
We then have $p - q = \hll(\gamma)$,
and therefore $p$ and $q$ define the same point in $\rootnormal$;
we let $(n_0 + n_1)/2$ denote this point (in $\rootnormal$) in this case. 
Thinking of $\alpha$ as a third connection for $\gamma$ determined by a pants $P \in \oogoodPants(\gamma)$,
and letting $\invariants(P)  = (n_0, n_1, w)$, 
we observe that $p$ and $q$ are the feet of the two \emph{short} orthogeodesics in $P$ to $\gamma$. }
 %(There are two natural definitions of $(n_0 + n_1)/2$, namely $n_0+(n_1-n_0)/2$ and $n_1+(n_0-n_1)/2$, and these two definitions differ by $0$ and hence coincide.)
We may then define  $\pi\from \goodInvariants\gamma \to \rootnormal$ 
 by $$\pi(n_0, n_1, w) = (n_0 + n_1)/2.$$

%We remark that $\goodInvariants \gamma$ is invariant under the action of $\normalGroup \gamma$ on $\spaceOfInvariants(\gamma)$, and that 
% $\pi \from \goodInvariants \gamma \to \rootnormal$
% is equivariant. 
 
We let $u = \pi \following \invariants$. We observe that reversing the orientation of $\alpha$ only reverses the order of $n_0$ and $n_1$ in $\invariants(\alpha)$, without changing $\transportCL(\alpha)$, and hence $u$ is well-defined on $\ogoodPants$.

We can now state and prove the following theorem:
\newcommand{\joe}{C_{\textrm{count}}}
\begin{theorem} \label{counting pants}
There exist positive constants $q, s$ depending on $\Gamma$ such that for any $\epsilon>0$ the following holds when $R$ is sufficiently large. Let $\gamma$ be an $(R, \epsilon)$-good curve that goes at most $sR$ into the cusp. If $B \subset \rootnormal$,
then 
\begin{align*}
(1-\delta)\Vol(\nbhd_{-\delta}(B))
 &\le \frac{ \# \{ P \in \ogoodPants(\gamma) \mid u(P) \in B \}}
  {\joe(\epsilon)\epsilon^{\color{kwred}4 \color{black}}  e^{\new{4R-l(\gamma)}}/|\HThree/\Gamma|}\!\!\!
&\le (1+\delta)\Vol(\nbhd_\delta(B)),
\end{align*}
where $\delta = e^{-qR}$ and \new{$\joe(\epsilon) \to 1$ as $\epsilon \to 0$}. 
\end{theorem}

\begin{proof}
We have
\begin{equation} \label{o oo}
 \#\{  P\in \ogoodPants(\gamma) \mid u(P) \in B \} = \frac12\,  \#\{  P\in \oogoodPants(\gamma) \mid u(P) \in B \}.
\end{equation}
By definition (and abuse of notation), 
$$  \{  P\in \oogoodPants(\gamma) \mid u(P) \in B \} =  \{ \alpha \mid \invariants(\alpha) \in A \},$$
where $A=\pi^{-1}(B)\cap \goodInvariants \gamma$, and in the second set $\alpha$ is the third connection associated to $P$.
By Theorem \ref{counting connections},  
\begin{equation} \label{eta estimate}
(1 - \eta) |\neigh_{-\eta}(A)| \le {32\pi^2}{\numberOfConnections(A) |\HThree/\Gamma|} \le (1 + \eta) |\neigh_\eta(A)|
\end{equation}
\new{where $\eta = Ce^{-q_{\ref{counting connections}}R}$ for some universal $C$ and $q_{\ref{counting connections}}$ is the constant formerly known as $q$ in Theorem \ref{counting connections}. }
\new{We now need only bound $|\neigh_{\pm\eta}(A)|$ in terms of $|\neigh_{\pm\eta}(B)|$;
to first approximation, 
we simply want to compute $|A|$ in terms of $|B|$. }

\new{We define
$$
\rho \from \goodInvariants \gamma \to \rootnormal \times  \normalGroup \gamma \times \C/2\pi i \Z
$$
by
$$\rho:(n_0, n_1, \transportCL)\mapsto ((n_0+n_1)/2, n_1-n_0, \transportCL).$$
We write this map as $\rho=(\pi, \rho_0)$, where $\rho_0(n_0, n_1,\transportCL)=(n_1-n_0, \transportCL)$. }

Note also that since $\goodInvariants \gamma$ is defined only in terms of $n_1-n_0$ and $\transportCL$, we have 
$$ \rho(\goodInvariants \gamma )= \rootnormal \times \goodInvariantsZ \gamma ,$$
where $\goodInvariantsZ \gamma$ is one component of the set of pairs $$(v, \transportCL)\in \normalGroup \gamma \times \C/2\pi i \Z$$ for which 
 $$|h(v,\transportCL-i\pi)-2R|<2\epsilon\quad\text{and}\quad |h(-v,\transportCL-i\pi)-2R|<2\epsilon.$$
\new{Moreover $\rho$ is a 2 to 1 cover, whose fibers come from the two ways of ordering the unordered pair $\{n_0, n_1\}$.
We also observe that $\rho$ is locally measure-preserving between the measure given on $\spaceOfInvariants(\gamma)$ in Section \ref{sec:counting connections}, and the Euclidean measure for its codomain multiplied by $e^{2 \Re w}$. }

\begin{enew}
Since
\begin{equation*}
A = \rho^{-1}(B \times \goodInvariantsZ \gamma)
\end{equation*}
and $\rho$ is a locally measure-preserving 2 to 1 cover, 
we have 
\begin{equation}  \label{A B}
|A| =  2 |B|  |\goodInvariantsZ \gamma|.
\end{equation}
So we'd like to make a good estimate of $|\goodInvariantsZ \gamma|$ in terms of $\epsilon$ and $R$. 
 
% The two components of $\mathbb{J}_{\epsilon, R}^0$ are interchanged by the involution 
% $$(v, \transportCL)\mapsto (v+i\pi, \transportCL+i\pi).$$

Let $T_{\epsilon, R}$ be the set of $(v,\transportCL)\in \normalGroup \gamma \times \C/2\pi i \Z$ for which  
$$|v+\transportCL - i \pi -2\log(2) - 2R| < 2\epsilon \quad\text{and}\quad  | l(\gamma)-v + \transportCL -i\pi -2\log(2) - 2R| < 2\epsilon,$$%
where as usual we mean that for each inequality there exists lifts of $v,w$ to $\C$ satisfying the inequality. There are two components of $T_{\epsilon, R}$. %which are interchanged by the involution 
% $$(v, \transportCL)\mapsto (v+i\pi, \transportCL+i\pi).$$
Let $S_{\epsilon, R}$ denote the component approximating $\goodInvariantsZ \gamma$. Note that $S_{\epsilon, R}$ depends on $R$ only very weakly, since changing $R$ merely translates $S_{\epsilon, R}$. 
In particular, 
\begin{equation} \label{S C_1}
|S_{\epsilon, R}| = C_1(\epsilon) e^{4R - l(\gamma) + 4 \log2} = 16 C_1(\epsilon) e^{4R - l(\gamma)}
\end{equation}
where $C_1(\epsilon)$ is the integral of $|dv| |dw| e^{2 Re w}$ over the ``round diamond'' given by $|v +w|, |v-w| < 2\epsilon$ (and $|dv|$ and $|dw|$ are complex area integration terms, and hence each two real dimensional). 
It is not too hard to see that $C_1(\epsilon) \sim 4 \pi^2\epsilon^{4}$ as $\epsilon \to 0$. 
\end{enew}

We have 
$\rho_0(\nbhd_\delta(A)) \subset  \nbhd_{\delta+O(e^{-R})}(S_{\epsilon, R}).$ 
By assuming $R$ is sufficiently large, we get $\rho_0(\nbhd_\delta(A)) \subset \nbhd_{2\delta}(S_{\epsilon, R})$, 
and hence 
\begin{equation} \label{rho 1}
\rho(\nbhd_\delta(A)) \subset \nbhd_{\delta}(B) \times \nbhd_{2\delta}(S_{\epsilon, R}).
\end{equation}
Similarly we have 
\begin{equation} \label{rho 2}
\nbhd_{-2\delta}(B)\times \nbhd_{-2\delta}(S_{\epsilon, R})\subset \rho(\nbhd_{-\delta}(A)).
\end{equation}
\removed{Let $C(\epsilon)/2$ denote $\Vol(S_{\epsilon, R})$.} 
\newline
\removed{(The division by 2 here will cancel a factor of two coming from the 2 fold cover.) } 
\newline
Using the fact that $S_{\epsilon, R}$ has diameter bounded independently of $R$, 
assuming $\delta$ is small enough 
(which is to say $R$ is large enough) 
we  can assume that 
\begin{equation} \label{S bounds}
\Vol(\nbhd_{2\delta}(S_{\epsilon, R}))/\Vol(\nbhd_{-2\delta}(S_{\epsilon, R})) = 1+ \new{C\delta/\epsilon}.
\end{equation}

\new{%
Combining \eqref{o oo}, \eqref{eta estimate}, \eqref{A B}, \eqref{S C_1}, \eqref{rho 1}, \eqref{rho 2}, and \eqref{S bounds}, 
letting $$\joe(\epsilon) = \frac{16C_1(\epsilon)}{32\pi^2\epsilon^4},$$
letting $q$ be any positive number less than $q_{\ref{counting connections}}$,
and letting $R$ be sufficiently large (given $\epsilon$), 
we obtain the Theorem,
because all multiplicative constants will be absorbed by the large value of $\delta/\eta$.
\end{proof}
}

%These results give the upper and lower bounds.  
%The constant in front of $\delta$ can be removed by replacing $q$ with smaller value, and the extra factor $1+\new{C\delta/\epsilon}$ can also be absorbed in this way. 
%\new{The proof concludes by letting $\joe(\epsilon) = \epsilon^4 |\IsomHThree/\Gamma| / (16C_1(\epsilon))$.}

Directly combining Theorem \ref{counting pants} and the second part of Theorem \ref{counting connections}, we get the following. 

\begin{corollary} \label{counting from one endpoint}
In the same situation as in Theorem \ref{counting pants}, if $B$ is bounded then there are at least 
$$(1-\delta)\Vol(\nbhd_{-\delta}(B))e^{2R} \color{kwred}\epsilon^4/C\color{black} - \new{C(\Gamma)} Re^{2R-2h}$$%
%\ann{Replace divided by $C(\epsilon)$ with times $C \epsilon^4$?}
such good pants $\alpha$ where the third connection does not have an intermediate cusp excursion of height at least $h$.
\end{corollary}

\subsection{Hamster wheels} \label{low hamster wheels}
We will not require an exact count of the number of good hamster wheels, but we will need to know that a great many exist. Motivated by Section \ref{generating hamsters}, we begin with the following consequence of Theorem \ref{counting connections}.

\begin{lemma}\label{counting rungs}
Let $s$  be as given by Theorem \new{\ref{counting pants}}.

Suppose that $\gamma_l, \gamma_r$ are closed geodesics that go at most $sR$ into the cusp, and $p\in N(\gamma_l), q \in N(\gamma_r)$. Let $\mathbf n$ denote the number of orthogeodesic connections $\lambda$ from $\gamma_l$ to $\gamma_r$ whose feet lie within $\e/(10R)$ of $p,q$, for which 
    $$
    |\mathbf d_{\lambda}(\gamma_l, \gamma_r) - (R - 2 \log \sinh 1)| < \epsilon/(10R),
    $$
    and which do not  have an intermediate cusp excursion of height at least $h$. 
For $\e$ fixed, $R$ large enough, and $h\geq \HWUnnecessaryExcursions$, then $\mathbf{n}$ is at least a constant times $R^{-6} e^{2R}$.
\end{lemma}

\begin{proof}
The restrictions on $\lambda$ define a subset $A \subset \spaceOfInvariants(\gamma_l, \gamma_r)$ of measure a constant times $(\epsilon/10R)^6 e^{2R}$. It follows from the first part of Theorem \ref{counting connections} that the number of $\lambda$ satisfying all but the final restriction (no unnecessary cusp excursions) is bounded above and below by a constant times $(\epsilon/10R)^6 e^{2R}$. 

By the second part of Theorem \ref{counting connections}, the total number of $\lambda$ that unescessarily go at least height $h$ into the cusp is $O( R e^{2R-2h})$. When $h$ is large enough for this quantity to be $o(R^{-6} e^{2R})$ we get the result.
\end{proof}

%For any (good) curve $\alpha$, 
% the \emph{excursions} of $\alpha$ into the cusp
% are the components of $\alpha$ intersect the standard horoball of $M$. 
%Now $\beta$ be an inner boundary of a hamster wheel;
% we say that an excursion of $\beta$ is \emph{unnecessary}\ann{Standardize the terminology used to discuss unnecessary cusp excursions.} 
%  if the foot of neither orthogeodesic from $\beta$ to the two outer boundaries lies in this excursion. 

\new{\begin{corollary} \label{better rungs}
If we replace $10R$ with $3R$ where it appears in Lemma \ref{counting rungs},
we can obtain the same estimate (with a different implicit multiplicative constant),
and furthermore require that $\lambda$ has no immediate excursion into the cusp of height more than $\log R + C(\epsilon, M)$ more than the height of the corresponding endpoint of $\lambda$. 
\end{corollary}

\begin{remark}\label{R:triv}
In the remainder of the paper, we will make frequent use  the basic estimate that a geodesic that is bent away from the cusp by angle $\theta$ goes $-\log(\sin(\theta))$ higher into the cusp than its starting point, which is approximately $-\log(\theta)$ when $\theta$ is small. See Figure \ref{F:BasicBend}. 
 \begin{figure}[ht!]
\makebox[\linewidth][c]{
\includegraphics[width=.7\linewidth]{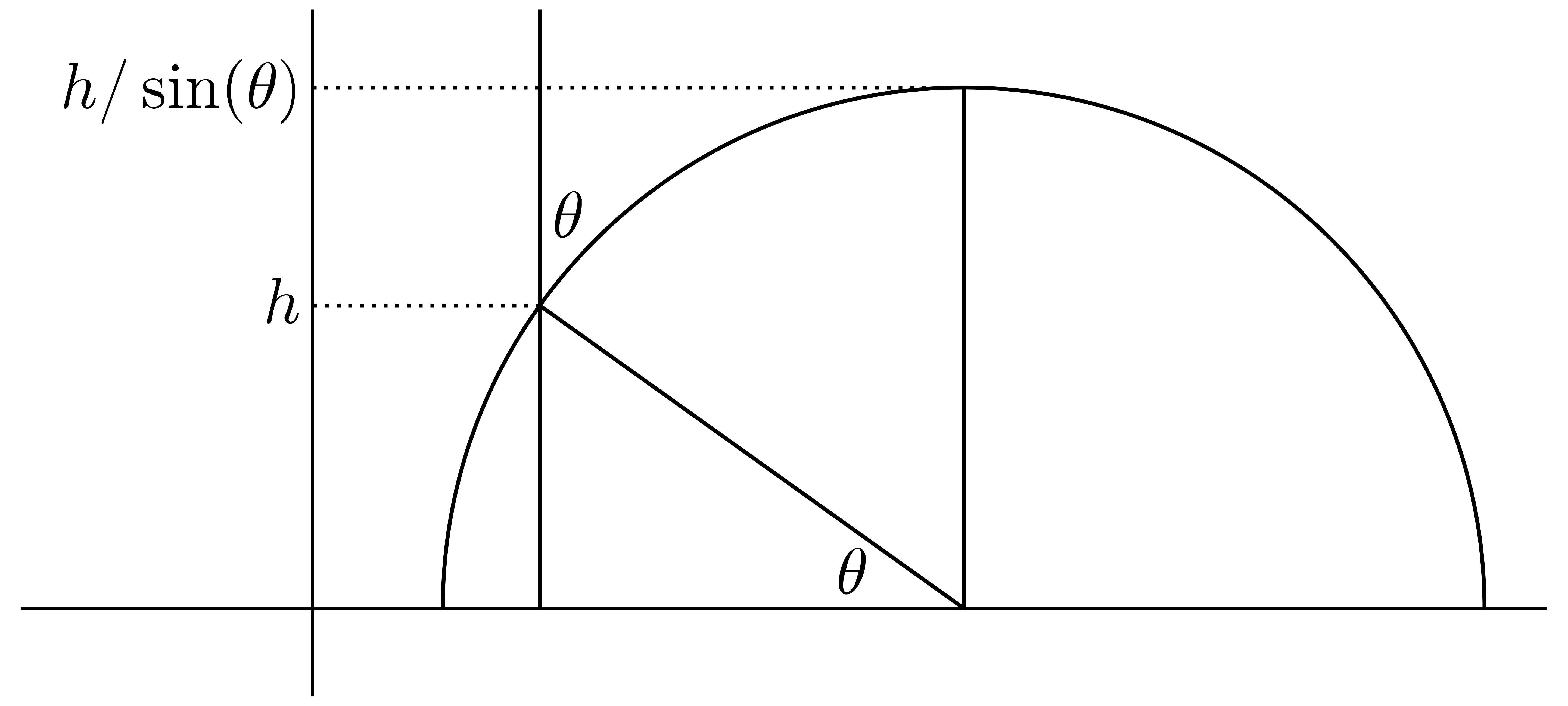}}
\caption{}
\label{F:BasicBend}
\end{figure}
\end{remark}

\begin{proof}[Proof of Corollary \ref{better rungs}]
Given $p \in N(\gamma_l)$, by Remark \ref{R:triv}
we can find $p' \in N(\gamma_l)$ such that $|p - p'| < \epsilon/3R$, 
and the ray from $p''$ goes up at most height $\log R + C(\epsilon)$ when $|p' - p''| < \epsilon/10R$. 
We can likewise find $q'$ given $q \in M(\gamma_r)$. 
Then we apply Lemma \ref{counting rungs} to $p'$ and $q'$. 
\end{proof}}

%Let $\gamma_{0}, \gamma_{1}$ be $(R, \epsilon)$-good curves in $M$ that go at most $sR$ into the cusp,
% (with $R \in \Z^{+}$), 
% and let $V_{0}, V_{1}$ be slow and constant turning unit normal fields
%  on $\gamma_{0}$ and $\gamma_{1}$.  
%%Then we can find plenty of hamster wheels $H$
%% with $\gamma_{0}$ and $\gamma_{1}$ as $\gamma$ and $\delta$
%% and $V_{0}$ and $V_{1}$ as the constant turning vector fields
%% in the construction above. 
%
%\begin{theorem} \label{making hamsters}\ann{We should delete this old stuff.}
%Given $\gamma_0, \gamma_1$ and $V_0, V_1$ as above,
% we can find a hamster wheel $\eta$
% with outer boundaries $\gamma_0$ and $\gamma_1$ 
%  and with $V_0$ and $V_1$ as the slow and constant turning vector fields at these outer boundaries,
% and for which no inner boundary  of $\eta$ has an intermediate excursion into the cusp
%  of height greater than $\HWUnnecessaryExcursions$. 
%\end{theorem}
%
%\begin{proof}
%\end{proof}

\new{%
\begin{theorem} \label{hamsters for loners} 
Given a good curve $\gamma_0$ with height less than $sR$ and a slow and constant turning vector field $V_0$ on $\gamma_0$, 
we can find a hamster wheel $\eta$ with $\gamma_0$ as one outer boundary (and $V_0$ as a good field on $\eta$),
such that 
\begin{enumerate}
\item \label{first}
the height of the other outer boundary $\gamma_1$ is at most $\log R$,
\item
no rung of $\eta$ has an intermediate excursion into the cusp of height more that $\HWUnnecessaryExcursions$
\item \label{start height}
no rung of $\eta$ has an excursion of height more than $2 \log R$ adjacent to $\gamma_1$,
\item \label{end height} \label{last}
no rung of $\eta$ has an excursion of height more than $\log R + C(\epsilon, M)$ more than its starting height adjacent to $\gamma_0$.
\end{enumerate}
\end{theorem}}

%\new{%
%\begin{proof}
%By Lemma \ref{high curve count},
%we can find $\gamma_1$ of height at most $0.75 \log R$ (for large $R$).   
%Then we apply the proof of Theorem \ref{making hamsters},
%using Corollary \ref{better rungs} to insure that the \emph{immediate} excursions of the rungs (that include an endpoint of the rung)
%are at most $\log R + C(\epsilon, M)$ higher than the associated endpoint.
%This implies \ref{start height} and \ref{end height}. 
%\end{proof}
%}

\new{
\begin{proof}%[Direct proof of Theorem \ref{hamsters for loners} without \ref{making hamsters}]
By Lemma \ref{high curve count},
we can find $\gamma_1$ of height at most $0.75 \log R$ (for large $R$).   
We arbitrarily choose a slow and constant turning vector field for $\gamma_1$.
We can then choose a bijection between evenly spaced marked points on $\gamma_0, \gamma_1$ as in Section \ref{generating hamsters}. 
For each pair of points given by the bijection, 
we apply Corollary \ref{better rungs} to the values of the slow and constant turning vector fields at those points.
In this way we obtain the rungs of $\eta$.
Moveover,
by our choice of $\gamma_1$,
and the conclusions of Lemma \ref{counting rungs} and Corollary \ref{better rungs},
we obtain properties \ref{first} through \ref{last}.
\end{proof}}
\section{The Umbrella Theorem} \label{section umbrella}
%%%%%%%%%%%%%%%%%%%%%%%%%%%%%%%%%%%%%%%%%%%
%\begin{todo}
%\item
%\end{todo}
At the center of our construction is the following theorem, which we call the Umbrella Theorem: 
\begin{theorem} \label{theorem umbrella}
For $R > R_0(\epsilon)$:
Let $P$ be an unoriented good pair of pants
 %The lower bound should be the max of $\HWUnnecessaryExcursions$ and $1.51\log(R)$. (This second quantity is the depth of the new boundary, which could be improved, an the associated necessary cusp excursions.)
and suppose $\gamma_0 \in \d P$. 
Suppose that $h_T \ge \HWUnnecessaryExcursions$. 
Then we can find a good assembly of good components $U\equiv U(P, \gamma_0)$ such that we can write $\d U = \{\gamma_0\} \cup E$, and the following hold.
\begin{enumerate}
\item
For any good component $Q$ with $\gamma_0 \in \d Q$,
if $Q$ and $P$ are $\epsilon$-well-matched at $\gamma_0$ (for some orientations on $Q$ and $P$)
then $Q$ and $U$ are $2\epsilon$-well-matched at $\gamma_0$.
\item
$h(\alpha) < h_T$ for all $\alpha \in E$. 
\item
% \ann{We let \underline{x} denote max(x, 0)}
$\size E < \umbrellaBoundarySize{\gamma_0}$, where $K = \new{K(\epsilon)}$ is given by Theorem \ref{good close to Fuchsian}.
\end{enumerate}
\end{theorem}
Regarding the first point, we remark that the notion of being well-matched with $U$ will not depend on orientations. 

Outside the proof of this theorem we will write $E$ as $\dout (P, \gamma_0)$, so we always have $\d U(P, \gamma_0) = \{\gamma_0\} \cup \dout (P, \gamma_0)$, and we will refer to $E$ as the \emph{external boundary} of $U(P, \gamma_0)$. 

We will apply the Umbrella Theorem when $\gamma_0$ lies below some cutoff height but $P$ goes above this cutoff height. In this case the ``umbrella" $U(P, \gamma_0)$ will serve as a replacement for $P$, as the first statement gives that it can be matched to anything that $P$ can be matched to. 

The umbrella $U$ will be a recursively defined collection of hamster wheels; at each stage of the construction, we generate new hamster wheels, one of whose outer boundaries is an inner boundary of one of the previous hamster wheels (or, at the first stage, $\gamma_0$). The construction takes place in two steps. 
\begin{enumerate}
 \item
 We first attach one hamster wheel to $\gamma_0$ to divide up $\gamma_0$'s entries into the cusp.
 \item
 We then recursively add hamster wheels,  tilting gently away from the cusp, until we reach the target height. 
\end{enumerate}
Throughout we take care to avoid unnecessary cusp excursions, and to minimize the heights of the cusp excursions that are necessary. \color{kwred} The precise meaning of ``tilting" will be made clear below; informally, it means that there is some bend between the hamster wheels that is chosen to reduce the excursion into the cusp. \color{black}

The set of hamster wheels in the umbrella has the structure of a tree; the children of a hamster wheel $H$ are all those hamster wheels constructed at the next level of the recursive procedure that share a boundary with $H$.

\subsection{The construction of the umbrella}
In this subsection we describe the construction of the umbrella $U$, which by definition is a  good assembly; 
 the termination of this construction
  and the bound on the size of $U$ 
 will be shown in Section \ref{umbrella area}. The umbrella can be viewed as an unoriented assembly, since orientations are not required to define well-matching for hamster wheels, and it can also be viewed as an oriented assembly, with two possible choices for the orientations. Here we will treat it as an unoriented assembly. 

We will let $U = U_{1} \cup U_{2}$, where $U_{i}$ is constructed in Step i.  

\bold{Step 1 (Subdivide): } 
We begin by equipping $\gamma_0$ with a slow and constant turning vector field that, at the two feet of $P$, points in the direction of the feet. \color{kwred}We let $U_{1} = \{ H_0 \}$,  where $H_0$ is a good hamster wheel given by Theorem \ref{hamsters for loners}
  with $\gamma_0$ as one of the outer boundaries, and with the given slowly turning vector field. 
    \color{black}

%We may rotate this vector field by acting by the purely imaginary subgroup of $\normalGroup {\gamma_0}$. We may rotate by an amount between $-\maxVectorFieldBend$ and $\maxVectorFieldBend$ in such a way that at every point of $\gamma_0$ in the cusp, the angle between the vector field and the cusp direction is at least $\epsilon/R$. Indeed, we may mark off $R$  evenly spaced points, and for each such point in the cusp we may exclude the interval of rotation angles that are bad at this point. Since we exclude at most $R$ intervals of size $\epsilon/R$, it follows there is some amount of rotation with the desired effect. 
%
%We let $U_{1} = \{ H_0 \}$,  where $H_0$ is a good hamster wheel given by Theorem \ref{making hamsters}
%  with $\gamma_0$ as one of the outer boundaries, and with the given slowly turning vector field, and with the other outer  boundary an arbitrarily chosen good curve of height at most $0.51 \log(R)$, which exists by Lemma \ref{high curve count}, and a similarly chosen vector field. The statement of Theorem \ref{making hamsters} gives that $\gamma_0$ does not have any unnecessary cusp excursions above $\HWUnnecessaryExcursions$. By the choice of the vector field, the necessary cusp excursions go at most $-\log(\epsilon/(2R))+O(1)= \log(R)+O(1)$ deeper into the cusp. Therefore each boundary has at most one excursion of height at least $h_T$ into the cusp, and all boundary geodesics go at most $\log(R)+O(1)$ further into the cusp than $\gamma_0$. 
  
 The first claim of Theorem \ref{theorem umbrella} follows from the construction of $H_0$ and the definition of well-matched. \color{kwred}   Each boundary of $H_0$ has at most one excursion of height at least $h_T$ into the cusp as long as we assume $h_T> 4 \log R$.\color{black}
   
\bold{Step 2 (Recursively add downward tilted hamster wheels): }
We recursively add hamster wheels to the umbrella using the following procedure. 
 For each inner boundary $\delta$  of a hamster wheel $H$ that we've already added to the umbrella, if $h(\delta)>h_T$
  we add a new hamster wheel $H_\delta$ with $\delta$ as the outer boundary as follows.
  
We will denote by $v_{\delta, H}$ the vector field on $\delta$ arising from $H$, and $-v_{\delta, H}$ the result of rotating this vector field by $\pi$. Informally, we may say that $v_{\delta, H}$ points into $H$, whereas $-v_{\delta, H}$ points in the opposite direction, away from $H$. 
  
  If $-v_{\delta, H}$ is within $\maxStepThreeUmbrellaBend$ of pointing straight down at the highest point of $\delta$, then we pick $H_\delta$ with this vector field. Otherwise, we rotate $-v_{\delta, H}$ away from the vertical upwards direction by an amount in $[\minStepThreeUmbrellaBend, \maxStepThreeUmbrellaBend]$, again measuring angle at the highest point of $\delta$, and pick $H_\delta$ with this new vector field.
  As in Step 1, we may arrange for $H_{\color{kwred}\delta\color{black}}$ to have at most one excursion of height at least $h_T$ into the cusp. 
  
  Step 2 continues until we have no inner boundaries with height at least $h_T$; we will show that Step 2 terminates with the following two results. Until these results are established, we must allow for the possibility that the umbrella has infinitely many components.

\begin{theorem} \label{umbrella height}
The height $h(U)$ of the highest geodesic in $U$ is at most the height of $\gamma_0$ plus  $\log R$
plus a constant depending on $\epsilon$. 
\end{theorem}

\begin{theorem} \label{umbrella area}
The umbrella $U$ is finite 
 (in the sense that the construction terminates)
 and the number of components in $U$ is at most $\umbrellaNumberOfComponents{\gamma_0}$,
  where $K = \new{K(\epsilon)}$ is given by Theorem \ref{good close to Fuchsian}.
\end{theorem}

\subsection{The height of the umbrella} \label{section umbrella height}

We say that a good hamster wheel $H$ is \emph{very good} if
\begin{enumerate}
\item
 one outer boundary has height less that $h_T$,
\item
 the other outer boundary only has one excursion into the cusp with height greater than $h_T$, and
\item
 no inner boundary of $H$ has an intermediate excursion into the cusp of height greater than $4 \log R$.
 \end{enumerate}
We will call the outer boundary with height greater than $h_T$ the high outer boundary. 

Let $H <  \isomHThree$ be a very good hamster wheel group,
 and $\gamma$ a lift to the upper half space $\H^3$ of the high outer boundary of $H$ that meets the standard horoball about $\infty$ in the upper half space. \new{Assume that $v_{\gamma, H}$ points at least $20\e$ away from the vertical upwards direction at the highest point of $\gamma$.} 
We define a plane $\P(H, \gamma)\subset \H^3$ as follows. 
We let $x$ be the highest point of $\gamma$ 
 and we consider the plane through $\gamma$ tangent to $v_{\gamma, H}(x)$. Cutting along $\gamma$ divides this plane into two half planes. We bend the half plane  into which $v_{\gamma, H}(x)$ points up by $\PTiltUp$ to obtain a half plane that we call $\P^+(H, \gamma)$.  We then define $\P(H, \gamma)$ to be the plane containing $\P^+(H, \gamma)$. 
 See Figure \ref{F:DefOfP}. The intuition is that $\P^+(H, \gamma)$ is some sort of approximation to the part of the umbrella that is on the $H$ side of $\gamma$, but that because of the bending this part of the umbrella should lie below $\P^+(H, \gamma)$. 
 
\begin{figure}[ht!]
\makebox[\textwidth][c]{
\includegraphics[width=\linewidth]{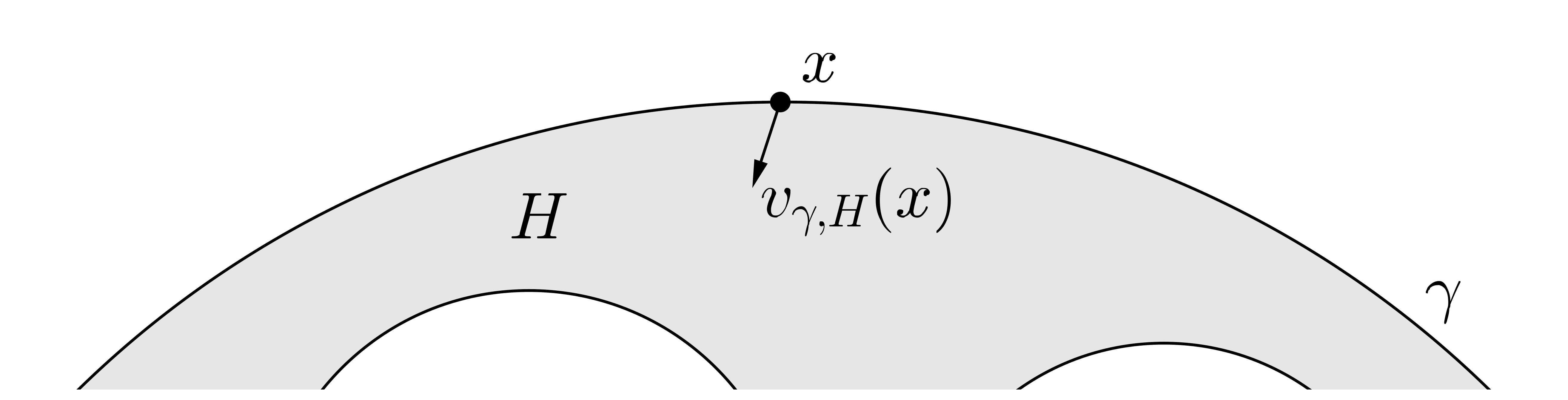}}
\caption{The definition of $\P(H, \gamma)$. In this picture $v_{\gamma, H}(x)$ seems to be pointing downwards, but the more worrisome situation is when it points upwards.}
\label{F:DefOfP}
\end{figure}

\begin{lemma}\label{under}
Let $H$ be a very good hamster wheel group,
 let $\gamma$ be the high outer boundary of $H$,
 and let $\gamma_1$ be some inner boundary of $H$ that intersects the standard horoball. 
 Suppose that $ v_{\gamma, H}(x)$ does not point within $\minStepThreeUmbrellaBend$ of straight up the cusp, and that $- v_{\gamma_1, H}(\color{kwred}y\color{black})$ does not point within $\maxStepThreeUmbrellaBend$ of straight down the cusp, \color{kwred}where $y$ is the highest point of $\gamma_1$\color{black}. 
Let $K$ be a hamster wheel attached to $\gamma_1$ with $\gamma_1$ as an outer boundary of $K$,
 and assume that $K$ bends down from $\gamma_1$ by at least $\minStepThreeUmbrellaBend$, as in Step 2 of our construction above. 
Then $\P^+(K, \gamma_1)$ lies below $\P(H, \gamma)$. 
\end{lemma}

We will derive Lemma \ref{under} from the following lemma:
\begin{lemma} \label{dihedral}
Let $\gamma, \gamma_1$ be disjoint  geodesics in $\H^3$,
  let $\eta$ be their common orthogonal, let $\dc(\gamma, \gamma_1) = d + i\theta$, 
  and suppose that $\theta \le \pi/2$. 
Let $P$ be the plane through $\gamma$ and $\eta$. 
Let $y_1$ be a point on $\gamma_1$,
 and let $P_1$ be the plane through $\gamma$ and $y_1$. 
Then the dihedral angle between $P$ and $P_1$ is at most $\frac  {|\theta|}d$.
\end{lemma}

\begin{proof}
We use the upper half \color{kwred}space \color{black} model. 
Without loss of generality, 
 assume the endpoints of $\gamma$ are 0 and $\infty$,
 and that the endpoints of $\gamma_1$ are $1/z$ and $z$, 
 with $\Re z, \Im z \ge 0$. 
The two endpoints of $\gamma$ are interchanged by $z\mapsto1/z$, as are the two endpoints of $\gamma_1$. 
Thus the endpoints of $\eta$ must be the two fixed points of $z\mapsto 1/z$, namely $-1$ and 1. See Figure \ref{F:DihedralAngle}.
 \begin{figure}[ht!]
\makebox[\textwidth][c]{
\includegraphics[width=\linewidth]{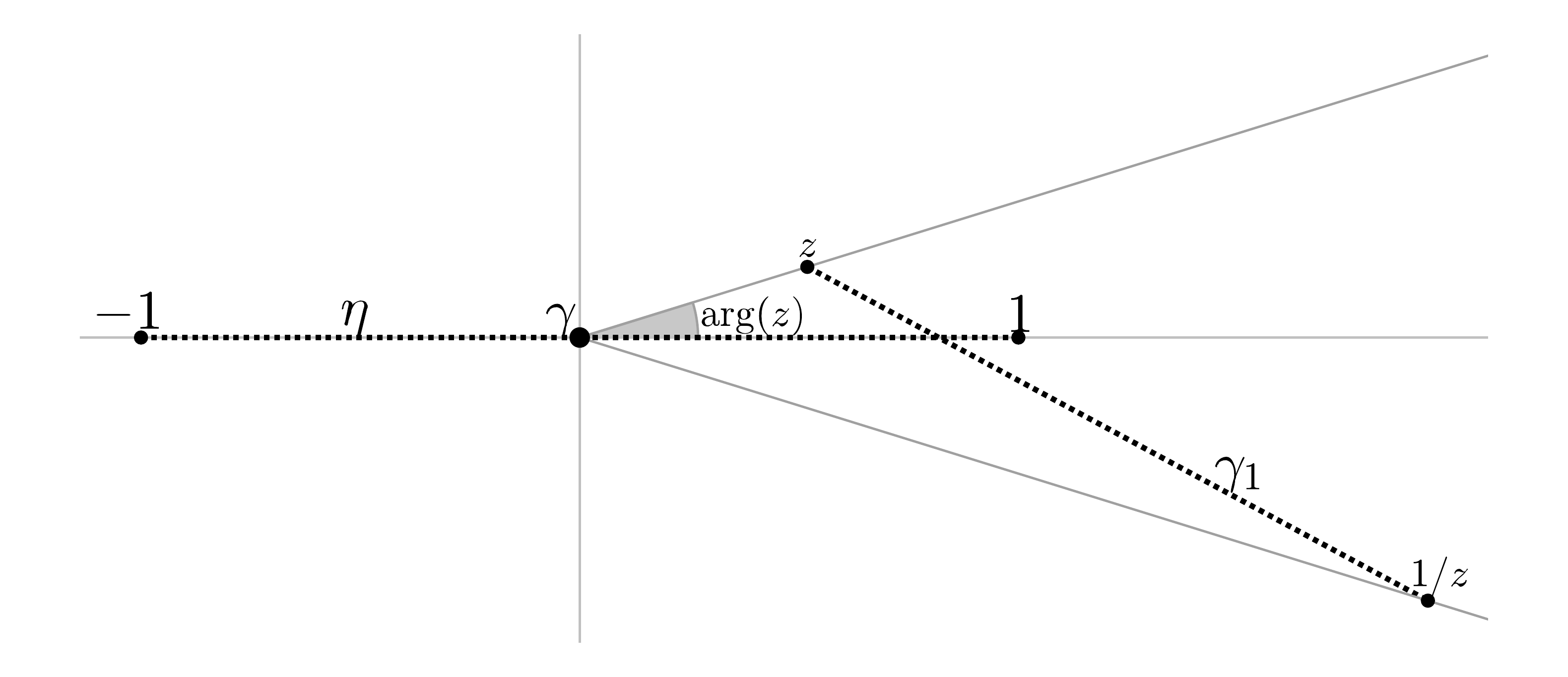}}
\caption{Proof of Lemma \ref{dihedral}. This picture in $\C$ depicts the endpoints of the geodesics $\eta$ and $\gamma_1$ in the upper half space model of $\H^3$, and, in dotted lines, their vertical projections to $\C$. The geodesic $\gamma$ projects to a single point (the origin). Any plane containing $\gamma$ will intersect $\C$ in a line through the origin.}
\label{F:DihedralAngle}
\end{figure}
The planes $P$ and $P_1$ meet the complex plane in lines through the origin.
For  $P$,
 this line goes through 1,
and for   $P_1$, 
 this line has argument between $-\arg z$ and $\arg z$. 
Therefore
 the dihedral angle between $P$ and $P_1$ is at most $\arg z$.

The complex distance between $\gamma$ and $\gamma_1$ is given by 
 \begin{equation} \label{distance from cr}
  \cosh(d+i\theta)= \frac{[0,z, 1/z, \infty]-1}{[0,z, 1/z, \infty]+1}=\frac{1+z^2}{1-z^2},
 \end{equation}
 using the cross ratio formula in \cite[p.150]{kour2}. 

By \eqref{distance from cr}, we have
\begin{equation}
 z^2 = \frac{q - 1}{q + 1}
\end{equation}
where 
\begin{align*}
 q &= \cosh(d + i\theta) \\
    &= \cosh d \cos \theta + i \sinh d \sin \theta.
\end{align*}
\new{
We observe that
\begin{equation}\label{eq:zq}
\arg z^2 = \arg \frac{q - 1}{q + 1} = \arg (q-1)(\overline q +1).
\end{equation}
Moreover,
\begin{align*}
(q-1)(\overline q + 1) &= |q|^2 + q - \overline q - 1 \\
&= \cosh^2 d \cos^2 \theta + \sinh^2 d \sin^2 \theta + 2i \sinh d \sin \theta - 1 \\
&=(1 + \sinh^2 d)(1-\sin^2 \theta) + \sinh^2 d \sin^2\theta -1 \\ & \qquad + 2i \sinh d \sin \theta  \\
&= \sinh^2 d - \sin^2 \theta + 2i \sinh d \sin \theta \\
&= (\sinh d + i \sin \theta)^2.
\end{align*}
This, \eqref{eq:zq}, and a sign check implies
\begin{equation}
\arg z = \arg (\sinh d + i \sin \theta)
\end{equation}
and hence
\begin{align*}
|\arg z| &= \left|\arctan \frac {\sin \theta}{\sinh d}\right| \\
&\le \left|\frac {\sin \theta}{\sinh d} \right|\\
&\le  \frac {|\theta|} d . \qedhere
\end{align*}}

\end{proof} 

%We observe that if $d$ is bounded below in Lemma \ref{dihedral}, then the dihedral angle is $O(\theta)$. 

\begin{proof}[Proof of Lemma \ref{under}]
Recall that $v_{\gamma, H}$ denotes the vector field along $\gamma$ determined by $H$. At the highest point of $\gamma$, by the definition of $ \P(H, \gamma)$ this vector field makes angle  $\PTiltUp$ with $ \P(H, \gamma)$. By the definition of a good curve, it follows that at every point of $\gamma$ in the given standard horoball, the angle between $v_{\gamma, H}$ and  $ \P(H, \gamma)$ is at least $\PTiltUp-\epsilon$. Hence by the definition of a good hamster wheel, it follows that for the medium orthogeodesic $\eta$ from $\gamma$ to $\gamma_1$, the angle between $\eta$ and $ \P(H, \gamma)$ must be at least $\PTiltUp-2\epsilon$. 

Our first claim is that $\P(H, \gamma) \cap \gamma_1 = \emptyset$. Otherwise, let $y_1 \in \P(H, \gamma) \cap \gamma_1$. 
By the definition of a hamster wheel, the complex distance $d+i\theta$ between $\gamma$ and $\gamma_1$ satisfies  
$$\left|\frac{d+i\theta}I -1\right| <\epsilon,$$
where $I\in \R$ is the distance in the ideal hamster wheel. In particular, $|\theta|/I < \epsilon$, so $|\theta|/d < 2\epsilon$. Hence by Lemma \ref{dihedral} the plane $P_1=\P(H, \gamma)$ through $y_1$ and $\gamma$ makes dihedral angle at most $3 \epsilon$ to the plane $P$ through $\gamma$ and the orthogeodesic $\eta$ from $\gamma$ to $\gamma_1$. This contradicts the fact from the previous paragraph that the angle between $\eta$ and $\P(H, \gamma)$ is at least $\PTiltUp-2\epsilon$. We have shown that $\gamma_1$ lies under $\P(P, \gamma)$.  

Let $L$ be the plane through $\gamma_1$ and $\eta$, and let $L^+$ be the half plane of $L$ (cut along $\gamma_1$) that does not intersect $\gamma$. Our second claim is that $\P(H, \gamma)$ does not intersect $L^+$, that is, $L^+$ lies below $\P(H, \gamma)$. Indeed, the intersection $L\cap \P(H, \gamma)$ contains the point $\gamma\cap \eta$, because $\eta\subset L$ and $\gamma\subset \P(H, \gamma)$. As a non-empty intersection of two planes, $L\cap \P(H, \gamma)$ must be a geodesic. By the previous claim, $L\cap \P(H, \gamma)$ does not intersect $\gamma_1$, hence this geodesic lies in $L\setminus L^+$, and the second claim is proved. 

Our final claim is that $\P^+(K, \gamma_1)$ lies below $L^+$.  
 To see this, \color{kwred} recall that $y$ denotes \color{black} the highest point of $\gamma_1$, and let $L_H^+$ denote the half plane through $\gamma_1$ and $-v_{\gamma_1, H}(y)$. Similarly let $L_K^+$ denote the half plane through $\gamma_1$ and $ v_{\gamma_1, K}(y)$.
 
We will now see that the final claim follows because $L^+$ is  close to $L_H^+$, and $\P^+(K, \gamma_1)$ is  close to $L_K^+$, and $L_K^+$ lies significantly below  $L_H^+$. To be more precise, by the definition of a good hamster wheel, $L_H$ makes angle at most $\epsilon$ with $\eta$.  Because of our construction of $K$, the angle between $L_K$ and $L_H$ is at least $\minStepThreeUmbrellaBend$. By definition of $\P^+(K, \gamma_1)$, it has angle $\PTiltUp$ from $L_K$. Hence $\P^+(K, \gamma_1)$ lies at angle at least $\minStepThreeUmbrellaBend-\PTiltUp-\epsilon$ below $L^+$.
 \end{proof}

\begin{proof}[Proof of Theorem \ref{umbrella height}]
Let $\delta$ be one of the boundaries of the hamster wheel $H_0$ constructed in Step 1. As we remarked in Step 1, the height of $\delta$ is at most the height of the starting geodesic $\gamma_0$ plus $\log(R)+O(1)$. Let $H_1$ denote the hamster wheel, which is constructed in the first iteration of Step 2, that has outer boundary $\delta$. By \color{kwred}Remark \ref{R:triv}\color{black}, the half plane $\P^+(H_1, \delta)$ goes at most a constant depending on $\epsilon$ farther into the cusp than $\delta$. 

Roughly speaking, we can establish the theorem by using Lemma \ref{under} repeatedly, until the assumptions are not satisfied because the umbrella is pointing too close to straight down the cusp. This last situation is in fact helpful, and we start by giving a separate argument to handle it. 

Let $H$ be any hamster wheel in the umbrella that has $H_1$ as an ancestor. Let $\gamma$ be the outer boundary of $H$ in the cusp, and $x$ be the highest point on $\gamma$. We first claim that if $\color{kwred}-\color{black}v_{\gamma, H}(x)$ points away from the cusp by more than $2\epsilon$, \color{kwred}i.e. has angle of more than $\pi/2+2\epsilon$ from vertical, \color{black}  then the same is true for the children of $H$, and moreover the children bend downward by some definite amount $\theta_0/2>0$ independent of $\epsilon$.

To prove the claim, let $\gamma_1$ be an inner boundary of $H$, and let $\eta$ be the medium orthogeodesic from $\gamma$ to $\gamma_1$. By the definition of a good hamster wheel, $\eta$ points down from the cusp by a non-zero amount where it meets $\gamma$. Hence, there is a universal constant $\theta_0$ so that $\eta$ points up into the cusp by angle at least $\theta_0$ where it \color{kwred}meets \color{black} $\gamma_1$. (Here we use the convention that at each endpoint of $\eta$, the direction of $\eta$ is given by a tangent vector pointing inwards along $\eta$.)  See Figure \ref{F:theta0}.  
\begin{figure}[ht!]
\makebox[\textwidth][c]{
\includegraphics[width=.7\linewidth]{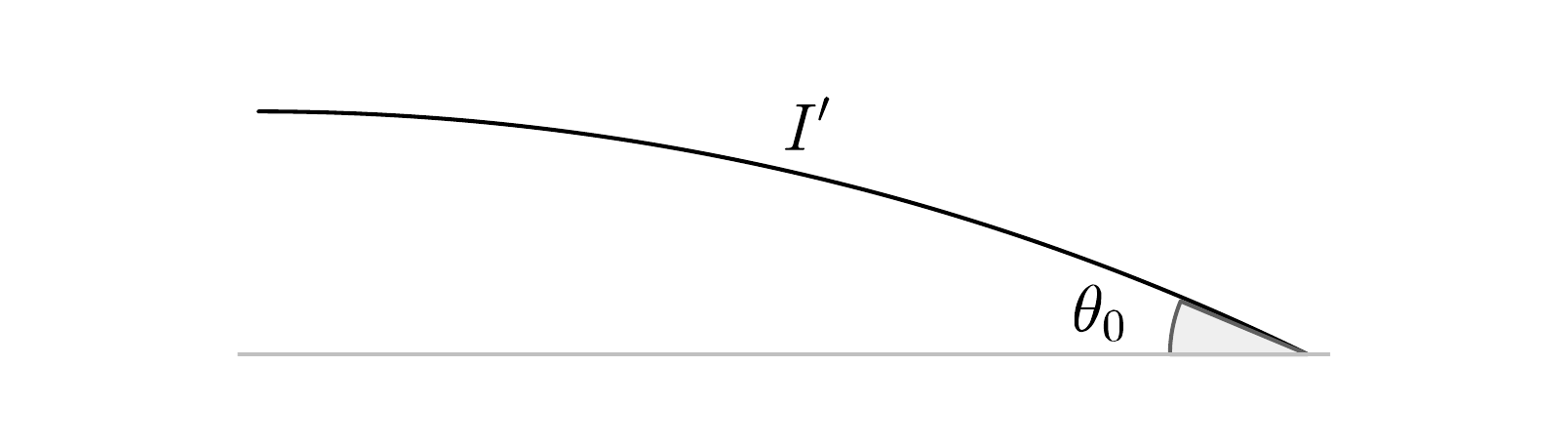}}
\caption{The definition of $\theta_0$. Here $I'$ denotes a lower bound for the length of $\eta$, which exists by comparison to a perfect hamster wheel. \color{kwred} Given a geodesic segment in the upper half plane of length $I'$ that starts tangent to the horizontal direction, $\theta_0$ is the angle to the horizontal at the other endpoint. \color{black}}
\label{F:theta0}
\end{figure}
It follows that, at the point where $\gamma_1$ meets $\eta$, the vector field $v_{\gamma_1, H}$ points up by angle at least $\theta_0-\epsilon$, and hence at the highest point of $\gamma$ it points up by at least $\theta_0-2\epsilon$. Thus if $K$ is attached to $H$ at $\gamma_1$, then $v_{\gamma_1, K}$ will point down by at least $\theta_0-2\epsilon$. This proves the claim. 

The claim gives that as soon as a hamster wheel points down by more than $\theta_0/2$, then so do all its descendants. If $H$ points down by more than $\theta_0/2$, then all the inner boundaries of $H$ have smaller height than the outer boundary. So, to prove the theorem it suffices to check the boundedness of height for hamster wheels for which $-v_{\gamma, H}(x)$ points upwards into the cusp, or points downwards into the cusp by at most $\epsilon$.  

Let $H$ be any such hamster wheel in the umbrella that has $H_1$ as an ancestor. Then by iterating  Lemma \ref{under}, we see that $K$ lies under $\P^+(H_1, \delta)$. This proves the theorem.
\end{proof}

\subsection{The area of the umbrella}
\def\ap{A_{\operatorname{perf}}}
\def\goodAlpha{\alpha}
\def\perfectAlpha{\alpha_{\operatorname{perf}}}
\def\perfectGamma{\gamma_{\operatorname{perf}}}
We proceed with the proof of Theorem \ref{umbrella area},
which will be a corollary of Theorem \ref{area bound},
which is a general statement about the number of components in the high part of a good assembly. 
Theorem \ref{area bound} will be established by applying Theorem \ref{good close to Fuchsian} and observing the same statement for perfect assemblies. 
We define the height of an assembly to \new{be } the height of the highest point on one of the geodesics in the assembly. 
\begin{lemma} \label{perfect count}
Let $\ap$ be a perfect assembly of hamster wheels. 
Then the number of geodesics $\gamma$ of $\ap$ 
 for which $h(\gamma) > h(\ap) - d$
 is $O(e^d)$. 
\end{lemma}

\begin{proof}
Let $z$ be the highest point of the \emph{plane} of $\ap$. 
There is a universal constant $\delta_0$ such \color{kwred}$z$ \color{black}  lies within $\delta_0$ of $\ap$. 
Therefore $h(z) < h(\ap) + \delta_0$. 
The area of the part of the plane  that has height above $h(z) - d$ is  $O(e^d)$. 
Every disk of unit radius in this plane can intersect at most $O(1)$ geodesics of the perfect assembly. 
\end{proof}

\begin{comment}
We can now prove our main estimate on the number of hamster wheels in the umbrella.
\begin{proof}[Proof of the main estimate]
We tilt the first hamster wheel  $H$ so that the highest point of $\P(H, \gamma)$ is at most $\log R$ plus the height of $\gamma$. 
Now, letting $\epsilon$ be small,
 we have a $K$-quasiconformal conjugacy between our assembly and a perfect assembly. 
Let's define the \emph{height differential} of a geodesic in our plane of pants to be the height of the geodesic subtracted from height of the highest point in the assembly. 
(And likewise in the perfect plane of pants). 
Because this conjugacy is $1/K$ H\"older,
 the height differential of a cuff of the perfect model is at most $K$ times the height differential of the cuff in the good assembly. 
Therefore, by Lemma \ref{perfect count},
 we have at most $R^2 e^{K\Delta h}$ cuffs, 
 where we can take $K$ as close as we want to 1. 
\end{proof}
\end{comment}

To obtain a version of Lemma \ref{perfect count} for good assemblies, 
\color{kwred}we \color{black} will need a theorem on quasiconformal mappings:
\begin{theorem} \label{qc holder}
Suppose $f\from \C \to \C$ is $K$-quasiconformal, and $f(S^1)$ has diameter at most 2. 
Then for all $z_0, z_1 \in S^1$, we have
$$
\abs{f(z_0) - f(z_1)} < 64\abs{z_0 - z_1}^{1/K}.
$$
\end{theorem}
To prove this theorem, we will require the following (from \color{kwred}\cite{ahlfors} \color{black} pp. 35--47, esp. (17)):
\begin{theorem}\label{T:mxyz}
For $x, y, z \in \C$, let $m(x, y; z)$ be the modulus of the largest annulus in $\C$ separating $x$ and $y$ from $z$. 
Then
\begin{equation}
\log \frac{\abs{z-x}}{\abs{y-x}} \le 2 \pi m \le \log 16\left(\frac{\abs{z-x}}{\abs{y-x}} + 1\right).
\end{equation}
\end{theorem}
We can now prove the theorem:
\begin{proof}[Proof of Theorem \ref{qc holder}]
If $\abs{z_1 - z_0} > 1$ then the statement is trivial. 
Otherwise we can find $y \in S^1$ such that $\abs{z_i - y} > 1$ for $i = 0, 1$. 
Letting $m$ be the largest modulus of an annulus in $\C$ separating $z_0$ and $z_1$ from $y$, \color{kwred} Theorem \ref{T:mxyz} implies that \color{black}
\begin{equation} \label{teich from ahlfors}
2 \pi m \ge \log\frac1 {\abs{z_1 - z_0}}.
\end{equation}
Letting $M$ be the largest modulus of an annulus in $\C$ separating $f(z_0)$ and $f(z_1)$ from $f(y)$, we find
\begin{align}
2 \pi M  &\le \log 16\left(\frac{ \abs{f(y) - f(z_0)} } {\abs{f(z_1) - f(z_0)}} + 1 \right) \\
		  &\le \log \frac{64}{\abs{f(z_1) - f(z_0)}}. 
\end{align}
Moreover, $M \ge m/K$ because $f$ is $K$-quasiconformal. 
The theorem follows from a simple calculation. 
\end{proof}

\new{
We will also need a weak converse to Theorem \ref{qc holder}:
\begin{lemma} \label{uniform injectivity}
For all $K, \epsilon$ there exists $\delta$:
Let $f\from \C \to \C$ be $K$-quasiconformal,
such that the diameter of $f(S^1)$ is at least 1.
Suppose that $x, y \in S^1$ satisfy $|x - y| \ge \epsilon$.
Then $|f(x) - f(y)| > \delta$.
\end{lemma}

\begin{proof}
First consider the case where the diameter of $f(S^1)$ is exactly 1,
and 0 lies in the image of the closed unit disk. 
The space of $K$-quasiconformal maps with these properties form a compact family in the local uniform topology.
Therefore,
if there were no such $\delta$,
we could find such a map $f$ and distinct points $x, y \in S^1$ such that $f(x) = f(y)$, a contradiction.

The slightly more general case of $f$ given in the Lemma follows by rescaling and translation.
\end{proof}}

We can then prove:
\def\chat{\hat\C}
\begin{lemma} \label{good diameter}
There exists a universal constant $\epsilon_0$ such that the following holds. 
Suppose that $A$ is a good assembly,
and let $\alpha$ be the highest geodesic on $A$. 
Suppose that the limit set of $A$ does not go through $\infty$.
Then the distance between the endpoints of $\alpha$ is at least $\epsilon_0$ times the diameter of the limit set of $A$.
\end{lemma}

\new{
\begin{proof}
We first observe the same statement for a \emph{perfect} assembly $\ap$,
because any point on the plane for $\ap$ (e.g. the point of greatest height) must lie at a bounded distance from a geodesic in $\ap$. 
Then,
given $A$,
we let $\ap$ be the perfect version of $A$,
and $h\from \chat \to \chat$ be the $K$-quasiconformal map given by Theorem \ref{good close to Fuchsian};
we can assume that $h(\infty) = \infty$. 
Then the desired lower bound follows from the same for $\ap$ and Lemma \ref{uniform injectivity}.
\end{proof}}

We first need the following:
\begin{lemma} \label{delta height ratio}
Suppose that $\ap$ is  perfect assembly and $A$ is a good assembly,
 $\ap$ and $A$ have finite height,
 and $\ap$ and $A$ are related by a $K$-\color{kwred}quasiconformal \color{black} map fixing $\infty$. 
%Suppose that $\alpha_0$ and $\goodAlpha_0$ are geodesics in $\ap$ and $A'$ respectively,
%and suppose that the highest points of $\alpha_0$ and $\goodAlpha_0$
% are within $d_0$ of being the highest points in the respective assemblies, 
 Suppose that \color{kwred}$\eta\in \ap$ and $\eta'\in A$ are related by this quasiconformal map. 
Then \color{black}
$$
h(A) - h(\eta') \ge K^{-1} (h(\ap) - h(\eta)) - C_0. 
$$
\end{lemma}
\begin{proof}
Let $\perfectAlpha$ and $\goodAlpha$ be the highest geodesics of $\ap$ and $A$.
We normalize the limit set of $A$ so that it has diameter 2, 
and we assume that the limit set of $\color{kwred}\ap\color{black}$ is $S^1$.
Then we have, by Theorem \ref{qc holder},
\begin{equation}
\diam{\partial \eta'} \le 64 (\diam{\partial \eta})^{1/K}
\end{equation}
and by Lemma \ref{good diameter},
\begin{equation}
\diam{\partial \goodAlpha} \ge 2\epsilon_0.
\end{equation}
Moveover $\diam{\partial \perfectAlpha} \le 2$. 
Therefore
\begin{align*}
\log \frac {\ddm{\goodAlpha}}{\ddm{\eta'}} 
&\ge \log\frac{2\epsilon_0}{64} 
+ \frac 1K \log\frac 1 {\ddm{\eta}}  \\
&\ge \log\frac{2\epsilon_0}{64} 
+ \frac 1K \log\frac  {\ddm{\perfectAlpha}}{\ddm{\eta}} 
- \frac 1K \log 2 \\
%&\ge \log\frac{2\epsilon_0}{64} 
%+ \frac 1K \log\frac 1 {\ddm{\eta}} -\log 2  \\
&\ge \frac 1K\log \frac {\ddm{\perfectAlpha}}{\ddm\eta} - \log \frac{128}{2 \epsilon_0}.
\end{align*}
Then we just observe that 
$$
h(\beta) - h(\gamma) = \log \frac {\ddm \beta}{\ddm \gamma}
$$
for all geodesics $\beta$, $\gamma$, with endpoints in $\C$. 
\end{proof}

\def\ddd{\Delta}
We can now bound the number of high geodesics in any good assembly.
\begin{theorem} \label{area bound}
Suppose $A$ is a $K$-good assembly with $K$ sufficiently close to 1. 
Then the number of geodesics $\gamma$ with $h(\gamma) > h(A) - \ddd$
 is at most $C(K)e^{K\ddd}$. 
\end{theorem}
\begin{proof}
Let $\ap$ be the perfect version of $A$,
 and let $h\from \chat \to \chat$ be a $K$-\color{kwred}quasiconformal \color{black} map given by Theorem \ref{good close to Fuchsian};
 by postcomposing $h$ with a hyperbolic isometry we can assume that $h(\infty) = \infty$
 and that $A$ has finite height.  
By Lemma \ref{delta height ratio},
every geodesic in $A$ that we want to count
maps to a geodesic $\perfectGamma$ in $\ap$ with $h(\perfectGamma) > h(\ap) - (K\ddd + KC_0)$.
By Lemma \ref{perfect count},
there are at most $C(K)e^{K\ddd}$ such geodesics in $\ap$. 
\end{proof}

Theorem \ref{umbrella area} then follows as a corollary to Theorem \ref{area bound}. The bound on the number of boundary geodesics of the umbrella in Theorem \ref{theorem umbrella} follows from the bound on the number of components in Theorem \ref{umbrella area} because each umbrella has at most $R+2$ boundary components. This concludes the proof of Theorem \ref{theorem umbrella}.  %So really the bound in the theorem should have an $R+1$ instead of an $R$.
\subsection{Semi-randomization with hamster wheels}
\new{
In our main construction, detailed in Section \ref{sec:matching}, we will match up a multi-set of umbrellas and pants. To accomplish the matching at a given closed geodesic, we will require that the number of umbrellas with the given geodesic as a boundary is negligible compared to the number of pants with the given geodesic as a boundary. To obtain such an estimate, we would need to know that the outer boundaries of umbrellas are at least somewhat reasonably divided up among all possible closed geodesics that might occur as their boundary. We are able to force this favorable situation by replacing the umbrellas with averages of umbrellas. Borrowing terminology from  \cite{KM:Ehrenpreis}, we call this semi-randomization. 
}
%\begin{enew}
%In Section \ref{sec:matching} we will bound the number of umbrellas that we need in our construction, which we will wish to combine with the estimate in Theorem \ref{theorem umbrella} on the size of an umbrella to show in Theorem \ref{main theorem} that the matching construction of Theorem \ref{basic matching} is unaffected by the perturbation introduced by the umbrellas (as described in Theorem \ref{perturbed matching}). In order to do so we need to construct our umbrellas so that not only is there a bound on the size of each one, there is an effective bound on the weight that the boundary of the umbrella gives to each good curve. 
%
%The simplest way to do this is to ``randomize'' each umbrella by adding a weighted sum of hamster wheels (with total weight one) to each external boundary curve of the umbrella. This will then allow us to think of the whole umbrella as a weighted sum of internally matched hamster wheels with a ``semirandom'' boundary. Here we borrow the terminology of \cite{KM:Ehrenpreis}.
%A finite set $A$ has a natural probability measure  
%a $K$-semirandom measure on a finite set $A$ is a (not necessarily probability) measure that gives each element of $A$ measure at most $K/|A|$. If we have another finite (multi-)set $B$ and a map $f\from B \to A$ then we say that the \ldots. 
%\end{enew}

Recall from \cite[Definition 10.1]{KM:Ehrenpreis} that a function $f\from (X, \mu) \to (Y, \nu)$ between measure spaces is called $K$-semirandom if 
\begin{equation} \label{eq:sr}
f_*\mu \le K \nu.
\end{equation}
We can generalize this to a function $f$ from $X$ to measures on $Y$:
we let
\begin{equation} \label{eq:mm}
f_*\mu = \int f(x) \, d\mu(x)
\end{equation}
and define $K$-semirandom in this setting again by \eqref{eq:sr}.
By default, given a finite (multi-)set we give it the uniform probability distribution (counting multiplicity).
When $X$ and $Y$ are finite sets, a map from $X$ to measures on $Y$ is just a map from $X$ to weighted sums of elements of $Y$, and \eqref{eq:mm} becomes
$$
(f_*\mu)(y) = \sum_x f(x)(y) \mu(x),
$$
where $f(x)(y)$ is the weight that $f(x)$ gives to $y$. 
In the case of interest, $X$ will be a multiset of hamster wheels, $Y$ will be the set of good curves, and $f$ will be the boundary map from hamster wheels to formal sums of good curves. In this case we will let ``the boundary of the random element of $X$ is a $K$-semirandom good curve'' mean ``the boundary map $f$ is $K$-semirandom".

Let $\gamma$ be a good curve with formal feet,
 or at least a slow and constant turning vector field.
%We show that we can ``randomize'' $\gamma$ by adding a weighted sum of hamster wheels. 
%That is, 
 We prove the following theorem, \color{kwred}where we refer to the constant $s$ from Lemma \ref{counting rungs}.\color{black}
\begin{theorem} \label{hamster semirandom}
Given $\gamma$ as above \color{kwred}with height at most $sR$\color{black},
there is a (multi-)set $H_\gamma$ of hamster wheels that have $\gamma$ as one outer boundary,
 and which match with the slow and constant turning vector field for $\gamma$,
 such that the average weight given by an element of $H_\gamma$
 (by the inner boundaries and the other outer boundary)
  to any good curve $\alpha$
 is at most  $C(\epsilon, M) R^{14} e^{2h(\gamma)}$ times the average weight given to $\alpha$ in the set of all good curves (which is \color{kwred}1 divided by \color{black} the number of good curves). 

In other words,
the non-$\gamma$ boundary
 of the \emph{random element} of $H_\gamma$ is a $C(\epsilon, M) R^{14} e^{2h(\gamma)}$-semirandom good curve. 
 
In addition,
the hamster wheels will have height at most \new{$\max(h(\gamma) + \log R, 4 \log R) + C(\epsilon, M)$}. 
\end{theorem}

This theorem will in turn follow from 
\begin{theorem} \label{two curve hamster semirandom}
Given $\gamma_1, \gamma_2$ good curves \color{kwred}with height at most $sR$ \color{black} with slow and constant turning vector fields,
 there is a (multi-)set $H_{\gamma_1, \gamma_2}$
  of hamster wheels that have $\gamma_1, \gamma_2$ as outer boundaries,
 and which match with the slow and constant turning vector fields for $\gamma_1$ and $\gamma_2$,
 such that the average weight given by an element of $H_\gamma$
 (by the inner boundaries)
 to any good curve $\alpha$
 is at most  $C(\epsilon, M)R^{13}e^{2(h(\gamma_1)+h(\gamma_2))}$ 
  times the average weight given to $\alpha$ in the set of all good curves.

In other words,
the total inner boundary
 of the \emph{random element} of $H_{\gamma_1, \gamma_2}$ is a $C(\epsilon, M)R^{13}e^{2(h(\gamma_1)+h(\gamma_2))}$-semirandom good curve. 

\new{In addition,
 the hamster wheels in $H_{\gamma_1, \gamma_2}$ have height at most $$\max(\max(h(\gamma_1), h(\gamma_2)) + \log R, 4 \log R) + C(\epsilon, M).$$} 
\end{theorem}

\begin{proof}
We arbitrarily mark off the $2R$ points on $\gamma_1$ and $\gamma_2$ 
 (and pair them off in a way that respects the cyclic ordering) 
 in the way described in  Section \ref{generating hamsters},
 and then we let $H_{\gamma_1, \gamma_2}$ \new{be the (multi-)set of hamster wheels we can get over all choices of orthogeodesic connections (\new{rungs}) obtained from Corollary \ref{better rungs}.  }
Each connection $U$ that we choose in this construction is officially a geodesic segment connecting $\gamma_1$ and $\gamma_2$ with endpoints near the marked points on the two outer boundaries;
we can slide any such connection along small intervals in $\gamma_1$ and $\gamma_2$ to connections that literally connect the marked points; we will think of this modified connection as the official connection $U$ in this proof. 

Let us take any good curve $\alpha$. 
Suppose that it appears as an inner boundary of some $\eta \in H_{\gamma_1, \gamma_2}$. 
Then this inner boundary is made from two ``new connections'' $U$ and $V$,
 and $U$ meets $\gamma_1$ and $\gamma_2$ at its endpoints $u_1$ and $u_2$,
 and likewise for $V$.
Let $\tau_1$ be the segment from $u_1$ to $v_1$ (of approximately unit length),
and likewise define $\tau_2$. 
There are only $R$ choices for $\tau_1$,
 and the choice of $\tau_1$ determines the choice for $\tau_2$.
We claim that given these choices,
there are $O(Re^{2(h(\gamma_1) + h(\gamma_2))})$ choices of $U$ and $V$ that result in $\alpha$ as the corresponding inner curve. 
%To be more precise,
%we mark the foot of the common perpendicular of $\tau_1$ and $\alpha$ on $\alpha$;
%when we know this up to a universal constant
% (and know $\tau_1$ and $\tau_2$),
%then we know what $U$ and $V$ must be.

Let us carefully verify this claim. 
When we are given $\tau_1$ and $\tau_2$, 
and the medium orthogeodesics $\eta_1$ and $\eta_2$ from these segments to $\alpha$,
then $U$ and $V$ are determined.
For each unit length segment $\sigma$ of $\alpha$,
there are $O(e^{2h(\tau_1)})$ possible medium orthogeodesics $\eta_1$ 
 that connect $\tau_1$ to $\sigma$. 
This is because there are $O(e^{2h(\tau_1)})$ lifts of $\sigma$ to $\H^3$ 
 that lie within a bounded-diameter radius of a given lift of $\tau_1$.
Likewise there are $O(e^{2h(\tau_1)})$ possibilities for $\eta_2$.
Summing over $2R$ unit length segments $\sigma$,
 we obtain the claim. 

Given $\tau_1$ and $\tau_2$, by Lemma \ref{counting rungs},
there are at least $C(\epsilon, M)R^{-6}e^{2R}$ possible connections $U$,
and likewise for $V$, 
for a total of $C(\epsilon, M) R^{-12} e^{4R}$ possible pairs $U, V$.
\new{So the probability,
taking a random element of $H_{\gamma_1, \gamma_2}$,
that the inner curve corresponding to $\tau_1$ and $\tau_2$ is $\alpha$ is
$$
O(Re^{2(h(\gamma_1) + h(\gamma_2))}/(R^{-12} e^{4R}))
= O(R^{13}e^{-4R}e^{2(h(\gamma_1) + h(\gamma_2))}),
$$ 
and the expected number of times that $\alpha$ appears in a random hamster wheel 
is therefore $O(R^{14}e^{2(h(\gamma_1) + h(\gamma_2)) - 4R})$ (because we sum over all $\tau_1$ and $\tau_2$). }
Since the number of good curves $\alpha$ is $O(R^{-1}e^{4R})$,
we get a $CR^{13}e^{2(h(\gamma_1) + h(\gamma_2))}$-semirandom good curve as the boundary for the random hamster wheel 
 (with outer boundary $\gamma_1$ and $\gamma_2$). 
\end{proof}

\begin{proof}[Proof of Theorem \ref{hamster semirandom}]
Given $\gamma_1$ (and its slow and constant turning vector field),
\new{we take random good $\gamma_2$ 
 of height at most $0.75 \log R$,
 and then take the random element of $H_{\gamma_1, \gamma_2}$. }
The total inner boundary of the resulting (two-step) random hamster wheel
 is $CR^{13}e^{2(h(\gamma_1) + \frac 12 \log R)} = CR^{14}e^{2h(\gamma_1)}$-semirandom by Theorem \ref{two curve hamster semirandom},
 and the new outer boundary is of course 1-semirandom
 by construction.  
\end{proof}

We can now prove a randomized version of Theorem \ref{theorem umbrella}:
\begin{theorem} \label{random umbrella}
%For $R > R_0(\epsilon)$:
%Let $P$ be an unoriented good pair of pants
%and suppose $\gamma_0 \in \d P$. 
%Suppose that $h_T \ge \HWUnnecessaryExcursions$. 
With the same hypotheses as in Theorem \ref{theorem umbrella}, but assuming $h_T>\HWUnnecessaryExcursions+2\log(R)$,
we can find a $\Q^+$ linear combination $\hat{U}\equiv \hat{U}(P, \gamma_0)$ of good assemblies of good components  
such that each good component in this linear combination satisfies properties 1 and 2 of Theorem \ref{theorem umbrella}
and we can write $\d \hat{U} = \gamma_0 + \sum n_\alpha \alpha$, 
where  
\begin{align*}
n_\alpha &< C(\epsilon, M) e^{2h_T}R^{14}\umbrellaBoundarySize{\gamma_0}/|\Gamma| \\
&= C(\epsilon, M) R^{N_{\new{\umb}}} e^{2h_T + K\max(0, h(\new{\gamma_0}) - h_T)}/|\Gamma|
\end{align*}
 for all $\alpha$. Here we define $N_{\new{\umb}}=16$ for convenience,
 \new{and we let $|\Gamma| = \epsilon^2 e^{4R}/R$}. 
%\item
%\ann{$Q$, $P$ and $U$ should be oriented in order to be matched}
%For any good component $Q$ with $\gamma_0 \in \d U$,
%if $Q$ and $P$ are $\epsilon$-well-matched at $\gamma_0$
%then $Q$ and $U$ are $2\epsilon$-well-matched at $\gamma_0$.
%\item
%$h(\alpha) < h_T$ whenever $n_\alpha > 0$, and, 
% \ann{We let \underline{x} denote max(x, 0)}
%$\size E < \umbrellaBoundarySize{\gamma_0}$, where $K = 1 + C_0\epsilon$ is given by Theorem \ref{good close to perfect}
\end{theorem}

\begin{proof}
The result will follow by combining Theorems \ref{theorem umbrella} and  \ref{hamster semirandom}. Let $U(P,\gamma_0)$ be given by Theorem \ref{theorem umbrella} using target height $h_T-2\log(R)$. 
Each $\gamma\in \dumb (P, \gamma_0)$ comes with a slow and constant turning vector field, and we may thus consider the $H_\gamma$ provided by \ref{hamster semirandom}. If we pick an element of $H_\gamma$ for each $\gamma\in \dumb (P, \gamma_0)$, these elements can be matched to $U(P,\gamma_0)$ to give a good assembly. We define $\hat{U}(P,\gamma_0)$ to be the average over all ways of picking an element of $H_\gamma$ for each $\gamma\in \dumb (P, \gamma_0)$ of the resulting good assembly. 

All the good assemblies have height %\ann{A: we mean, the outer boundary has height ...?}
at most $h_T$, since each $\gamma\in \dumb (P, \gamma_0)$ has height at most $h_T-2\log(R)$, and since elements of $H_\gamma$ have height at most $\log(R)+O(1)<2\log(R)$ more than the height of $\gamma$.  

By definition we have $\gamma_0 + \sum n_\alpha \alpha$, where $n_\alpha$ is the sum over all $\gamma\in \dumb (P, \gamma_0)$ of the average weight given to $\alpha$ by $H_\gamma$. By Theorem \ref{theorem umbrella} we have $$|\dumb (P, \gamma_0)|<\umbrellaBoundarySize{\gamma_0},$$
and by Theorem \ref{hamster semirandom} the average weight given to $\alpha$ by each $H_\gamma$ is at most $C(\epsilon, M) R^{14} e^{2h(\gamma)}/|\Gamma|$, so the desired estimate on $n_\alpha$ follows. 
\end{proof}

%%%%%%%%%%%%%%%%%%%%%%%%%%%%%%%%%%%%%%%%%%%
\section{Matching and the main theorem}
\label{sec:matching}
%%%%%%%%%%%%%%%%%%%%%%%%%%%%%%%%%%%%%%%%%%%
%\begin{todo}
%\item
%Add the proof of Theorem \ref{perturbed matching}.
%\end{todo}
\subsection{Spaces of good curves and good pants}
For any set $X$,
we let $\N X$, $\Z X$, and $\Q X$ denote the formal weighted sums of elements of $X$ (with coefficients in $\N$, $\Z$ and $\Q$). 
We can also think of these as maps from $X$ to $\N$, $\Z$, or $\Q$ with finite support, 
and we will often write $\alpha(x)$ for $\alpha \in \N X$ (or $\Z X$ or $\Q X$) and $x \in X$.
There are obvious maps $X \to \N X \to \Z X \to \Q X$,
and we will often apply these maps (or pre- or postcompose by them) without explicit comment. 
 
We recall the following definitions:
\begin{enumerate}
\item
$\goodCurves$ is the set of (unoriented) good curves.
\item
$\ogoodCurves$ is the set of \emph{oriented} good curves.
\item
$\goodChains$ is the abelian group $\Q\ogoodCurves/\left< \gamma + \gamma^{-1}\right>$ of \emph{good chains}.   
\item
$\goodPants$ is the set of (unoriented) good pants.
\item
$\ogoodPants$ is the set of \emph{oriented} good pants.
\end{enumerate}
We can think of $\goodChains$ as the abelian group of maps $\alpha\from \ogoodCurves \to \Q$ for which $\alpha(\gamma) + \alpha(\gamma\inv) = 0$ for all $\gamma \in \ogoodCurves$. 

There are obvious maps $\ogoodCurves\to\goodCurves$,  $\ogoodCurves\to\goodChains$
and $\ogoodPants \to \goodPants$.
For $\gamma \in \goodCurves$ and $\alpha \in \Q\goodPants$,
we let $\alpha(\gamma)$ be the restriction of $\alpha$ to the pants that have $\gamma$ as a boundary.%

We likewise define $\alpha(\gamma^*)$ for $\gamma^* \in \ogoodCurves$ and $\alpha \in \Q\ogoodPants$.  
In a similar vein,
we let $\goodPants(\gamma)$ denote the unoriented good pants that have $\gamma$ as a boundary,
and $\ogoodPants(\gamma^*)$ denote the \emph{oriented} good pants that have the oriented curve $\gamma^*$ as boundary. 

There are boundary maps $\d\from \ogoodPants \to \N\ogoodCurves$
 and $|\d|\from \goodPants \to \N\goodCurves$.
The former map induces a map $\d\from \ogoodPants \to \goodChains$  (and $\d\from \Z\ogoodPants\to\goodChains$) and the latter induces $|\d|\from \ogoodPants \to \N\goodCurves$
(and $|\d|\from \N\ogoodPants\to\N\goodCurves$).  
While $\d$ and $|\d|$ are both defined on $\N\ogoodPants$, 
they measure two different things:
if $\alpha \in \N\ogoodPants$ is a sum of oriented good pants,
then $|\d|\alpha(\gamma) \in \N$ is the number of pants in $\alpha$ that have $\gamma$ as an unoriented boundary,
and $\d\alpha(\gamma) \in \goodChains$ is the difference between the number of pants in $\alpha$ that have $\gamma$ as a boundary and the number that have $\gamma\inv$ as a boundary. 

There is a canonical identification of $N^1(\sqrt \gamma)$ with $N^1(\sqrt{\gamma^{-1}})$,
but they are different as torsors:
if $n_1 = n_0 + x$ in $N^1(\sqrt \gamma)$, 
 where $x \in \C/(\hll(\gamma)\Z + 2\pi i \Z)$,
then $n_1 = n_0 - x$ in $N^1(\sqrt{\gamma^{-1}})$.  
Finally $\foot_{\gamma^{-1}}\pi^{-1} = \foot_{\gamma}\pi$ 
(where we now identify $N^1(\sqrt \gamma)$ with $N^1(\sqrt{\gamma^{-1}})$, and write $\pi^{-1}$ for $\pi$ with reversed orientation).%\ann{Do we use this paragraph?} \ann{It's weird to have $\pi^{-1}$, since change the orientation doesn't matter here.}
%\subsection{The main construction and proof of the main theorem}
%We want to take all the good pants up to a certain height,
%and match them off with a twist by 1. 
%If we were to take \emph{all} pants,
%we would be able to match them off at all good curves $\gamma$ that are not too high. 

We recall from \eqref{E:curvecount} that the total size of $\goodCurves$ is on the order\footnote{We say $f(R)$ is on the order of $g(R)$ if (both are positive and) $f(R)/g(R)$ is bounded above and below.}
of
 $\goodCurvesNumber$, 
which we denote by $\numberOfGoodCurves$. 
We also observe from Theorem \ref{counting pants}
that for each good curve $\gamma$ \color{kwred} of height at most $sR$\color{black},
$\abs{\goodPants(\gamma)}$ is on the order of $R\epsilon^4 e^{2R}$.
It follows that the total size of $\goodPants$ is on the order of $\epsilon^6 e^{6R}$, 
which we denote by $\numberOfGoodPants$.
For each good curve $\gamma$, 
$\abs{\goodPants(\gamma)}$ will then be on the order of $\crunchedGoodPantsPerCurve$.

\subsection{Matching pants}\label{marriage theorems}
At each good curve $\gamma$, we wish to match each oriented good pants with $\gamma$ as a cuff with another such good pants that induces the opposite orientation on $\gamma$. We can form a bipartite graph, whose vertices are oriented good pants with $\gamma$ as a cuff, and where we join two good pants by an edge if they are well-joined by $\gamma$. The Hall Marriage Theorem says that we can find a matching in this bipartite graph, i.e. a set of edges such that each vertex is incident to exactly one edge, if and only if and only if for each subset of each of the two subsets of the vertices given by the bipartite structure, the number of outgoing edges from that subset is at least the size of the subset. We now carry out this strategy. 

For $\gamma \in \ogoodCurves$,
we let $\tau\from N^1(\sqrt \gamma) \to N^1(\sqrt \gamma)$ be defined by
 $\tau(v) = v + i\pi +1$. 
In this subsection, we will prove the following theorem, 
which is very similar to Theorem 3.1 of \cite{KM:Immersing}.
\begin{theorem} \label{basic matching}
For all $W, \epsilon$ there exists $R_0$, such that for all $R > R_0$:
Let $\gamma$ be an oriented good curve with height at most $W \log R$.
Then there exists a permutation $\sigma_\gamma\from \goodPants(\gamma) \to \goodPants(\gamma)$ such that 
\begin{equation} \label{basic:approximate}
|\foot_\gamma(\sigma_\gamma(\pi)) - \tau(\foot_\gamma(\pi))| < \epsilon/R
\end{equation}
for all $\pi \in \ogoodPants(\gamma)$. 
\end{theorem}
%We observe that if 
%$\sigma_\gamma\from \ogoodPants(\gamma) \to \ogoodPants(\gamma^{-1})$ 
%satisfies \eqref{basic:approximate} for $\gamma$, 
%then $\sigma_{\gamma\inv}:= \sigma_\gamma\inv$ satisfies \eqref{basic:approximate} for $\gamma\inv$. 
%Hence for each \emph{unoriented} good curve $\gamma$ 
%we obtain a matching between $\ogoodPants(\gamma^*)$ and $\ogoodPants((\gamma^*)\inv)$, 
%where $\gamma^*$ is an oriented version of $\gamma$.  

\def\nsg{N^1(\sqrt \gamma)}
\def\perturbationAllowance{C'(\epsilon, M) R^{-3}}
\begin{enew}
We can also make a more general statement, which amounts to a perturbation of Theorem \ref{basic matching} by a small set $A'$ which may include hamster wheels (or even whole umbrellas with $\gamma$ as an boundary outer boundary). For the purpose of this Theorem, when $\gamma$ is an outer boundary of a hamster wheel (which may be part of an umbrella) $\pi$, we let $\foot_\gamma(\pi)$ be an arbitrary unit normal vector to $\gamma$ in the slow and constant turning vector field for $\pi$.  
\end{enew}
\begin{theorem} \label{perturbed matching}
\color{kwred} For any $\e>0$, there exists a constant $C'(\e, M)$ such that for all $W$,
%There exists a $C'(\e, R)$ such that for all $W, \epsilon$ 
there exists $R_0$, such that for all $R > R_0$: \color{black}
Suppose that $\gamma \in \ogoodCurves$ \color{kwred} of height at most $W\log R$\color{black},
and $A = A' + \sum_{\pi \in \goodPants(\gamma)} \pi$
is a formal $\Q^+$ linear sum of unoriented good pants (or good components, or good assemblies), 
and $|A'|  <\perturbationAllowance \goodPantsPerCurve$. 
Then
we can find $\sigma_\gamma\from A \to A$ such that for all $\pi\in A$ we have
\begin{equation} \label{perturbed:approximate}
|\foot_\gamma(\sigma_\gamma(\pi)) - \tau(\foot_\gamma(\pi))| < \epsilon/R,
\end{equation}
where here for convenience we use $A$ to also denote the multiset that results when we clear the denominators in $A$. 
%Therefore, 
%taking one of each orientation for our pants and assemblies, 
%we can well-match everything across $\gamma$. 
\end{theorem}

%We can also make a more general statement that requires some more notation. 
%Suppose that $A_+$ and $A_-$ are elements of $\N\ogoodPants(\gamma)$ and $\N\ogoodPants(\gamma^{-1})$ respectively.
%We say that $A_+$ and $A_-$ can be \emph{well-matched}
% if there is a bijection $\sigma$ between the multisets such that 
%$$
%|(\foot_\gamma(\sigma_\gamma(\pi))) - \tau(\foot_\gamma(\pi))| < \epsilon/R.
%$$
%We say that $(A_-, A_+)$ is \emph{nearly standard} if there exists $L$ such that
%$$
%|A_+ \symDiff (L\cdot \one \ogoodPants(\gamma))| \le \perturbationAllowance L \goodPantsPerCurve,
%$$
%and 
%$$
%|A_-\symDiff (L\cdot \one \ogoodPants(\gamma\inv))| \le \perturbationAllowance L \goodPantsPerCurve,
%$$
%where $C'(\epsilon, M)$ is defined in the proof of Theorem \ref{perturbed matching}.
%We will prove the following theorem in this subsection.
%\begin{theorem} \label{perturbed matching}
%For all $W, \epsilon$ there exists $R_0$, such that for all $R > R_0$:
%Let $\gamma$ be a good curve with height at most $W \log R$.
%Suppose that $A_+$ and $A_-$ are nearly standard and have the same total cardinality. 
%Then they can be well-matched.
%\end{theorem}
%
\def\cheeger{{R}}
Recall that the \emph{Cheeger constant} $h(M)$ for a smooth Riemannian manifold $M^n$ is the infimum of $|\d A|/|A|$ over all $A \subset M$ for which $|A| \le |M|/2$. (Here $|\d A|$ denotes the $n-1$-dimensional area of $A$). 
The following theorem holds for any such manifold $M$:
\begin{theorem} \label{thm:macro-cheeger}
Suppose $A\subset M$ and $\measure {\nbhd_\eta(A)} \le 1/2 \measure{M}$,
then 
$$
\frac
{\measure{N_\eta(A)}}
{\measure{A}}
\ge
1 + \eta h(M).
$$
\end{theorem}
\begin{proof}
We have 
$$
\measure{\nbhd_\eta(A) \setminus A} 
= \int_0^\eta \measure{\d \nbhd_t(A)} dt 
\ge \eta h(M)\measure{A}.
\qedhere
$$
\end{proof}
The following is effectively proven in \cite{howards1999isoperimetric}:
\begin{theorem}
Let $T$ be a flat 2-dimensional torus, 
and let $a$ be the length of the shortest closed geodesic in $T$. 
Then $h(T) = 4a/|T|$. 
\end{theorem}
\begin{corollary}
For a good curve $\gamma$, 
we have 
\begin{equation*}
h(N^1(\sqrt\gamma)) = \frac{4\cdot 2\pi}{2\pi \Re \hll(\gamma)} = \frac{4}{\Re \hll(\gamma)} > \frac{4}{R + \epsilon} > \frac 1 R.
\end{equation*}
\end{corollary}
Therefore, by Theorem \ref{thm:macro-cheeger},
we have,
for any good curve $\gamma$:
\begin{corollary} \label{cheeger}
If $A\subset N^1(\sqrt \gamma)$ and $\measure {\nbhd_\eta(A)} \le 1/2 \measure{N^1(\sqrt \gamma)}$,
then 
$$
\frac
{\measure{N_\eta(A)}}
{\measure{A}}
>
1 + \frac \eta \cheeger.
$$
\end{corollary}

We can now make the following observation:
\begin{theorem} \label{basic marriage}
Suppose that $\eta \ge 8\cdot\cheeger\delta$,  
 and $h(\gamma) < sR$,
 where \color{kwred}$\delta = e^{-qR}$ and $s>0$ \color{black} are  from Theorem \ref{counting pants}.
Then for any translation $\tau\from \rootnormal \to \rootnormal$,
 we can find a bijection $\sigma\from \goodPants(\gamma) \to \goodPants(\gamma)$
 such that for every $\pants \in \goodPants(\gamma)$,
  \[
  \abs{
   \foot_\gamma(\sigma(\pants)) - \tau(\foot_\gamma(\pants))
   }  \le \eta.
  \]
\end{theorem}

Before proving Theorem \ref{basic marriage} we will introduce the following notation.
Suppose that $A \subset \rootnormal$. 
Then 
\begin{align*}
\#A  &:= \abs{
	\{ 
	    \alpha \text{ a third connection for $\gamma$} \suchThat u(\alpha) \in A
	\} 
                   }  \\
 &= \abs{
          \{ \pants \in \goodPants(\gamma) \suchThat \foot_\gamma(\pants) \in A \} 
        }.
\end{align*}

\begin{proof}[Proof of Theorem \ref{basic marriage}]
By the Hall Marriage Theorem, 
 this follows from the statement that $\#\nbhd_\eta(A) \ge \#\tau(A)$
 for every finite set $A$.
When $\measure{\nbhd_{\eta/2}(A)} \le \frac12 \measure{N^1(\sqrt \gamma)}$,
this follows from Theorem \ref{counting pants}, 
where $C_1 = e^{2R}/C(\epsilon)$, 
and $C(\epsilon)$ is as it appears in the statement of Theorem \ref{counting pants}:
\begin{align} \label{neighborhood translate}
\# \nbhd_\eta(A) 
&\ge C_1 (1 - \delta) \Vol(\nbhd_{\eta-\delta}(A)) \\
\nonumber&\ge C_1 (1 - \delta) \Vol(\nbhd_{\eta/2}(A)) \\
\nonumber&\ge C_1 (1 + \delta) \Vol(\nbhd_\delta(A)) \quad\quad\text{ (by \color{kwred}Corollary \color{black}  \ref{cheeger})}\\
\nonumber&= C_1 (1 + \delta) \Vol(\nbhd_\delta(\tau(A))) \\
\nonumber&\ge \# \tau(A). 
\end{align}

Otherwise, 
let $A' = \nsg - \nbhd_\eta(A)$.
Then 
$$\measure{\nbhd_{\eta/2}(A')} \le 1/2 \measure{\nsg},$$
and by the same reasoning,
$\size{\nbhd_\eta(\tau(A'))} \ge \size {\tau^{-1}(\tau(A'))}$,
and hence 
$\size{\tau(\nbhd_\eta(A'))} \ge \size {A'}$. 
Therefore 
$$
\size{\tau(\nsg - A)} \ge \sizep{\nsg - \nbhd_\eta(A)}
$$
(because $\nbhd_\eta(A') \subset \nsg - A$),
and hence $\size{\nbhd_\eta(A)} \ge \size{\tau(A)}$.
\end{proof}

\begin{proof}[Proof of Theorem \ref{basic matching}]
When $R$ is large, 
 we have $\epsilon/R > 8\cdot\cheeger e^{-qR}$
 and $W \log R < sR$ (for any previously given $W$ and $s$). 
It follows then from Theorem \ref{basic marriage} that there is a $\sigma\from\goodPants(\gamma) \to \goodPants(\gamma)$ 
 such that 
\[
  \abs{
   \foot_\gamma(\sigma(\pants)) - \tau(\foot_\gamma(\pants))
   }  \le \epsilon/R. \qedhere
\]
\end{proof}

%Now to prove Theorem \ref{perturbed matching}.
%\ann{This definition should be moved earlier}%
%\ann{Really we should be taking \emph{closed} neighborhoods (or require $U$ to be open)}%
%We say that finite multisubsets $A, B$ of a metric space $X$ are $\eta$-equivalent if $\size{A} = \size{B}$
%and $\size{A \cap U} \le \size{B \cap \nbhd_\eta(U)}$ for all $U \subset X$. 
%It follows from the Hall Marriage Theorem that $A$ and $B$ are $\eta$-equivalent if and only if there is a bijectiion $\sigma\from A \to B$ such that $d(\sigma(a), a) \le \eta$ for all $a \in A$. 
%Here is a slightly less obvious corollary to the Hall Marriage Theorem:
%\begin{lemma} \label{perturbed hall}
%Suppose that $X$ is a metric space, and $U, U', V, V' \subset X$ are finite (multi-)sets
% such that for all $A \subset X$ for which $\nbhd_\eta(A) \neq X$,
%$$
% |\nbhd_\eta(A) \cap U| \ge |A \cap V| + 2M.
%$$
%Suppose then that $|U' \symDiff U| \le M$
% and $|V' \symDiff V| \le M$,
% and $|U| = |V|$ and $|U'| = |V'|$.
%Then $U'$ and $V'$ are $\eta$-equivalent. 
%\end{lemma}
%\begin{proof}
%When $\nbhd_\eta(A) = X$,
%clearly $\sizep{\nbhd_\eta(A) \cap U'} \ge \sizep{A \cap V'}$.
%When $\nbhd_\eta(A) \neq X$,
%\begin{align*}
%\sizep{\nbhd_\eta(A) \cap U'} 
%&\ge \sizep{\nbhd_\eta(A) \cap U} - M \\
%&\ge \sizep{A \cap V} + M \\
%&\ge \sizep{A \cap V'}. \qedhere
%\end{align*}
%\end{proof}
%
%With this observation it is a simple matter to prove Theorem \ref{perturbed matching}
% along the lines of the proof of Theorem \ref{basic matching}.
\begin{proof}[Proof of Theorem \ref{perturbed matching}]
We will choose $C'(\epsilon, M)$ at the end of the proof of this Theorem. 
In light of the Hall Marriage Theorem, we need only show that 
\begin{equation} \label{have extra}
\size{\nbhd_{\epsilon/R}(A)} \ge \size{\tau(A)} + 2 \perturbationAllowance \goodPantsPerCurve.
\end{equation}
when $\nbhd_{\epsilon/R}(A) \subsetneq \squareRootUnitNormalBundle \gamma$.
When $R$ is sufficiently large, we have $\epsilon/2R \ge \eta$, 
where $\eta$ satisfies the hypothesis of Theorem \ref{basic marriage}. 
Then by equation \eqref{neighborhood translate} in the proof of that theorem, 
we have $\size{\nbhd_{\epsilon/2R}(A)} \ge \size{\tau(A)}$. 
On the other hand, 
if $\nbhd_{\epsilon/R}(A) \neq \squareRootUnitNormalBundle \gamma$,
then $\nbhd_{\epsilon/R}(A) - \nbhd_{\epsilon/2R}(A)$ contains a disk $D$ of radius $\epsilon/4R$ in $\squareRootUnitNormalBundle \gamma$. 
It follows from Theorem \ref{counting pants} that
\begin{align}
\nonumber \size D 
&\ge  C(\epsilon, M) e^{2R}/R^2.
\end{align}
%\ann{A: Is the $\epsilon/8R$ supposed to be $\epsilon/4R$?}
%&= 2C'(\epsilon, M) R^{-3} \goodPantsPerCurve, \label{just right}
%\end{align}
We observe that we can chose $C'(\epsilon, M)$ so that
$$
C(\epsilon, M) e^{2R}/R^2 \ge 2C'(\epsilon, M) R^{-3}\goodPantsPerCurve.
$$
Then we have
$$
\size D \ge 2C'(\epsilon, M) R^{-3} \goodPantsPerCurve.
$$
This implies \eqref{have extra}, and the Theorem. 
\end{proof}
\subsection{The size of the bald spot}\label{bald}
We also let $\goodPantsWithHeight h$ denote the set of good pants with height less than $h$, 
and $\goodPantsWithHeightMoreThan h$ be its complement in $\goodPants$.
We likewise define $\goodCurvesWithHeight h$. 
We also let $\ogoodPantsWithHeight h$, $\ogoodPantsWithHeightMoreThan h$, and $\ogoodCurvesWithHeight h$ be the analogous objects for oriented pants and curves. 

We observed in Section \ref{marriage theorems} that,
when $\alpha$ has reasonable height, 
the feet of the good pants $\goodPants(\alpha)$ are evenly distributed around $\alpha$.
On the other hand, 
the set of feet of pants in $\goodPantsWithHeight h(\alpha)$ has a ``bald spot'';
there are regions in $\nsg$ where there are no feet of pants in $\goodPantsWithHeight h(\alpha)$. 

Before discussing the size and shape of the bald spot on a geodesic in $M$,
it is useful to do the same calculation for a geodesic $\gamma$ in $\H^3$,
where as usual we think of $\H^3$ as the upper half space in $\R^3$.
We assume that $\gamma$ is not vertical,
so it has a unique normal vector $v_0(\gamma) \in N^1(\gamma)$ that is based on the highest point of $\gamma$ and points straight up.  
The unit normal bundle $N^1(\gamma)$ is a torsor for $\C/2\pi i\Z$,
so any $v \in N^1(\gamma)$ is uniquely determined by $v - v_0 \in \C/2\pi i \Z$.

%We suppose that we have chosen disjoint closed ``standard'' horoballs in the cusps of $M$,
%and we say that everything outside these horoballs has height 0.
%The height of a point in a standard horoball is its distance to the boundary of that horoball. 

In the lemma that follows, 
we let the height of a point $(x, y, z)$ in the upper half-space of $\R^3$ (our working model of $\H^3$) be $\log z$. 
\begin{figure} \label{phi}
\includegraphics[width=.6\linewidth]{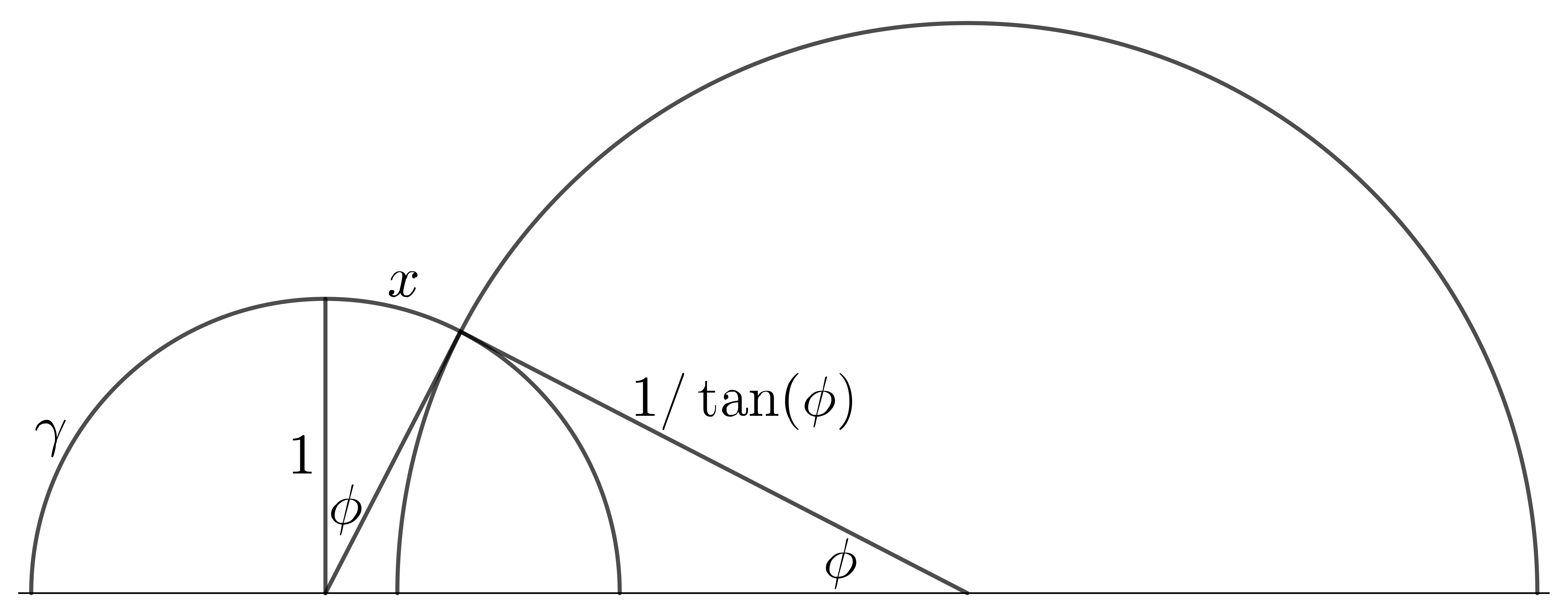}
\caption{}
\end{figure}
\begin{lemma} \label{up bound}
Suppose $\gamma$ is a geodesic of height 0 in $\H^3$ and $v \in N^1(\gamma)$,
and let $h$ be the height of highest point on the geodesic ray starting at $v$. 
Then writing $v - v_0(\gamma) = x + i \theta$,
with $|\theta| \le \pi$,
we have
\begin{equation} \label{x bound}
|x| < 2e^{-h}
\end{equation}
always, 
and 
\begin{equation} \label{theta bound}
|\theta| \le 2e^{-h} 
\end{equation}
whenever $h > 0$. 
\end{lemma}
\begin{proof}
For \eqref{x bound}, 
we observe that $h$ is greatest, 
given $x$,
when $\theta = 0$.
We then get 
$$|x| = \log (\sec \phi + \tan \phi)$$
 and 
$$h = \log \cos \phi - \log \sin \phi = -\log \tan \phi,$$
where $\phi$ is the (positive) (Euclidean) angle between the base points of $v_0$ and $v$
measured from the center of the semicircle $\gamma$ in the Poincare model. 
Then 
%\eqref{x bound} follows immediately from these equations and the estimates $\sec \phi < \tan \phi + 1$ and $e^x  \ge x + 1$. 
\begin{align*}
|x| < \log(2\tan\phi + 1) 
\le  2\tan\phi = 2e^{-h}.
\end{align*}

For \eqref{theta bound},
we just observe that when $h > 0$,
then $|\theta| < \pi/2$ and  $h \le -\log \sin |\theta|$,
so $|\theta| \le 2\sin |\theta| \le 2e^{-h}$. 
\end{proof}
%While it would be interesting to discuss the shape of this bald spot, 
%it turns out that it suffices for our purposes to control its size.
{
\def\myCutoffHeight{\mathbf h}
\def\reducedCutoffHeight{\mathbf{h'}}
\let\rch\reducedCutoffHeight
We can now estimate the number of bad connections 
 (whose absence forms the bald spot) 
for a given good curve $\gamma$. 
We let $h_+(\gamma) = \max(h(\gamma), 0)$. 
\begin{theorem} \label{bald spot size}
For all $\myCutoffHeight > 0$ \new{with $\myCutoffHeight < sR$}, and $\gamma \in \goodCurvesWithHeight{\myCutoffHeight}$, 
we have
$$
\left|\goodPantsWithHeightMoreThan\myCutoffHeight(\gamma) \right| < 
 C(\epsilon, M)\new{R} e^{2(h_+(\gamma)-\myCutoffHeight)}\left| \gp(\gamma) \right|. 
$$
\end{theorem}

\begin{proof}
Suppose that $\pi \in \goodPantsWithHeightMoreThan\myCutoffHeight(\gamma)$,
and let $\alpha$ be the third connection for $\pi$. 
A little hyperbolic geometry shows that the distance between the two new cuffs of $\pi$ and $\gamma \cup \alpha$
 can be at most $\log(1 + \sqrt 2) + O(e^{-R}) < 1$. 
So $h(\pi) > \myCutoffHeight$ and $h(\gamma) + 1 \le \myCutoffHeight$ implies that $h(\alpha) > \myCutoffHeight - 1=: \rch$. \new{(Note that the result is trivial when $h(\gamma)>\myCutoffHeight-1$.)}

Now suppose that $h(\alpha) \ge \rch \ge 0$. 
Then $\alpha$ has a cusp excursion of height $\rch$, 
which can be \new{either intermediate or non-intermediate}. 

By Theorem \ref{counting connections},
the number of $\alpha$ with an \new{intermediate }excursion of height $\rch$ or greater is $O(\new{R}e^{-2\rch}\epsilon^{-\new{4}}\goodPantsPerCurve)$.
Let us now bound the number of $\alpha$ with a \new{non-intermediate }excursion of this height. 
For every such $\alpha$, 
there is a lift of $\gamma$ with height at least 0 (and at most $h(\gamma)$),
 such \new{that } the inward-pointing normal at one endpoint of $\alpha$ lies in 
the square with sides $2e^{\rch - h}$
 (in the torus torsor coordinates)
centered at the unique unit normal of this lift of $\gamma$ that points straight up. 
We can think of this region as a region in $N^1(\gamma)$, 
and by \color{kwred}Theorem \ref{counting pants}\color{black},
there are 
$$
O(e^{2(h-\rch)}\epsilon^{4}e^{2R}) = 
O\left(R^{-1}e^{2(h-\rch)}\goodPantsPerCurve\right)
$$ such $\alpha$ that that have an inward-pointing normal at an endpoint that lies in the region. 
%It follows from Theorem \ref{counting connections},
%with some understanding of the reasoning behind Theorem \ref{counting pants},
%that there are $O(e^{2(h-\rch)}\goodPantsPerCurve)$ such $\alpha$.
As the geodesic $\gamma$ has $O(R)$ excursions into the cusp, 
there are $O(e^{2(h - \rch)}\goodPantsPerCurve)$ such $\alpha$
 (with \new{an initial or terminal excursion }of height $\rch$)
in total. 
The Theorem follows.
\end{proof}
}

\begin{remark}
The proof of Theorem \ref{bald spot size} shows that are regions  in the unit normal bundle to the geodesic $\gamma$ that force an initial or terminal excursion of the orthodesic.  These regions can be sizable, and if one thinks of the the normal vectors as  ``hair" on the geodesic, truly they represent regions near the top where there is no hair whatsoever. However  an intermediate cusp excursion farther out along an orthogeodesic can cause quite microscopic regions, which are topologically fairly dense but measure theoretically extremely sparse, to be part of the ``bald spot". So one could say that the ``bald spot" also includes some thinning of the hair all along the geodesic.
\end{remark}

The next result will control the total number of times that a given good curve $\eta$ will appear in the boundaries of all the (randomized) umbrellas that will be required in our construction. 

\begin{theorem}\label{total umbrella}
For any $\eta \in \goodCurves$, we have 
\begin{equation*} 
\sum_{\gamma \in \goodCurvesWithHeight{h_c}} 
  \sum_{P \in \goodPantsWithHeightMoreThan{h_c}(\gamma)}
    |\dout|\randomUmbrella(P, \gamma)(\eta)
    \le C(\e, M) R^{N_{\new{\umb}}+\new{2}} e^{-(2-K)(h_C - h_T)} \goodPantsPerCurve.
\end{equation*}
\end{theorem}

%\alert{We have to prove this theorem. Here is some stuff ripped from Section \ref{endgame} that might help.}

%We have the following estimate:
\newcommand{\EgammaSize}{\baldSpotSize\goodPantsPerCurve\randomUmbrellaBoundarySize{h(\gamma)}}
\newcommand{\EhSize}{R e^{-2h(\gamma)}\abs{\goodCurves}\EgammaSize|_{h(\gamma) = h}}
\newcommand{\EhSizeX}{\abs{\goodPants} R^3e^{K (h - h_T) - 2h_C}}
\newcommand{\EhSizeSummed}{R^3 e^{K(h_C - h_T) - 2 h_C}}
\newcommand{\totalUmbrellaBoundarySize}{\EhSizeSummed\numberOfGoodPants}
%\begin{theorem} \label{curve ratio}
%For any $h_T, h_C$ as above, 
%$$
%\norm{B_1}_1 \le \cem\totalUmbrellaBoundarySize.
%$$
%\end{theorem}
We can prove this theorem with Lemma \ref{high curve count} and Theorems \ref{bald spot size} and \ref{random umbrella}.
\begin{proof}
%Let $\gamma$ be a good curve. 
%Recall that $U_\gamma^U$ is the sum of $U(\pi, \gamma)$ over $\pi\in U_\gamma$. 
Let us first take $\gamma \in \goodCurvesWithHeight{h_C}$.
For any $\pi \in \goodPants(\gamma)$, 
we have,
by Theorem \ref{random umbrella},
that  
$$\dout\randomUmbrella(\pi, \gamma)(\eta) < \randomUmbrellaBoundarySize{h(\gamma)}.$$
 
Now we let $\randomUmbrella(\gamma)$ denote the sum 
 over $\pi \in \goodPantsWithHeightMoreThan{h_C}$ 
 of  $\randomUmbrella(\pi, \gamma)$. 
We obtain using Theorem \ref{bald spot size} and the above inequality that
\begin{align*}
\dout \randomUmbrella(\gamma)(\eta) 
&\le \cem \new{R} e^{2(h(\gamma) - h_C)}\goodPantsPerCurve\randomUmbrellaBoundarySize{h(\gamma)} \\
&= C(\epsilon, M) R^{N_{\new{\umb}}\new{+1}}e^{2(h(\gamma) + h_T - h_C) + K\max(0, h(\gamma) - h_T)}
      \frac{|\Pi|}{|\Gamma|^2}
\end{align*}

Then letting $\randomUmbrella(h)$ be the sum of all $U(\gamma)$ over all good curves $\gamma$ with $h(\gamma) \in [h, h+1)$, 
we obtain
\begin{align*}
\dout\randomUmbrella(h)(\eta)
&\le  C(\epsilon, M) R e^{-2h}|\Gamma| R^{N_{\new{\umb}}\new{+1}}
        e^{2(h + h_T - h_C) + K\max(0, h - h_T)}\frac{|\Pi|}{|\Gamma|^2} \\
&\le  C(\epsilon, M) R^{N_{\new{\umb}} + \new{2}} e^{2(h_T - h_C) + K\max(0, h - h_T)} \goodPantsPerCurve
\end{align*} 
Letting $\mathbf{\randomUmbrella}$ be $\randomUmbrella(h)$ summed over all integral $h$ in  $[0,  h_C)$, 
we obtain
$$
\dout\mathbf{\randomUmbrella}(\eta) 
< C(\epsilon, M) R^{N_{\new{\umb}} + \new{2}} e^{-(2-K)(h_C - h_T)} \goodPantsPerCurve. \qedhere
$$
\end{proof}

\subsection{Constructing a nearly geodesic surface} \label{endgame}
\def\steplabel#1{\bold{#1}\par}
Our object is to find a (closed) surface subgroup of $\Gamma$ that is $K$-quasi-Fuchsian and thereby prove Theorem \ref{main theorem}. 
We can assume that $K < 3/2$. 
%We then let $\epsilon = (K-1)/2C_0$, where $C_0$ is as it appears in Theorem \ref{good close to Fuchsian}.
\new{We then let $\epsilon$ be small enough so that  $K(2\epsilon)<K$, where $K(\epsilon)$ is from Theorem \ref{good close to Fuchsian}. }
We will let $h_T = 6 \log R$, and $h_c = h_T + 2(N_{\new{\umb}} + \new{6})\log R$,  so that  
\begin{equation} \label{small correction}
R^{N_{\new{\umb}}+\new{2}} e^{-(2-K)(h_C - h_T)} < R^{-4}.
\end{equation}
Since the left hand side appears in Theorem \ref{total umbrella}, this will give control over the boundary of the randomized umbrellas. 

We also choose $R$  large enough for all previous statements that hold for $R$ sufficiently large. This depends on the constants $h_T/\log R $ and $h_C/\log R$, so technically first we fix these constants, then we pick $R$, and then $h_T$ and $h_C$ are defined as numbers. 

We will determine the components for the surface (the pants and hamster wheels) in two steps. 
%\begin{figure}[ht!]
%\includegraphics[width=.7\linewidth]{Steps.pdf}
%\caption{The four steps of the proof of Theorem \ref{main theorem}, and their relation to the height of the components involved. The constants $h_T, h_T', h_T''$ are defined using results from previous sections, and $h_C$ and $h_S=h_T-2\log R$ are defined only at the end of the argument, by equation \eqref{small correction} in the Completion of the Proof of Theorem \ref{main theorem}.}
%\label{F:Steps}
%\end{figure}

\steplabel{Step 0: All good pants with a cutoff height.}
Let 
$$
A_0 := \sum_{P \in \goodPantsWithHeight{h_c}} P
$$
be the formal sum of all unoriented good pants $P$ over $P \in \goodPantsWithHeight{h_C}$.

\steplabel{Step 1: The Umbrellas.}
We let 
$$
A_1 := 
\sum_{\gamma \in \goodCurvesWithHeight{h_c}}
  \sum_{P \in \goodPantsWithHeightMoreThan{h_c}(\gamma)}
    \randomUmbrella(P, \gamma),
$$
be the sum of randomized umbrellas for all good pants that have a cuff above the cutoff height and a cuff below the cutoff height. Note that $A_1$ is a linear combination of (unoriented) good assemblies. 
 
\begin{proof}[Proof of Theorem \ref{main theorem}]
%and therefore, 
%by \eqref{correction size},
%\begin{equation} \label{error control}
% |B_2|(\alpha), {|B_3|(\alpha) < R^{-4}}{\crunchedGoodPantsPerCurve}.
%\end{equation}
\def\total{\mathbf A}
We claim that we can then assemble the pants and good assemblies of $\total := A_0 + A_1$
 (after clearing denominators)
 into a \emph{closed} good assembly $\A$.
 
Let us see that we can do this at every $\gamma \in \goodPantsWithHeight{h_C}$. 
We first consider a rational sum $A_\gamma$ comprising one of each good pants in $\goodPants(\gamma)$ 
(including the ones in $\goodPantsWithHeightMoreThan{h_C}$), 
and also all the good assemblies in $A_1$ (counting multiplicity) with $\gamma$ as an external boundary. 
By \eqref{small correction} and Theorem \ref{total umbrella}, 
the latter term has total size at most $R^{-4} \goodPantsPerCurve$. 
Therefore,
by Theorem \ref{perturbed matching},
after clearing denominators,
we can apply the doubling trick (see Section \ref{components:doubling}) and match off all the oriented versions of the good pants and good assemblies in $A_{\gamma}$.
Moreover, by Property 1 in Theorem \ref{theorem umbrella}, 
and its analog in Theorem \ref{random umbrella},
we can replace each pants $\pi$ in $\goodPantsWithHeightMoreThan{h_C}$ with its corresponding umbrella $\randomUmbrella(\pi, \gamma)$,
and still have everything be well-matched.
Thus we have constructed a closed $2\epsilon$-good assembly, and we are finished.

%At each $\gamma$ with $h_S < h(\gamma) \le h_C$, 
% we can match according to $\sigma_\gamma$, 
% and use the assigned umbrellas for the unmatched pants.
%We can also internally match within each of the umbrellas.
%At the umbrella boundaries we match the peripheral hamster wheels with the semi-randomizing hamster wheels.
%For $\gamma$ with $h''_T < h(\gamma) \le h_S$,
%we can match by the observation in Step 0.
%
%Finally, 
% suppose that $h(\gamma) \le h''_T$.
%We've matched up all hamster wheels in the umbrellas;
%let $\total'(\gamma)$ be what remains. 
%Then $A_0(\gamma) \subset \total'(\gamma) \subset (A_0 + A_2 + A_3)(\gamma)$, 
%so 
%$$
%|\total'(\gamma) \symDiff A_0(\gamma)| < |A_2(\gamma)| + |A_3(\gamma)| < R^{-4}\goodPantsPerCurve
%$$
%by \eqref{error control}.
%By this estimate and \eqref{sym diff}, 
%$$
%|\total'(\gamma) \symDiff  \ogoodPants(\gamma)|  < R^{-4}\goodPantsPerCurve,
%$$
%and likewise with $\gamma$ replaced $\gamma\inv$. 
%Therefore $(\total'(\gamma), \total'(\gamma\inv))$ is nearly standard, and hence, 
%by  Theorem \ref{perturbed matching}, 
% can be well-matched. 
%By construction we obtain a good assembly. 

To prove ubiquity, we want to make a quasi-Fuchsian subgroup with limit set close to a given circle $C$. 
To do so, we first find a point in the hyperbolic plane through $C$ that has height less than 0,
and then find a pair of good pants $P$ whose thick part passes near this point, and which ``nearly lies'' in this hyperbolic plane. 
If $S_\A$ is not connected, we take a component containing a copy of $P$, which, passing to a subassembly, we will now denote by $S_\A$. 

By Theorem \ref{good close to Fuchsian},
  $\rho_\A$ is $K$-quasi-Fuchsian;
  moveover, its limit set comes close to the approximating plane through $P$. 
\end{proof}

\draftnewpage
\appendix
\section{Good assemblies are close to perfect}
\begin{tododo}
%\item
%\st{Add references or proofs for Theorems \protect\ref{thm:hyperbolic-local-global} and \protect\ref{thm:qi-d}}.
%\item
%Rewrite the introductory subsection, 
%so as to (at least roughly) define ``compliant map $e\from \F(\d \hat A) \to \F(\d A)$",
%define ``map of $\epsilon$-bounded distortion at distance $D$",
%and then state that we can make a compliant map, and that every compliant map has $\epsilon$-bounded distortion at distance $D$.
%\item
%Use Theorem \ref{thm:app:linear-local-to-global} in the proof of Theorem \ref{thm:app:control}.
%%\item
%%\st{Fold Lemma \ref{lem:app:triangle}} into Lemma \ref{lem:app:homologous}.}
%\item
%Split Section \ref{subsec:control}, and move the other subsections around as on the whiteboard.
\item
Review Theorems \ref{generate good}, \ref{thm:e-compliant} and \ref{thm:app:control}, 
and get all the constants straight. 
\item
Consider proving Theorem \ref{thm:app:control} for pants first, and then for general assemblies. 
\item
Review for completeness, correctness and professionalism.
\end{tododo}

\subsection{Introduction}
The goal of this appendix is to prove Theorem \ref{good close to Fuchsian}. A version of this result, using assemblies of only good pants (without any good hamster wheels), is the content of \cite[Section 2]{KM:Immersing}. 

Our approach is  different from that of \cite{KM:Immersing}. In fact, our discussion provides a complete alternative to \cite[Section 2]{KM:Immersing} that is shorter, simpler, 
generalizes much more easily to other semisimple Lie groups,
and is much more constructive. (In principle,   it is wholly constructive.) In addition, this appendix generalizes the results of \cite{KM:Immersing} in such a way as to include hamster wheels. 

As a result of this appendix (and \cite{KW:counting}), we do not rely on any results from \cite{KM:Immersing}. 

We consider a good assembly $\A$ with $S_\A$ connected. As in \cite{KM:Immersing}, we construct a perfect assembly $\hat{\A}$ to which $\A$ is compared. \new{Essentially $\hat \A$ is made of a perfect pants or hamster wheel for every good one in $\A$. } In contrast to \cite{KM:Immersing}, we have no need to define an interpolation (one-parameter family of good assemblies) between $\hat\A$ and $\A$. 

We also construct a map $e$ relating $\hat\A$ and $\A$, \new{mapping each perfect component in $\hat A$ to the corresponding good one in $\A$.} The construction of $\hat{\A}$ is carried out in Section \ref{app:perfect-model}, where we also prove quantitative estimates called $M, \epsilon$-compliance on the comparison map $e$. These estimates should be thought of as \emph{hyperlocal}, in that they involve only pairs of geodesics that are adjacent in the assembly. 

The compliance estimates are a basic starting point for the analysis of the assembly. In the case where $\A$ only contains pants, the compliance estimates are very easy. In the case of hamster wheels, many of the required estimates are deferred to the final section of this appendix, Section  \ref{subsec:hw}. 

Geodesics in $\A$ can come very close to each other in hyperbolic distance. Thus, for two given geodesics in $\A$ that are bounded hyperbolic distance apart, they could have large combinatorial distance in $\A$. The novelty of this appendix, in comparison to \cite{KM:Immersing}, is that elementary estimates on products of matrices in $\PSL_2(\C)$ are used to understand the relative position of such pairs of geodesics in $\A$. These elementary estimates are contained in Section \ref{subsec:matrix}. 

Sections \ref{subsec:linearseq} and \ref{app:bounded-distortion} apply these algebraic estimates to obtain a geometric estimate on the comparison between $\hat\A$ and $\A$, called $\e$-bounded distortion to distance $D$. The definition is isolated in Section \ref{app:defs}.  This estimate says that relative positions in $\A$, as measured by isometries taking one frame in $\A$ to another frame in $\A$ at distance at most $D$ from the first frame, are almost the same as the corresponding relative positions in the perfect assembly $\hat\A$.   

The bounded distortion estimates are quite strong, and  express that on bounded scales $\new{\A}$ is a ``nearly isometric" version of $\hat\A$. In Sections \ref{app:boundary} and \ref{subsec:conclusions} we explain how standard results allow one to go from bounded distortion to global estimates on $\A$ and to conclude the proof of Theorem \ref{good close to Fuchsian}.

\subsection{An estimate for matrix multiplication}\label{subsec:matrix}
For any element $U$ of a Lie algebra (for a given Lie group), 
we let $U(t)$ be a shorthand for $\exp(tU)$.
Let $X\in \sl_2(\R)$ be $\tutu{\frac 12} 0 0 {-\frac 12}$,
and let $\theta = \tutu 0 {-\frac 12}  {\frac12} 0$.
We let $Y = e^{(\pi/2) \ad_\theta} X$,
so that $Y(t) = \theta(\pi/2) X(t) \theta(-\pi/2)$. 
We will think of $\SL_2(\R)$ as a subset of $M_2(\R)$;
we may then add or subtract elements of $\SL_2(\R)$ from each other, and take their matrix (operator) norm%
\footnote{\new{Explicitly $\norm{U} = \sup \{ |Uv| \mid |v| = 1 \}$}, where $|v|$ is the $L^2$ norm of $v$.}. 
We observe that $\norm X, \norm Y \le 1$. 

Let's begin by stating a theorem.
\begin{theorem} \label{thm:app:main-matrix-estimate}
Suppose $(a_i)_{i=1}^n$, $(b_i)_{i=1}^n$,  $(a'_i)_{i=1}^n$, $(b'_i)_{i=1}^n$ are sequences of complex numbers,
and $A, B, \epsilon$ are positive real numbers such that $\epsilon < \min(1/A, 1/e)$,
\[
\sum_{i=1}^n |a_i| e^{|b_i|} \le B, 
\]
and for all $i$,
\begin{enumerate}
\item 
$2|a_i| e^{|b_i|+1} \le A$
\item \label{delta b}
$|b_i - b_i'| < \epsilon$, and
\item \label{delta a}
$|a_i - a_i'|  < \epsilon |a_i|$.
\end{enumerate}
Then 
\begin{equation}
\norm{
\prod_{i=1}^n Y(b_i) X(a_i) Y(-b_i) - \prod_{i=1}^n Y(b'_i) X(a'_i) Y(-b'_i)
}
\le
12 e^{A + 2B} B\epsilon. 
\end{equation}
\end{theorem}

We will begin our discussion of Theorem \ref{thm:app:main-matrix-estimate} with a few simple lemmas:
\begin{lemma}
For all complex numbers $a$ and $b$,
we have 
\begin{equation} \label{eq:badY}
\norm {e^{b\, \ad_Y} aX} \le |a| e^{|b|}.
\end{equation}
\end{lemma}
\begin{proof}
By direct calculation, 
we have
\[ 
e^{b\, \ad_Y} aX = a((\cosh b) X \new{+}  (\sinh b) \theta).
\]
Moreover, $\< X, X\>  = \frac 12$, $\<\theta, \theta\> = \frac 12$, and $\<X, \theta\> = 0$. 
(Where $\< \cdot, \cdot \>$ denotes the standard inner product on $M_2 \equiv \R^4$.)
So 
\begin{align*}
\norm{e^{b\,\ad_Y} aX}^2_2 
&= a^2\left(\frac 12 \cosh^2 b + \frac 12 \sinh^2 b\right) \\
&\le a^2\left(\frac 12 e^{2|b|} + \frac 12 e^{\new{-}2|b|}\right) \\
&\le a^2 e^{2 |b|}. 
\end{align*}
We have the derived the desired upper bound for the $L^2$ (Hilbert-Schmidt) norm, and hence for the operator norm. 
\end{proof}

\begin{lemma}
For any $U \in \sl_2(\C)$ (or much more generally), 
we have 
\begin{equation} \label{eq:exp-bound1}
\norm{\exp(U)} \le e^{\norm{U}}
\end{equation}
and,
when $\norm{U} \le A$,
\begin{equation} \label{eq:exp-bound2}
\norm{\exp(U) - 1} \le e^{\norm{U}} - 1\le \frac{e^A - 1}A \norm{U} \le e^A \norm{U}
\end{equation}
\end{lemma}
\begin{proof}
The point is that we're thinking of both $\sl_2(\C)$ and $\SL_2(\C)$ as subsets of $M_2(\C)$,
and hence we have $\exp(U) = e^U$ for $U \in \sl_2(\C)$. 
Then \eqref{eq:exp-bound1} and the first inequality of \eqref{eq:exp-bound2} follow from the power series for $e^x$;
the second inequality
%follows from the observation that $(e^x - 1)/x$ is increasing for $x > 0$, 
%which in turn 
follows from the convexity of $e^x$. 
\end{proof}

\begin{lemma} \label{lem:UA}
We have, for $U\in M_n(\C)$ \new{and $A \in \GL_n(\C)$}, 
\begin{equation}
\norm{U - AUA^{-1}} \le 2\norm{U-1} \norm{A-1}\norm{A^{-1}}.
\end{equation}
\end{lemma}
\begin{proof}
We have 
\begin{align*}
U - AUA^{-1} &= (UA - AU) A^{-1} \\
&= [U, A] A^{-1} \\
&= [U-1, A-1] A^{-1}
\end{align*}
and $[S, T] \le 2 \norm S \norm T$ for all $S$ and $T$. 
\end{proof}

We can now proceed to our two basic estimates:
\begin{lemma} \label{lem:change-b}
Suppose $|b' - b| \le 1$, and $|a| e^{|b|} \le A$. 
Then 
\begin{equation}
\norm{Y(b) X(a) Y(-b) - Y(b') X(a) Y(-b')} \le 10 e^{A}|a| e^{|b|} |b'-b|. 
\end{equation}
\end{lemma}
\begin{proof}
Under the given hypotheses, 
we have,
by \eqref{eq:badY} and \eqref{eq:exp-bound2},
\begin{equation}
\norm{\exp(e^{b \, \ad_Y} aX) - 1} \le e^A |a| e^{\new{|b|}}. 
\end{equation}
Hence $\norm{U-1} \le e^A |a| e^{\new{|b|}}$,
where
$U = Y(b) X(a) Y(-b)$. 
Also, by \eqref{eq:exp-bound2} and \eqref{eq:exp-bound1},
\begin{equation}
\norm{Y(b'-b) - 1} \le (e-1) |b'-b|.
\end{equation}
and
\[
\norm{Y(b-b')} \le e^{|b'-b|} \le e. 
\]
So,
by Lemma \ref{lem:UA} and the above,
\begin{align*}
\norm{U - Y(b' - b)UY(b-b')} 
&\le 2 \norm{U-1}\norm{Y(b'-b) - 1}\norm{Y(b - b')} \\
&\le 2e^A |a|e^{|b|} (e-1) e |b'-b| \\
&\le 10 e^A |a| e^{\new{|b|}} |b'-b|. \qedhere
\end{align*}
\end{proof}

\begin{lemma} \label{lem:change-a}
When $|a'-a|e^{|b|} \le 1$ and $|a| e^{|b|} \le A$,
\begin{equation}
\norm{Y(b)X(a')Y(-b) - Y(b)X(a)Y(-b)} \le 2 e^{A + |b|}|a'-a|.
\end{equation}
\end{lemma}
\begin{proof}
We then have,
by \eqref{eq:badY}, \eqref{eq:exp-bound1}, and \eqref{eq:exp-bound2},
\begin{align*}
\left\Vert Y(b)X(a')Y(-b)\right. &- \left.Y(b)X(a)Y(-b) \right\Vert \\
&= \norm{\exp(e^{b\, \ad_Y} a'X) - \exp(e^{b\, \ad_Y} aX) } \\
&= \new{\norm{\exp(e^{b\, \ad_Y} aX) (\exp(e^{b\, \ad_Y} (a'-a)X) - 1)}}  \\
&\le e^A (e-1) e^{|b|} |a'-a|. \qedhere
\end{align*}
\end{proof}

Lemmas \ref{lem:change-b} and \ref{lem:change-a} can be combined into the following. 
\begin{lemma} \label{lem:change}
When $A > 0$ and $a, b, a', b'$ are complex numbers such that $2|a|e^{|b|+1} \le A$, $|b-b'| < \epsilon$, $|a-a'| < \epsilon |a|$, and $\epsilon < \min(1/e, 1/A)$,
we have
\begin{equation} \label{change:0}
\norm{
Y(b) X(a) Y(-b) - Y(b') X(a') Y(-b')
}
\le 
12 e^{A + |b|}|a|\epsilon.
\end{equation}
\end{lemma}
\begin{proof}
By Lemma \ref{lem:change-b},
we have 
\begin{equation} \label{change:1}
\norm{Y(b) X(a) Y(-b) - Y(b') X(a) Y(-b')} \le 10 e^{A}|a| e^{|b|} |b'-b|. 
\end{equation}
Then by Lemma \ref{lem:change-a} and $|b' - b| < 1$,
we have
\begin{equation} \label{change:2}
\norm{Y(b')X(a')Y(-b') - Y(b')X(a)Y(-b')} \le 2 e^{A + |b|}|a'-a|.
\end{equation}
Combining \eqref{change:1} and \eqref{change:2},
we obtain
\eqref{change:0}. 
\end{proof}

Before proving Theorem \ref{thm:app:main-matrix-estimate},
we observe that
\begin{equation} \label{eq:simple}
\norm{Y(b)X(a)Y(-b)} \le e^{e^{|b|} |a|}
\end{equation}
by \eqref{eq:badY} and \eqref{eq:exp-bound1}.

\begin{proof}[Proof of Theorem \ref{thm:app:main-matrix-estimate}]
When $\norm{U_i}, \norm{U'_i} \le A_i$, and $A_i \ge 1$,
then, letting 
$$\U_i = \prod_{j < i} U'_i \prod_{j \ge i} U_i,$$
we have
\begin{align} \label{eq:mme-1}
\norm{\prod_i U_i - \prod_i U'_i} 
&\le \sum_{i=1}^{n} \norm{\U_i - \U_{i+1}} \\
\notag &\le \left(\prod_i A_i\right)\sum_i\norm{U_i - U'_i}.
\end{align}
In \eqref{eq:mme-1} we let $U_i = Y(b_i) X(a_i) Y(-b_i)$,  $U'_i = Y(b'_i) X(a'_i) Y(-b'_i)$, and $A_i = e^{2 |a_i| e^{|b_i|}}$. 
Our hypotheses imply that 
\[
|a'_i| e^{|b'_i|} \le (1 + 1/e) |a_i| e^{|b_i| + 1/e}\le 2 |a_i| e^{|b_i|}
\]
and hence $\norm{U_i}, \norm{U'_i} \le A_i$ by \eqref{eq:simple}. 
Moreover, by our hypotheses and Lemma \ref{lem:change}, 
\begin{equation} \label{eq:mme-2}
\prod_i A_i  \le e^{2B},
\end{equation}
and 
\begin{equation}
\sum \norm{U_i - U'_i} \le 12e^A \epsilon \sum_i e^{|b_i|} |a_i| \le 12e^A \epsilon B. 
\end{equation}
The Theorem follows.
\end{proof}

\subsection{Frame bundles and \texorpdfstring{$\epsilon$}{epsilon}-distortion}\label{app:defs}
By an $n$-frame in an oriented $n$-dimensional Riemannian manifold $M$ 
we of course mean a point $x \in M$ along with an orthonormal basis for $T_x M$ with positive orientation. We let $\F M$ denote the set of all $n$-frames in $M$. 

We observe that any orthonormal set $v_1, \ldots, v_{n-1}$ of $n-1$ vectors in $T_x M$ can be completed to a unique $n$-frame $v_0, v_1, \ldots, v_{n-1}$.
We call the latter the \emph{associated $n$-frame} for the former. If $v \in T^1(\H^2) \subset T^1(\H^3)$, we can thus complete $v$ to be a 2-frame $w, v$ for $\H^2$, and then complete $w, v$ to be a 3-frame $q, w, v$ for $\H^3$; we call this the associated 3-frame for $v$. 
%If $\gamma$ is a geodesic in a 2-manifold $M$, then any unit normal vector $v$ to $\gamma$ determines a 2-frame $\F_2(v) \equiv w, v$ in $M$ in this way.  And any 2-frame $F$  in $\H^2$ can be completed to a unique 3-frame $\F_{23}(F)$ in $\H^3$; 
%likewise for a 2-frame in a geodesic subsurface of an oriented hyperbolic 3-manifold. 

If $\gamma$ is an \emph{oriented} geodesic in a 3-manifold $M$, then any unit normal vector $v$ to $\gamma$ determines a unique 2-frame $w, v$  for $M$, where $w$ is tangent to $\gamma$ and is positively oriented. We call this the associated 2-frame for $v$ with respect to $\gamma$. We can complete this frame to a 3-frame $q, w, v$ in $M$; we call this the associated 3-frame for $v$ with respect to $\gamma$. Let $\F(\gamma)$ be the set of all frames obtained in this way. Moreover, if $\gamma$ is an unoriented geodesic, we let $\F(\gamma) = \F(\gamma^+) \cup \F(\gamma^-)$, where $\gamma^+$ and $\gamma^-$ are the two possible oriented versions of $\gamma$. 

Now suppose that $v \in N^1(\gamma)$, where $\gamma$ is an oriented geodesic in $\H^2 \subset \H^3$ (and $N^1(\gamma) \subset T^1(\H^2)$). 
There are then two associated 2-frames for $v$: 
one by completing $v$ to be a 2-frame for $\H^2$, 
and the second by thinking of $v \in \N^1(\gamma) \subset T^1(\H^3)$, 
and then taking the associated 2-frame for $v$ with respect to $\gamma$, as described in the previous paragraph. 
We observe that when $v$ points to the \emph{left} of $\gamma$ (in $\H^2$), 
then these two associated 2-frames are the same (and hence the two associated 3-frames are also the same). 
The same story applies when $v \in N^1(\gamma)$, 
and $v$ and $\gamma$ lie in some geodesic subsurface of a hyperbolic 3-manifold. 
In particular, 
the associated 3-frames are the same when $\gamma$ is a boundary curve of a geodesic subsurface with geodesic boundary (with $\gamma$ oriented so that the subsurface is on its left), 
and $v$ is an \emph{inward pointing} normal vector to $\gamma$. 
In these settings we will simply refer to the associated 3-frame for $v$. 

%We observe that $\F_{23}(\F_2(v)) =\F_3(v, \gamma)$ when $\gamma$ is the oriented boundary of an geodesic submanifold $S$ of a hyperbolic 3-manifold, and $v$ is normal vector to $\gamma$ pointing into $S$. 
\subsubsection*{Distance and distortion in $\F\H^3$}
We fix
\alert{or have fixed? }%
a left-invariant metric $d$ on $\G$,
and, for $g \in \G$, let $d(g) = d(1, g)$. 
We observe the following:
\begin{lemma}
For all $D$ there exists $D'$:
If $U \in \SL_2(\C)$, and $\norm{U} < D$, then $d(U) < D'$.
Likewise with $\norm{\cdot}$ and $d(\cdot)$ interchanged. 
\end{lemma}
\begin{lemma}
For all $D, \epsilon$ there exists $\delta$:
Suppose that $U, V \in \SL_2(\C)$, $\norm{U}, \norm{V} < D$,  and $\norm{U - V} < \delta$.
Then $d(U, V) < \new{\epsilon}$. 
\end{lemma}

If $u, v \in \F\H^3$, we let $u \to v \in \G$ be uniquely determined by $u \cdot (u \to v) = v$. 
Given $X \subset \F\H^3$ and a map $\tilde{e}\from X \to \F\H^3$, 
we say that $\tilde{e}$ has \emph{$\epsilon$-bounded distortion to distance $D$}
if
\begin{equation} \label{eq:app:ep-dist}
d(u \to v, \tilde e(u) \to \tilde e(v)) < \epsilon
\end{equation}
whenever $u, v \in X$ and $d(u, v) < D$ (where $d(u, v) = d(u \to v)$).

Given $\tilde{e_1}, \tilde{e_2}\from X \to \F\H^3$,
we say that $\tilde{e_1}$ and $\tilde{e_2}$ are \emph{$\epsilon$-related}
if 
$$
d(\tilde{e_1}(x)\to\tilde{e_2}(x)) < \epsilon
$$
whenever $x \in X$. 
%Now suppose that $M$ and $N$ are hyperbolic 3-manifolds, 
%$f\from M \to N$ is continuous, 
%$A \subset \F(M)$, 
%and $e\from A \to \F(N)$ is arbitrary.
%Then we can lift $f 

We observe that our distance estimates can be concatenated in a natural way:
\begin{lemma} \label{lem:app:concat}
For all $\epsilon, D, k$ there exists $\delta$:
When $u_0, \ldots u_k, v_0 \ldots v_k \in \F(\H^3)$, and
$$d(u_i \to u_{i+1}, v_i \to v_{i+1}) < \delta,$$
and $$d(u_i \to u_{i+1}) < D,$$
then $$d(u_0 \to u_k, v_0 \to v_k) < \epsilon.$$ 
\end{lemma}

\subsubsection*{Some formal shenanigans}
%Suppose $C$ is a perfect pair of pants or hamster wheel in a quotient of $\H^2 \subset \H^3$;
%we let $\df C$ denote the 2-frames (or the 3-frames in $\H^3)$ associated to the inward pointing unit normal vectors for $\d C$.
%If $C$ is a \emph{good} component in a quotient of $\H^3$,
%we let $\df C$ denote the 3-frames associated to all unit normal vectors for $\d C$.
%
%If $\A$ is a perfect or good assembly,
%we let $\df \A$ denote the union\ann{Union or disjoint union?} of $\df C$ for each component $C$ of $\A$.
%Note that $\df C$ and $\df C'$ are disjoint for two adjacent components of a perfect assembly, 
%while the objects with the same notation overlap for two adjacent components of a good one. 
For a perfect pants or hamster wheel $Q$,
we let $\df Q$ be the set of associated 3-frames for inward pointing unit vectors based at a point in $\d Q$. 
For a \emph{good} pants or hamster wheel $Q$
%\cheesef{Programmers in Swift or Java will understand the difference between a perfect component and a good one that happens to be perfect.}
we let $\df Q$ be union of $\F(\gamma)$ for all \emph{oriented} $\gamma \in \d Q$. 
%For a perfect or good assembly $A$ we let $\df A$ be the union of $\df Q$ over all components $Q$ of $A$. 

%As in Section \ref{app:defs}, an element of $\df Q$ also gives rise to a frame in $\F\H^3$, and we let 
%\def\tdf{\tilde \d^\F}
%$\tdf Q$ denote the set of frames that arise in this way. Similarly, for an assembly $\A$, we have a corresponding subset 
% $\tdf\!\A \subset \F \H^3$.

In the setting for Theorem \ref{good close to Fuchsian},
we have components $C_i$ glued along boundaries,
so that some boundary geodesic of one $C_i$ is geometrically identified with a boundary geodesic of another $C_j$,
and so forth, 
to form an assembly $\A$. 
In Section \ref{app:perfect-model},
we will construct a \emph{perfect model} $\hat \A$ for $\A$;
in this section we should discuss what such a thing would be. 

%First of all, 
%we should say what kind of object $\A$ is to begin with. 
%It would be nice to say that is a 3-manifold, 
%or a surface inside a 3-manifold,
%made out of components that have been glued in the way we've required.
%But that it a large part of what we're trying to prove. 
%So the best we can say \emph{a priori} is that $\A$ is a countable collection of components,
%with infinitely isometric copies of each of the $C_i$,
%so that at each boundary of each $C_i$ we have an associated boundary of a copy of the $C_j$ it's supposed to be glued to,
%with the right geometric gluing.
%An assembly will naturally comes with a \emph{formal} fundamental group
%(from the \emph{topological} gluing)
%$\pi_1(\A)$ (or $\Deck(\tilde\A/\A)$),
%along with a map $\rho_\A\from \pi_1(\A)  \to \G$
%that is not a priori faithful.
%(In the proof of Theorem \ref{good close to Fuchsian} we will not even use that it is discrete). 

The perfect model will provide a perfect component $\hat C_i$ for each $C_i$ in $\A$, 
along with a suitable map $h\from \hat C_i \to M$ taking $\d \hat C_i$ to $\d C_i$, 
up to homotopy through such maps. 
Then the gluings of the $C_i$ determine gluings of the $\hat C_i$, and these $\hat C_i$ then form a topological surface. 
The perfect model $\hat \A$ also includes geometric (and isometric) identifications of the components of the $\d \hat C_i$,
so that we obtain an actual geometric surface $S_{\hat \A}$. 
In this \color{kwred}way \color{black} we obtain a homotopy class of maps $h\from S_{\hat \A} \to M$, 
such that $h$ maps each component $\hat C_i$ in $S_{\hat \A}$ to its corresponding component $C_i$ in $M$. 
We can then lift this map to $\tilde h\from \H^2\to \H^3$;
while the lift depends on the exact choice of $h$,
the homotopy class of $h$ determines the correspondence of each elevation%
\footnote{component of the whole preimage}
of each $\hat C_i$ with an elevation of $C_i$,
and also determines how the boundary geodesics of those \color{kwred}elevations \color{black} correspond. 

%A perfect model for $\A$ is an assembly $\hat \A$ of perfect components $\hat C_i$ with the same topological gluings as in $\A$ (and arbitrary geometric gluings). 
%Hence $\pi_1(\hat \A) = \pi_1(\A)$ canonically,
%and we have a representation $\rho_{\hat\A}\from \pi_1(\A) \to \G$ that is faithful and discrete.
%Such is the nature of perfection.

Suppose we have maps $e\from \df(\hat C_i) \to \df(C_i)$ that maps frames over each boundary of $\hat C_i$ to frames over the corresponding boundary in $C_i$. 
Then we can use $\tilde h$ to get a canonical lift of $e$ to $\df(\tilde {\hat{\new{\A}}})$:
for each boundary geodesic of elevation of each component $\hat C_i$,
take the corresponding object for $C_i$---determined, if you like, by $\tilde h$---and the lift $\tilde e$ of $e$ is determined by requiring that $\tilde e$ maps the (associated frames for) the inward pointing normal vectors at this boundary geodesic to our chosen corresponding one for $C_i$. 
We then say that $e\from \df(\hat \A) \to \df(\A)$ (or $e\from \df(\hat C_i) \to \df(C_i)$ or even $e\from \F(\hat \gamma) \to \F(\gamma)$)
has $\epsilon$-distortion at distance $D$ if and only if $\tilde e$ does. 

\subsection{Linear sequences of geodesics}\label{subsec:linearseq}
In this subsection we observe a corollary to Theorem \ref{thm:app:main-matrix-estimate}.

A \emph{linear sequence of geodesics} in $\H^2$
is a sequence $(\gamma_i)_{i=0}^n$ of disjoint geodesics in $\H^2$ with common orthogonals,
such that each geodesic separates those before it from those after it. 
We orient all the geodesics in a linear sequence such that the geodesics that follow a given geodesic (in the sequence)
are to the left of that geodesic. 

Given any two oriented geodesics with a common orthogonal,
there is a unique orientation-preserving isometry taking one to the other that maps one foot of the common orthogonal to the other. 
We can likewise define an isometry from the frame bundle (or unit normal bundle)
of one geodesic to the same object for the other,
such that one foot (as a unit vector) is mapped to the negation of the other foot. 

Returning now to a linear sequence $(\gamma_i)$ of geodesics,
given $x_0 \in \gamma_0$,
we can inductively define $x_i \in \gamma_i$ such that $x_{i+1}$ and $x_i$ are related by the isometry defined in the previous paragraph.
We say that the $x_i$'s form a \emph{homologous sequence} of points on the $\gamma_i$'s. 
(Here we are inspired more by the natural meaning of ``homologous" than by its meaning in mathematics). 
We have the following lemma:
\begin{lemma} \label{lem:app:homologous}
Suppose $(x_i)$ is a homologous sequence of points on a linear sequence $\gamma_i$ of geodesics, and take $y \in \gamma_n$. 
Then
$$d(x_n, y) \le d(x_0, y).$$
\end{lemma}

Before proving this lemma, we will prove another little lemma:
\begin{lemma} \label{lem:app:triangle}
Suppose $\eta$ is a geodesic in $\H^2$,
and $\gamma_0$, $\gamma_1$ are distinct geodesics orthogonal to $\eta$,
and oriented so that both point towards the same side of $\eta$.
Suppose $y_i \in \gamma_i$, for $i = 0, 1$.
Then,
taking signed distances along the $\gamma_i$,
\begin{equation}
|d(\eta, y_0) - d(\eta, y_1)| < d(y_0, y_1).
\end{equation}
\end{lemma}
\begin{proof}
Suppose the two $y_i$ are on the same side of $\eta$.
Then,
taking unsigned distances, 
we have
$$d(\eta, y_1) < d(\eta, y_0) + d(y_0, y_1)$$
by the triangle inequality,
and the same with $y_0$ and $y_1$ interchanged. 
The Lemma follows in this case. 

Now suppose that the two $y_i$ are on opposite sides of $\eta$.
Then the segment from $y_0$ to $y_1$ intersects $\eta$ in a point $u$,
and again taking unsigned distances, 
$d(y_0, \eta) < d(y_0, u)$ and $d(\eta, y_1) < d(u, y_1)$. 
The Lemma follows. 
\end{proof}
%\begin{lemma} \label{lem:app:triangle}
%Suppose that $\gamma_0$ and $\gamma_1$ are disjoint, 
%$x_0 \in \gamma_0$ and $x_1 \in \gamma_1$ are homologous,
%and $y_i \in \gamma_i$, for $i = 0, 1$.
%Then 
%$$
%|d(x_0, y_0) - d(x_1, y_1)| \le d(y_0, y_1).
%$$
%\end{lemma}
%\begin{proof}
%Let $\eta$ be the common orthogonal to $\gamma_0$ and $\gamma_1$.
%Then $$d(x_i, y_i) = d(y_i, \eta) - d(x_i, \eta),$$ and $d(x_i, \eta)$ is independent of $i$.
%But $$|d(y_1, \eta) - d(y_0, \eta)| < d(y_0, y_1)$$ by the triangle inequality. 
%\end{proof}

\begin{proof} [Proof of Lemma \ref{lem:app:homologous}]
For each $i$,
let $y_i$ denote the intersection of $\gamma_i$ with the geodesic segment from $x_0$ to $y$.

We claim that $|d(x_i, y_i) - d(x_{i+1}, y_{i+1})| \le d(y_i, y_{i+1})$ for each $i$ with $0 \le i < n$
(where the first two distances are signed distances).
To see this, let $\eta$ be the common orthogonal for $\gamma_i$ and $\gamma_{i+1}$. 
Then,
taking signed distances (along $\gamma_i$ and $\gamma_{i+1}$), 
we have
$$d(x_i, y_i) = d(\eta, y_i) - d(\eta, x_i),$$ 
$$d(x_{i+1}, y_{i+1}) =  d(\eta, y_{i+1}) - d(\eta, x_{i+1}),$$ and 
$$d(\eta, x_i) = d(\eta, x_{i+1}).$$
Moreover,
$$|d(\eta,y_{i+1}) - d(\eta,y_i)| < d(y_i, y_{i+1})$$
by Lemma \ref{lem:app:triangle}. 
The claim follows.

The Lemma follows by summation over $i$.  
\end{proof}

Now suppose that $(\gamma_i)_{i=1}^n$ is a linear sequence of geodesics in $\H^2$.
We let $\eta_i$ be the common orthogonal to $\gamma_i$ and $\gamma_{i+1}$ (oriented from $\gamma_i$ to $\gamma_{i+1}$),
and we let $u_i$ be the signed complex distance from $\gamma_i$ to $\gamma_{i+1}$ (it will in fact be real and positive),
and $v_i$ be the signed complex distance (along $\gamma_i$---so $v_i$ is real) from $\eta_{i-1}$ to $\eta_i$.
\begin{enew}
We have the following:
\begin{lemma} \label{marching bound}
Suppose $(\gamma_i)_{i=0}^{n+1}$ is a linear sequence, with $(u_i)$ and $(v_i)$ defined as above.
Let $D = d(\gamma_0, \gamma_{n+1})$. 
Suppose that $u_0, u_n \le 1$. 
Then 
\begin{equation}
\left| \sum_{i=1}^n v_i \right| \le D + 2 \log D - \log u_0 - \log u_n + 3.
\end{equation}
\end{lemma}
\begin{proof}
Let $O$ be the common orthogeodesic between $\gamma_0$ and $\gamma_n$,
and let $y_i$ be the intersection of $O$ and $\gamma_i$.
We have,
taking signed distances,
\begin{equation*}
d(\eta_{i-1}, y_i) = d(\eta_i, y_i) + v_i
\end{equation*}
and, by Lemma \ref{lem:app:triangle},
\begin{equation*}
|d(\eta_i, y_i) - d(\eta_i, y_{i+1})| \le d(y_i, y_{i+1}),
\end{equation*}
and hence 
\begin{equation*}
|d(\eta_i, y_i) - d(\eta_{i+1}, y_{i+1}) - v_{i+1}| \le d(y_i, y_{i+1}). 
\end{equation*}
Therefore
\begin{equation} \label{change d}
\Bigl |d(\eta_0, y_0) - d(\eta_n, y_n) - \sum_{i=1}^n v_i\Bigr| \le \sum_{i=0}^{n-1} d(y_i, y_{i+1}) \le D.  
\end{equation}
Moreover,
by hyperbolic trigonometry,
we have
\begin{equation*}
\tanh d(y_0, y_1) = \tanh u_0 \cosh d(\eta_0, y_0)
\end{equation*}
which implies
\begin{equation} \label{y0 bound}
|d(\eta_0, y_0)| \le \log 4 + \log D - \log u_0
\end{equation}
when $u_0 \le 1$.
Likewise
\begin{equation} \label{yn bound}
|d(\eta_n, y_n)| <  \log 4 + \log D - \log u_n.
\end{equation}
Combining \eqref{change d}, \eqref{y0 bound}, and \eqref{yn bound}, we obtain the Lemma. 
\end{proof}
\end{enew}

We say a sequence $(\gamma_i)$ of oriented geodesics in $\H^3$ is \emph{semi-linear}
if each pair of consecutive geodesics are disjoint and have a common orthogonal.
Given such a sequence we can again define a homologous sequence of points, and also of frames. 

Now suppose that $(\gamma_i)_{i=1}^n$ is a linear sequence of geodesics in $\H^2$, 
and $(\gamma'_i)_{i=1}^n$ is a semi-linear sequence of geodesics in $\H^3$.
We define $\eta_i$, $u_i$, and $v_i$ for $(\gamma_i)$ as above, 
and we likewise define $\eta'_i$, $u'_i$, and $v'_i$ for $(\gamma'_i)$; we observe that the $u'_i$ and $v'_i$ may be nonreal complex numbers.
Moreover, 
we can define a unique map $e\from \F(\gamma_0) \cup \F(\gamma_n) \to \F(\gamma'_0) \cup \F(\gamma'_n)$ 
such that $e\from \F(\gamma_0) \to \F(\gamma'_0)$ and $e\from \F(\gamma_n) \to \F(\gamma'_n)$ are isometric embeddings,
and $e$ maps the foot of $\eta_0$ on $\gamma_0$ to its primed analog,
and likewise the foot of $\eta_{n-1}$ on $\gamma_n$.

We say that the two sequences are \emph{$R, B, \epsilon$-well-matched} if \removed{$R + B > n$ and }the following properties hold for each $i$:
\begin{enumerate}
\item
$v_i = 1$,
\item \label{delta v}
$|v'_i - v_i| < \epsilon/R$,
\item
$B^{-1} < u_i e^{R/2} < B$, and
\item \label{delta u}
$|u'_i - u_i| < \epsilon |u_i|.$
\end{enumerate}

\begin{theorem} \label{thm:app:linear-local-to-global}
For all $B, D$ there exists a $K$, $\epsilon_0$, $R_0$ such that for all $R > R_0$ and $\epsilon < \epsilon_0$:
Suppose that $\gamma_i$ and $\gamma'_i$ are $R, B, \epsilon$ well-matched.
Then the map $e$ has $K\epsilon$-bounded distortion at distance $D$.
\end{theorem}

\begin{proof}
Suppose that $x \in \F(\gamma_0)$ and $y \in \F(\gamma_n)$ such that $d(x, y) < D$. Then in particular $d(\gamma_0, \gamma_n) \le D$,
so we can apply Lemma \ref{marching bound} to obtain
\begin{equation} \label{n bound}
n = \sum v_i < D + 2\log D + R + 2 \log B + 3 < R + C_1(B, D) < 2R
\end{equation}
when $R > C_1(B, D)$. 

We let $(x_i)$ be the homologous sequence of frames for $(\gamma_i)$ starting with $x_0 := x$,
and let $(x'_i)$ be the same on $\gamma'_i$, starting with $x'_0 := x' := e(x)$. 

Now we let $a_i = u_i$ and define $b_i$ such that $\foot_{\gamma_{i+1}} \gamma_i =  x_i Y(b_i)$.
We observe that $x_{i+1} = x_iY(b_i)X(a_i)Y(-b_i)$ and $b_{i+1} = b_i + v_i$,
and likewise with $b_i, x_i, $ etc. replaced with $b'_i, x'_i$ etc.
%(We've probably defined a few too many things). 
So we have $b_0 = b'_0$ by the definition of $e$,
and 
$$
b_i = b_0 + \sum_{j=0}^{i-1} v_i,
$$
and likewise for $b'_i$ and $v'_j$,
and hence,
by \eqref{n bound},
$$
|b'_i - b_i| = \left| \sum_{j=0}^{i-1} v'_i - v_i \right| < \frac{n\epsilon}R < 2\epsilon
$$
by Condition \ref{delta v} of well-matched\removed{, and taking $R_0 > B$}. 
So we have verified Condition \new{\ref{delta b} } of Theorem \ref{thm:app:main-matrix-estimate}.
Condition \new{\ref{delta a} } of Theorem \ref{thm:app:main-matrix-estimate} also follows from Condition \ref{delta u} above. 

We still need to control $|a_i| e^{|b_i|}$.
In our setting all the $a_i's$ are about $e^{-R/2}$ in size,
and $b_i = b_0  + i$,
so we really just need to control the largest $|a_i| e^{|b_i|}$ (which is actually either the first or the last) 
in order to control the sum of all of them. 
We have $a_0 e^{|b_0|} < K(D)$ when $d(x_0, \gamma_1) < D$, as it must be. 
In particular we have $|b_0| < B' + R/2$. 

We also want a similar bound for $|b_n|$.
By Lemma \ref{lem:app:homologous}, 
we have $d(x_n, y) \le d(x, y) < D$. 
Moreover,
as in the previous paragraph,
the distance from $\foot_{\gamma_{n-1}} \gamma_n$ to $y$ is bounded by $\log(K(D) d(\gamma_{n-1}, y) / a_{n-1})$,
which is in turn at most $K'(D) + R/2$. 
Thus we obtain our bound for $|b_n|$. 
\end{proof}

\subsection{The perfect model and \texorpdfstring{$M, \epsilon$}{M, epsilon}-compliance} \label{app:perfect-model}
\begin{enew}
In this {subsection }we construct the \emph{perfect model} $\hat \A$ for a good assembly $\A$,  
and map $e\from \df \hat \A \to \df \A$ with bounds on distortion and certain other forms of geometric control. 
We rely in this section on Theorem \ref{thm:app:omnibus}, proven in Section \ref{subsec:hw}, 
which functions as an ``appendix to the appendix". 

Let $Q$ be a good component (pants or hamster wheel), and let $\hat Q$ be its perfect model. 
Recall the definitions of short and medium orthogeodesics in Sections \ref{components:good-pants} and \ref{goodHW}. 
We say that a good component $Q$ is $\epsilon$-\emph{compliant} 
if for every short orthogeodesic $\eta$ of $Q$ and corresponding orthogeodesic $\hat \eta$ in $\hat Q$,  
 \begin{enumerate}
\item \label{close length}
we have $| l(\eta) - l(\hat \eta) | < \epsilon l(\hat \eta)$, and
\item \label{near formal}
each foot of $\eta$ lies within $\epsilon/R$ of the corresponding formal foot of $Q$. 
\end{enumerate}
\begin{lemma} \label{compliant component}
There is universal constant $C$ such that every $\epsilon$-good component is $C\epsilon$-compliant. 
\end{lemma}
\begin{proof}
Part \ref{near formal} for the Lemma is trivial for pants, and Part \ref{close length} for pants appears as Equation (8) in \cite{KM:Immersing}.
Part \ref{close length} in the case of a hamster wheel is Part \ref{req:length-short} of Theorem \ref{thm:app:omnibus},
and Part \ref{near formal} is Part \ref{req:feet:formal} of Theorem \ref{thm:app:omnibus}.
\end{proof}
%Let $Q$ be a good component and $\hat Q$ the corresponding perfect one. 

We say that $e\from \df  \hat Q \to \df  Q$ is $M, \epsilon$-compliant
if the following \new{four }properties hold:
\begin{enumerate}
\item
The map from $\d \hat Q$ to $\d  Q$ induced by $e$ is the restriction (to $\d \hat Q$) of an orientation-preserving homeomorphism from $\hat Q$ to $Q$. 
\item \label{control:affine}
The induced map is affine (linear) on each component of $\d \hat Q$, and maps each component $\df Q$ to (the frames determined by) a slow and constant turning vector field on $\d Q$.  
\item
$e$ maps each formal foot of $\hat Q$ to the corresponding formal foot of $Q$. 
\item
$e$ is $\epsilon$-distorted up to distance $M$. 
%\item 
%For all distinct ordered pairs $\alpha$, $\beta$ of boundary geodesics in $Q$ with $d(\alpha, \beta) < M$,
%we have $d(e(\foot_\alpha(\beta)), \foot_{e(\alpha)} e_\beta) < \epsilon$. }
\end{enumerate}

\begin{theorem} \label{thm:app:compliant-map-local}
For all $M$ there exists $K, R_0$ such that for $R > R_0$:
Let $Q$ be an $(R, \epsilon)$-good component and $\hat Q$ its formal model.
Then there is an $M, K\epsilon$-compliant map $\e\from \hat Q \to Q$. 
\end{theorem}
%\ann{Does $K$ depend on $M$? Maybe only $R_0$ does.}
\begin{proof}
If $Q$ is a pants, then there is a unique map  $e$ from $\df \hat Q$ to $\df Q$
that satisfies Conditions 1--3 of $M, \epsilon$-compliance. 
%Relatively simple estimates in hyperbolic geometry show that it is then $R, K_0\epsilon$-compliant, for a universal $K_0$, when $R \ge 1$. 
To show that $e$ is $\epsilon$-distorted up to distance $M$, suppose that $n_0, n_1 \in \df \hat Q$, and let $x_0, x_1 \in \d \hat Q$ be the base points of $n_0, n_1$ respectively. If $x_0$ and $x_1$ lie on the same cuff of $\hat Q$, the bound on $d(n_0 \to n_1, e(n_0) \to e(n_1))$ follows from Condition 2 and the $\epsilon$-goodness of $Q$ (as long as $M < R$, which we can assume).  If $x_0$ and $x_1$ lie on different cuffs, say $\gamma_0$ and $\gamma_1$, we can apply Theorem \ref{thm:app:linear-local-to-global} (with $n=1$ in the well-matched sequence) to the lifts these two adjacent geodesics in the universal cover, to obtain the desired distortion bound for a \emph{slightly different} $e$, where the frame bundles for the $\hat \gamma_i$ are mapped to those for the $\gamma_i$ by parallel translation. But the two $e$'s are $C\epsilon$-related (within $R$ of the short orthogeodesic), so we obtain the desired bound. 

If $Q$ is a hamster wheel, 
we choose an inner cuff $\alpha$ of $\hat Q$ arbitrarily.
There is then a unique $e\from \df \hat Q \to \df Q$ 
satisfying Conditions 1--3 of $M, \epsilon$-compliance that maps the feet (on the outer curves) of the medium orthogeodesics from $\alpha$ to the two outer curves
to the corresponding feet on $Q$. 

We claim that this $e$ is $C\epsilon$-distorted to distance $M$.

To prove this claim, we first consider the case when $n_0, n_1 \in \df \hat Q$ are on same cuff of $\hat Q$ or adjacent inner cuffs of $\hat Q$. Then we can proceed much as the same as in the case of a pants,
using Part \ref{req:feet:formal} of Theorem \ref{thm:app:omnibus} to compare the two $e$'s. 

Let us now bound $d(n_0 \to n_1, e(n_0) \to e(n_1))$ when $n_0$ and $n_1$ are the two feet of a medium orthogeodesic $\beta$ on $\hat Q$. This distance is bounded by $C\epsilon$ (for a universal C) because we control both the complex length (in $Q$) of $\beta$ (Part \ref{req:length-medium} of Theorem \ref{thm:app:omnibus}), and the complex distance between $e(n_i)$ and the foot of the corresponding orthogeodesic on $Q$. The control of this latter complex distance follows from Part \ref{req:feet:corr} of Theorem \ref{thm:app:omnibus}, applied to $\alpha$ (which we chose to define $e$) and $\beta$.

Finally, suppose that $n_0, n_1$ are neither on the same cuff or on adjacent inner cuffs. 
Let $x_0, x_1$ be their base points in $\d \hat Q$;
we can assume that $d(x_0, x_1) < R/2$,
Then we can get from $x_0$ to $x_1$ through a chain of points $y_0 = x_0, \ldots, y_n = x_1$ on $\d \hat Q$,
with $n \le 5$,
such that each consecutive pair of $y$'s either lie on the same cuff of $\d \hat Q$,
or are the two endpoints of a  medium orthogeodesic of $\hat Q$. 
Letting $u_k$ be the frame in $\df \hat Q$ with base point $y_k$, 
we find that $d(u_i \to u_{i+1}, e(u_i) \to e(u_{i+1})) < C\epsilon$ (for a universal $C$) by 
Condition 2 of $e$ (when $u_i, u_{i+1}$ lie on the same cuff), and the paragraph above (in the latter case for the pair $u_i, u_{i+1}$). Moreover, $d(y_i, y_{i+1})$ is bounded in terms of $M$. 
Applying Lemma \ref{lem:app:concat}, we obtain the desired result. 
\end{proof}

\alert{Do we need property \ref{control:affine}?}

Now suppose that $\A$ is a good assembly and $\hat \A$ is a perfect one. 
We say that $e\from \df \hat A \to \df \A$ is $M, \epsilon$-compliant if the following properties hold:
\begin{enumerate}
\item
For each component $\hat Q$ of $\hat \A$, 
there is a corresponding component $Q$ of $\A$ such that 
$e|_{\df \hat Q}$ is an $M, \epsilon$-compliant map from $\df \hat Q$ to $\df Q$. 
\item
If $\hat \gamma$ is a gluing boundary of $\hat \A$,
so $\hat \gamma \in \d \new{\hat Q_0}$ and $\hat \gamma \in \d \hat Q_1$,
and $n_0$ and $n_1$ are unit normal vectors to $\hat\gamma$ pointing towards $\hat Q_0$ and $\hat Q_1$ respectively,
and sharing the same base point on $\hat \gamma$,
then
$$
d(n_0 \to n_1, e(n_0) \to e(n_1)) < \epsilon.
$$
\end{enumerate} 

%\ann{Do we want to remove or edit this?}
%The details of these definitions are not so significant;
%what matters is that the compliance of a map $e\from \df \hat A \to \df A$ just depends on $e$ restricted to each component, 
%and on $e$ restricted to the frame bundle of each gluing geodesic in $\hat A$.
%We might say that $\epsilon$-compliance is a thus a \emph{hyperlocal} property.

\begin{theorem} \label{thm:e-compliant}
For all $M, \epsilon$ we can find $K, R_0$ such that for all $R > R_0$:
For any $(R, \epsilon)$-good assembly $\A$ we can find a perfect assembly $\hat \A$ and an $M, K\epsilon$-compliant map $e\from \df \hat A \to \df A$. 
\end{theorem}
%The proof of this Theorem requires the $M, \epsilon$-compliance of a good hamster wheel.
%Because this is technical to state and 
%which is defined and proven as Theorem \ref{} in Section \ref{subsec:hw}.

\begin{proof}
For each good component $Q$,
we construct an $M, K_{\ref{thm:app:compliant-map-local}}(M) \epsilon$-compliant map $e\from \df \hat A \to \df A$ as in the proof of Theorem \ref{thm:app:compliant-map-local}.
Now, 
given the whole assembly $\A$,
we assemble the perfect models for the components as follows:
two components with formal feet are joined with a shear by 1, 
and in all other cases there is a unique way to join the perfect models at each common curve $\hat\gamma$ 
 so that the (frames determined by) the two unit normals at each point of $\hat\gamma$ 
 map to frames with the same base point in the corresponding good curve $\gamma$. 
Condition 2 of $M, K\epsilon$-compliance for our global map $e$ then follows from the $R, \epsilon$-goodness of $\A$ 
(and Part \ref{req:feet:slow} of Theorem \ref{thm:app:omnibus}), 
because $e(n_0)$ and $e(n_1)$ always have base points within $\epsilon/R$ of each other, 
and always have a bending of at most $\epsilon$.
\end{proof}
\end{enew}

\subsection{Bounded distortion for good assemblies} \label{app:bounded-distortion}
%\ann{Again we should have one place where we define short, medium, and long orthogeodesics}
We can now state our theorem for good assemblies:
\begin{theorem} \label{thm:app:control}
For all $D$ there exists $C, R_0$ for all $\epsilon$ and $R > R_0$:
Let $\A$ be a $(R, \epsilon)$-good assembly. 
Then there is a perfect assembly $\hat \A$ and a map $e\from \df \hat \A \to \df \A$
that is $C\epsilon$-distorted at distance $D$.
\end{theorem}
%\ann{Old comment: A related theorem statement that appeared only in the technical outline is commented out in the tex here.}
%\begin{theorem} \label{thm:app:compliant-distortion}
%For all $D$ there exists $M, K$:
%If $e\from \df\hat \A \to \df\A$ is an $M, \epsilon$-compliant map from a $R$-perfect assembly to an $(R, \epsilon)$-good one,
%then $\tilde e\from \tdf \hat \A \to \tdf \A$ has $K\epsilon$-bounded distortion at distance $D$.
%\end{theorem}

\begin{proof}[Proof of Theorem \ref{thm:app:control}]
Let $\A$ be as given in the Theorem. 
By Theorem \ref{thm:e-compliant},
we have a perfect model $\hat \A$ and an $D, C\epsilon$-compliant map $e\from \df \hat \A \to \df \A$
(in this proof we will let all our constants $C$ depend on $D$). 
We work on the universal cover of the good assembly $\hat \A$, 
and prove that our map $e$ (lifted to the universal cover), 
has bounded distortion.
Accordingly suppose that we have two points $\new{p}$ and $\new{q}$, with $d(\new{p}, \new{q}) < D$, each lying on a boundary curve for $\hat \A$ in the universal cover.
We suppose that $\new{p}$ lies on $\gamma_0$, and $\new{q}$ on $\gamma_n$,
where $\gamma_1, \ldots, \gamma_{n-1}$ are the boundary curves for $\hat \A$, 
in sequence,
separating $\gamma_0$ and $\gamma_n$. 
We define $(\eta_i)$, $(u_i)$, and $(v_i)$ as in Section \ref{subsec:linearseq}.

\begin{figure}[ht!]
\makebox[\textwidth][c]{
\includegraphics[width=1.15\linewidth]{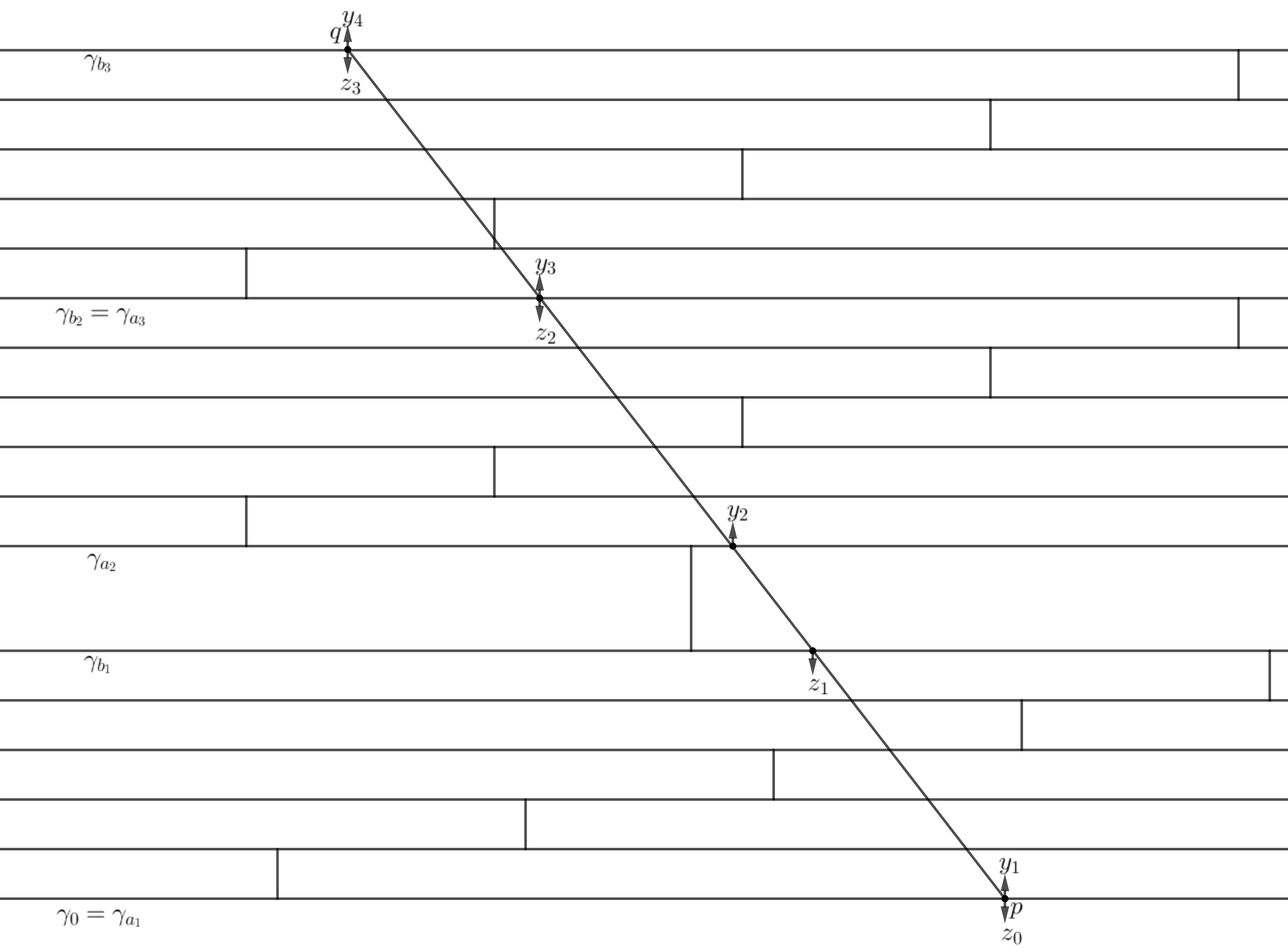}}
\caption{Going from $\new{p}$ to $\new{q}$. (In this picture $z_2 = -y_3$ because $b_2=a_3$; this is not always the case.)}
\label{F:gammas}
\end{figure}

\begin{enew}
Now if $\eta_i$ is a short orthogeodesic---between two cuffs of a pants or two inner cuffs of a hamster wheel---we have $C_1^{-1} e^{-R/2} < u_i < C_1 e^{-R/2}$ for a universal $C_1$.
Moreover, 
if $\eta_{i-1}$ and $\eta_i$ are both short orthogeodesics, 
then $v_i \equiv 1 \pmod R$,
and hence, 
by Lemma \ref{marching bound}  (applied to $(\gamma_j)_{j=i-1}^{i+1})$,
there is a universal constant $C_2$ such that 
if $d(\gamma_{i-1}, \gamma_{i+1}) < C_2$,
then $v_i = 1$. 
%Finally,
%we must have $v_i \le R + C_3$ for all $i$, again by Lemma \ref{marching bound}. 

We let a \emph{run} be an interval $\Z \cap [a, b]$ such that $\eta_i$ is a short orthogeodesic for $a \le i < b$,
and $v_i =1$ for $a < i < b$. We allow the case $a = b$, which is a trivial run for all $a$.
We can then find $a_1, \ldots, a_k$ and $b_1, \ldots, b_k$,
with $a_j \le b_j \le a_{j+1}$
where each $[a_j, b_j]$ is a \emph{maximal} run,
and these intervals cover $\Z \cap [0, n]$.
Then we must have $a_{j+1} \le b_j + 1$,
and $k$ is bounded (in terms of $D$) because $d(\gamma_{b_j - 1}, \gamma_{a_{i+1} + 1})$ is universally bounded below. 

We observe that $(\gamma_i)_{i = a_j}^{b_j}$ and $(\gamma'_i)_{i = a_j}^{b_j}$ are $R,C, \epsilon$ well-matched (in the sense of Section \ref{subsec:linearseq});
Conditions 1 and 3 hold by the definition of a run, and Conditions 2 and 4 follow from Lemma \ref{compliant component}. 

%To summarize:
%\begin{enumerate}
%\item
%$a_i \le b_i$ and $b_i \le a_{i+1} \le b_i + 1$,
%\item
%$(\gamma_i)_{i = a_j}^{b_j}$ and $(\gamma'_i)_{i = a_j}^{b_j}$ are $R,C, \epsilon$ well-matched (in the sense of Section \ref{}) 
%for some universal $C$. 
%%\item
%%$C_0 e^{-R/2} < d(\gamma_j, \gamma_{j+1}) < C e^{-R/2}$ when $a_i  \le j < j+1 \le b_i$,
%%\item
%%$\foot_{\gamma_{j+1}}\gamma_j - \foot_{\gamma_{j-1}}\gamma_j = 1$ when $a_i < j <  b_i$,
%\item
%$\foot_{\gamma_{b_j +1}}\gamma_{b_j} - \foot_{\gamma_{b_j-1}} \gamma_{b_j} = -R$ when $a_{i+1} = b_i$,
%\item
%$1/100 < d(\gamma_{b_j}, \gamma_{a_{i + 1}}) < D$ when $a_{i+1} = b_i + 1$.
%\end{enumerate}

%We then place $y_i$ on $\gamma_{a_j}$ and $z_i$ on $\gamma_{b_j}$ such that
%$$\foot_{\gamma_{a_j +1}} \gamma_{a_j} - y_i = (b_j - a_j -1)/2$$
%and 
%$$z_i - \foot_{\gamma_{b_j -1}} \gamma_{b_j}  = (b_j - a_j -1)/2.$$
We let $\alpha$ be the segment connecting $\new{p}$ and $\new{q}$.
We let $y_j$ be the unit normal vector to $\gamma_{a_j}$ (pointing forward to $\gamma_{b_j}$) at $\alpha \cap \gamma_{a_j}$,
and we let $z_j$ be the unit normal vector to $\gamma_{b_j}$ (pointing back to $\gamma_{a_j}$) at $\alpha \cap \gamma_{b_j}$.
See Figure \ref{F:gammas}.

We observe that $d(\alpha \cap \gamma_i, \eta_{i-1}), d(\alpha \cap \gamma_i, \eta_{i}) < R + C$ for all $i$. 

We make the following claims:
\begin{enumerate}
\item \label{y to z}
$d(y_j \to z_j, e(y_j) \to e(z_j)) < \epsilon$,
\item
$d(z_j\to y_{i+1}, e(z_j) \to e(y_{i+1})) < \epsilon$.
\end{enumerate}

%We observe that $\epsilon$-compliance implies that the sequence $(\gamma_j)_{j=a_i}^{b_j}$ is $R, B, \epsilon$-well-matched.%
%\annf{Or something like that}
Claim \ref{y to z} holds as a result of Theorem \ref{thm:app:linear-local-to-global} (and Lemma \ref{lem:app:concat}), 
along with the observation (as in the proof of Theorem \ref{thm:app:compliant-map-local}) that the $e$ for Theorem \ref{thm:app:linear-local-to-global} is $\epsilon$-related to the official one (constructed at the beginning of the proof of this Theorem) on the relevant parts of $\F(\gamma_{a_j})$ and $\F(\gamma_{b_j})$.

%The second claim follows from Theorem \ref{generate good} and the properties of $e$,
%when $a_{i+1} = b_i + 1$, and follows easily from the properties of $e$ when $a_{i+1} = b_i$. 
The second claim follows immediately from Lemma \ref{lem:app:concat} and the $D, C\epsilon$-compliance of $e$.

It then follows from Lemma \ref{lem:app:concat} that $d(z_0 \to y_{k+1}, e(z_0) \to e(y_{k+1})) < C\epsilon$.
\end{enew}
\end{proof}

\subsection{Boundary values of near isometries} \label{app:boundary}
We say that $X' \subset X$ is \emph{$A$-dense} in a metric space $X$ if $\nbhd_A(X') = X$.
(So $X'$ is dense in $X$ if it is $A$-dense for all $A > 0$.)

%Recall that $\F(\H^2)$ denotes the frame bundle for $\H^2$, and likewise for $\F(\H^3)$.
We think of $\H^2$ as a subset of $\H^3$, and hence $\F(\H^2)$ as a subset of $\F(\H^3)$. 
%We have an obvious embedding $e_{23}\from \F(\H^2) \to \F(\H^3)$. 

%If $u, v\in \F(\H^3)$ we let $u \to v$ denote the $g \in \PSL_2(\C)$ such that $u \cdot g = v$, so $g$ is the instruction that takes $u$ to $v$. 

\begin{theorem} \label{thm:app:dense-distorted}
For all $A$ there exists $B\new{, K}$ for all $\delta$ there exists $\epsilon$:
Suppose that $U \subset \F(\H^2)$ is $A$-dense and $e\from U \to \F(\H^3)$ has $\epsilon$-bounded distortion to distance $B$. 
Then $e$ is a $\new{K}$-quasi-isometric embedding, and $e$ extends to $\hat e\from \d \H^2\to \d \H^3$ to be $1 + \delta$-quasi-symmetric. 
\end{theorem}

This theorem will follow from the following more general theorems:
\begin{theorem} \label{thm:hyperbolic-local-global}
For all $K, \delta$ there exists $K', D$:
Suppose $X$ is a path metric space and $Y$ is $\delta$-hyperbolic, and $f\from X \to Y$ is such that 
$$
K^{-1} d(x, x') - K < d(f(x), f(x')\color{kwred})\color{black} < K d(x, x') + K
$$
whenever $d(x, x') < D$. 
Then $f$ is a $K'$-quasi-isometric embedding. 
\end{theorem}

\begin{proof}
%By Misha K's reference,
%a map is a $K$-quasi-isometry if and only if it maps geodesics to $K'$-quasi-geodesics. 
%(The reader should suitable interpret "if and only if" as to statements, with different order of dependence for $K$ and $K'$.)
Because $X$ is a path metric space,
it is enough to find $K'$ and $D$ such that every $D$-locally $K$-quasigeodesic in $Y$ is $K'$-geodesic;
this follows from  \cite[Ch. 3, Theorem 1.4]{CDP} (see also \cite[p. 407]{BH}).
%For all $K$ there exists $K', D$ such that every map of $\R$ (or $\Z$) that is a $K$-quasi-geodesic on every interval of size $D$ is $K'$-quasigeodesic.
%The Theorem follows. 
\end{proof}

\begin{theorem} \label{thm:qi-d}
Let $X$ and $Y$ be Gromov hyperbolic, and
let $f\from X\to Y$ be a quasi-isometric embedding.
Then $f$ extends continuously to an embedding $\hat f\from \d X \to \d Y$.
{Moreover $\hat f$ depends continuously on $f$ with the uniform topology on $\hat f$ and the local uniform topology on $f$.}
\end{theorem}

\begin{proof}
Let us first show $\hat f$ is well-defined. 
If two geodesic rays in $X$ stay at bounded distance,
then the image of each is at bounded distance from a geodesic ray in $Y$ (by the Morse lemma),
and those two geodesic rays in $Y$ stay at bounded distance, 
by thin triangles.

It is easy to check that $\hat f$ is an embedding and that it is a continuous extension.

If $f_0$ and $f_1$ are at bounded distances for large scales, 
then these quasi-geodesics (and there geodesic fellow-travelers)
stay at bounded distance for a long time,
and hence are nearby in the boundary. 
\end{proof}

\begin{proof}[Proof of Theorem \ref{thm:app:dense-distorted}]
We take $B = 2A + 1$.
Then whenever $U \subset \F(\H^2)$ is $A$-dense 
and $x, y \in U$, there is a sequence $z_0 = x, z_1, \ldots z_k = y$ in $U$, 
with $k < d(x, y) + 1$,
for which $d(z_i, z_{i+1}) < B$ for $i = 0\ldots k_1$.

Now suppose that we have a $U$ and an $e$ which has $\epsilon$-bounded distortion to distance $B$. 
Then for any $B'$ and $x, y \in U$ with $d(x, y) < B'$,
we have a sequence $z_i$ defined as above,
and hence
$$
d(e(x) \to e(y), x \to y) < F(\epsilon, B'). 
$$
where $F(\epsilon, B')$ is small when $\epsilon$ is small and $B'$ is bounded.

We then take $K$ in Theorem \ref{thm:hyperbolic-local-global} to be 2, 
and let $B'$ be the $D$ in that Theorem;
when $\epsilon$ is sufficiently small,
we can then apply Theorem \ref{thm:hyperbolic-local-global}
to obtain that $e$ is globally $K'$ quasi-isometric. 

By Theorem \ref{thm:qi-d},
$e$ extends continuously to $\hat e\from \d H^2 \to \d H^3$. 
We must show that $\hat e$ is $\delta$-quasisymmetric. 
This is the same as showing that all standard quadruples are at most $\delta$-distorted;
because we can move $e$ by M\" obius tranformations in domain and range,
it is enough to show that $|e(0)| < \delta'$,
under the assumption that $e(-1) = -1$, $e(1) = 1$, and $e(\infty) = \infty$. 
Suppose, 
to the contrary,
that there were no $\epsilon$ that insured that $|e(0)| < \delta$.
Then we can take a sequence of maps $e_n$ defined on a sequence of $A$-dense sets $U_n$ 
with $1/n$ bounded distortion to distance $B$,
and for which $|e(0)| \ge \delta$.
But then we can take a limit $e_\infty$ of these $e_n$ and find that $\hat e_\infty(0) = 0$
because $\hat e_\infty(1) = 1$ (and same for $-1$ and $\infty$) and $\hat e_\infty$ is a M\"obius tranformation.
This is a contradiction.
\end{proof}

\subsection{Conclusions}\label{subsec:conclusions}
From Theorems \ref{thm:app:control} and \ref{thm:app:dense-distorted} we can derive the following corollary:
\begin{theorem} \label{thm:app:good-quasi}
There exists $R_0$ for all $\epsilon$ there exists $\delta$ for all $R > R_0$:
Let $\A$ be a $(R, \delta)$-good assembly. 
Then there is a perfect assembly $\hat \A$ and an map $e\from \df \hat A \to \df \A$ 
such that $e$ extends to $\hat e\from \d \H^2\to \d \H^3$ to be $(1 + \epsilon)$-quasi-symmetric. 
\end{theorem}

\begin{proof}
%\ann{Maybe this proof should be removed? Or dramatically shortened?}
It is easy to check that there is an $a$ such that every perfect assembly is $a$-dense; 
we let the $A$ in Theorem \ref{thm:app:dense-distorted} be this $a$, 
and we let the $D$ in Theorem \ref{thm:app:control} be the resulting $B$ in Theorem \ref{thm:app:dense-distorted}. 
We let $\delta$ in Theorem \ref{thm:app:dense-distorted} be the given $\epsilon$,  
and then we let $\delta$ for the Theorem be 
 the $\epsilon$ from Theorem \ref{thm:app:dense-distorted}
 divided by the $C$ from Theorem \ref{thm:app:control}. 
%and then we let $\epsilon$ in Theorem \ref{thm:app:control} be the resulting $\epsilon$ in Theorem \ref{thm:app:dense-distorted}.
%We let $\delta$ be the resulting $\delta$ (in Theorem \ref{thm:app:control}). 
The Theorem follows. 
\end{proof}

%\begin{enew}
%We also have the following theorem:
%\begin{theorem} \label{complex interpolation}
%For all $\delta > 0$ there exists $\epsilon_1 > 0 , R_0, M, C_1, C_2$ such that for all $R > R_0$, $\epsilon < \epsilon_1$: 
%Let $\A$ be an $(R, \epsilon)$-good assembly.
%Then there is a perfect assembly $\ap$ and
%a holomorphic map $h$ from the unit disk $\Delta$ to $(R, C_2)$-good assemblies
% such that $h(0) = \hat \A$ and $h(C_1\epsilon) = \A$,
% and an $M,\delta$-compliant map $e\from h(0) \to h(z)$ for each $z \in \Delta$. 
%\end{theorem}
%
%\begin{proof}
%We first observe that there is a unique map $h$ from $\Delta$ to $(R, C_1\epsilon)$-good assemblies,
%with the given values at 0 and 1/2,
%such that the lengths of the cuffs,
%and the complex distances along cuffs between pairs of rungs and between rungs and formal feet on cuffs of pants and inner cuffs of hamster wheels, 
%vary complex linearly in $\Delta$.
%
%Then we observe that this implies the whole assembly (and the associated representation) varies holomorphically. 
%\end{proof}
%
%We then have the following corollary:
%\begin{theorem}
%[Good is $C\epsilon$ close to perfect model. ]
%\end{theorem}
%
%\begin{proof}
%We apply the Schwarz lemma to the map produced in Theorem \ref{complex interpolation}
%\end{proof}
%\end{enew}

The following theorem follows readily from the machinery in \cite[Ch. 4]{ahlfors}.
\begin{theorem} \label{thm:app:qsqc}
Let $f\from S^1 \to \chat$ be $M$-quasi-symmetric, 
in the sense that $f$ sends standard triples to $M$-semi-standard triples. 
Then $f$ extends to a $K$-quasiconformal map from $\chat$ to $\chat$,
where $K$ depends only on $M$
and $K \to 1$ as $M \to 1$.
Moreover if the original map conjugates a group of \Mobius\ transfomations to another such group, 
then the extension does as well.
\end{theorem}

We can complete the goal of this Appendix:

\begin{proof}[Proof of Theorem \ref{good close to Fuchsian}]
This is an immediate corollary of Theorems \ref{thm:app:good-quasi} and \ref{thm:app:qsqc}.\end{proof}

We can make a further observation that is not needed to prove Theorem \ref{good close to Fuchsian}.
The following is a consequence of the Douady-Earle extension applied to $\H^3$.
(See \cite{douady1986conformally} and \cite{tam1998quasiconformal}).
\begin{theorem} \label{thm:app:2to3}
Let $\rho$ and $\rho'$ be groups of isometries of $\H^3$ (and its boundary $S^2$),
and suppose that $f\from S^2\to S^2$ is $K$-quasiconformal and conjugates $\rho$ to $\rho'$.
Then we can extend $f$ to a quasi-isometry $\hat f\from \H^3 \to \H^3$ with the same property.
If $K$ is close to 1, 
then we can make $f$ be $C^\infty$, 
and $C^\infty$ close on ball of bounded (but large) size to an isometry. 
\end{theorem}

From Theorems \ref{thm:app:qsqc} and \ref{thm:app:2to3} it follows that we obtain a map from a closed hyperbolic surface to the given cusped hyperbolic 3-manifold that is everywhere locally $C^\infty$ close to being an isometric embedding. 

It also follows that this map is homotopic to a minimal surface, and hence that there are infinitely many minimal surfaces in our 3-manifold. 

\subsection{Good and perfect hamster wheels} \label{subsec:hw}
The object of this subsection is to provide some estimates on the geometry of good hamster wheels as defined in Section \ref{goodHW}. In particular, we prove Theorem \ref{thm:app:omnibus}, which is used to prove Theorems \ref{compliant component}, \ref{thm:app:compliant-map-local}, and \ref{thm:e-compliant} in Section \ref{app:perfect-model}.

Recall the definitions of short and medium orthogeodesics for a hamster wheel in Section \ref{goodHW}. 

Our theorem is as follows, stated for a good hamster wheel $H$ and its perfect model $H_R$. 
In this theorem, 
if $\alpha$ is an orthogeodesic for $H_R$, 
we let  $l_H(\alpha)$  and $l_{H_R}(\alpha)$ denote  the complex length of $\alpha$ in $H$ and $H_R$ respectively. 
\begin{theorem} \label{thm:app:omnibus}
There is are universal constants $C$, $R_0$ and $\epsilon_0$ such that for any $R, \epsilon$-good hamster wheel $H$, 
with $R > R_0$ and $\epsilon < \epsilon_0$:
\begin{enumerate}
\item \label{req:length-short} 
For all short orthogeodesics $\alpha$,
$$
\left|\frac{l_H(\alpha)}{l_{H_R}(\alpha)} - 1\right| < C\epsilon.
$$
\item \label{req:length-medium} 
For all medium orthogeodesics $\alpha$,
$$
|{l_H(\alpha)} - {l_{H_R}(\alpha)}| < C\epsilon.
$$
\item \label{req:feet:formal} \label{req:first}
Both feet of every short orthogeodesic lie within $C\epsilon/R$ of a formal foot.
\item \label{req:feet:slow}
The foot on the outer cuff of every medium orthogeodesic lies within $C\epsilon$ of the value of the slow and constant turning vector field at the associated point. 
\item \label{req:feet:corr} \label{req:last}
If $\alpha$ and $\beta$ are two short or medium geodesics with feet on a cuff $\gamma$ (in $H_R$),
then 
$$
|d_\gamma(\alpha, \beta) - d_{\gamma'}(\alpha', \beta')| < C\epsilon,
$$
where $\alpha'$, $\beta'$, and $\gamma'$ are the corresponding objects for $H$. 
\end{enumerate}
\end{theorem}

\begin{remark}
The feet described in Part \ref{req:feet:slow} of the Theorem will actually be within $C\epsilon/R$ of the constant turning vector fields. 
\end{remark}

\begin{figure}[ht!]
\makebox[\textwidth][c]{
\includegraphics[width=1.15\linewidth]{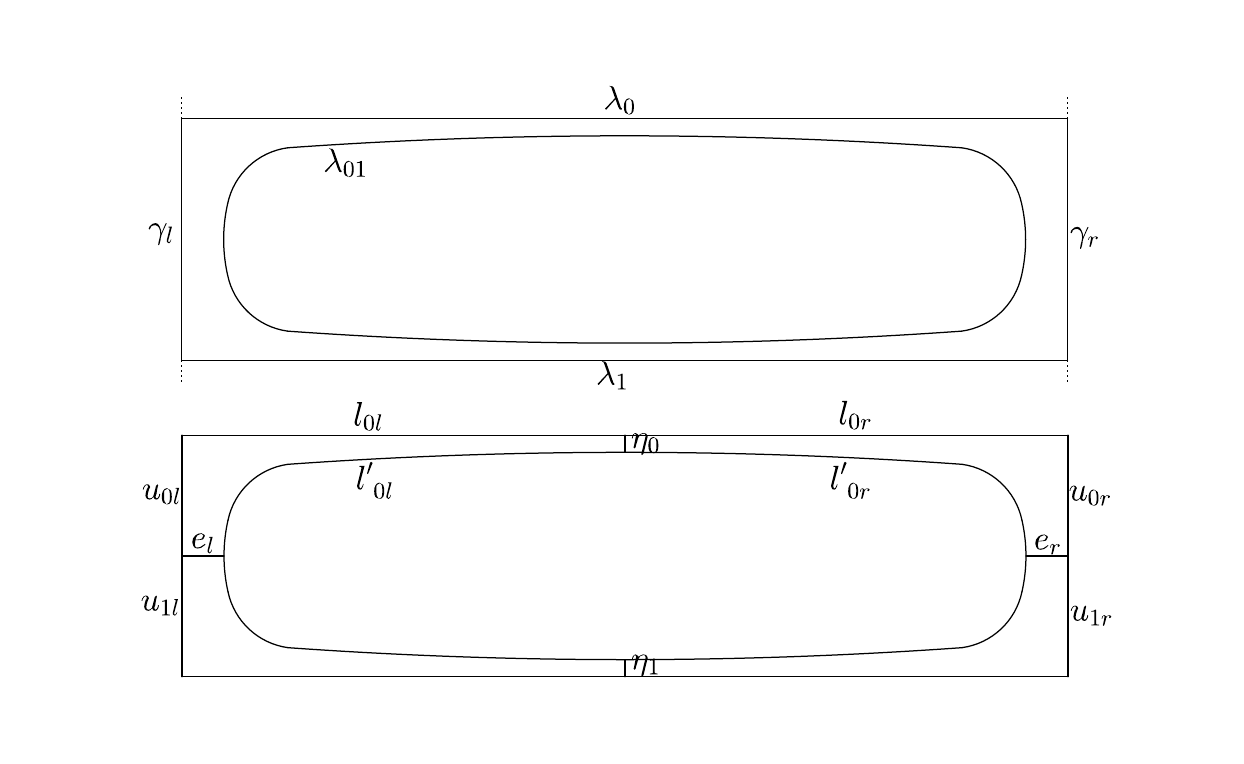}}
\caption{Top: Geodesics labelled by name. Bottom: Orthogeodesics labelled by complex length.}
\label{F:Spoke}
\end{figure}
\begin{comment}
        Before we begin the proof, we will need an explicit version of the Inverse Function Theorem:
        \begin{theorem}
        Suppose that $a \in \R^n$, and $f\from B_r(a) \to \R^n$ is $C^2$, 
        and $\norm{D^2f}_\infty < M_1$, and $\norm{Df(a)^{-1}} < M_2$, and $2r M_1 M_2 < 1$.
        Then $f^{-1}$ is defined and $2M_2$-Lipschitz on $B_{r/2M_2}(f(a))$. 
        \end{theorem}
\end{comment}

\begin{proof}%[Proof of Theorem \ref{generate good}]
The hamster wheel $H$ has two outer cuffs (``rims") $\gamma_r$ and $\gamma_l$, 
 and $R$ orthogeodesic connections (``rungs'') $\lambda_i$. 
 %, with feet at most $O(\epsilon/R)$ away from  $p_i$ and $q_i$ and meeting the constant turning vectors at angles of at most $O(\epsilon/R)$ and with length equal to $R + 2 \log \sinh 1 + O(\epsilon)$. 
We consider two consecutive such orthogeodesic connections in Figure \ref{F:Spoke},
 say $\lambda_0$ and $\lambda_1$.
Let $\lambda_{01}$ be the geodesic homotopic 
 to $\lambda_0$ plus $\lambda_1$ plus the segments on $\gamma_r$ and $\gamma_l$. 
We also consider the orthogeodesics
 from $\lambda_{01}$
 to the four segments $\gamma_l$, $\gamma_r$, $\lambda_0$, and $\lambda_1$,
 as in Figure \ref{F:Spoke}. 
Let the complex distance between pairs of geodesics be as in Figure \ref{F:Spoke},
 where we have labeled orthogeodesics between a pair of geodesics
  by the complex distance between that pair of geodesics.     
These four orthogeodesics cut each of the four segments into two subsegments,
 whose complex lengths we have labelled in Figure \ref{F:Spoke}.
 
We let $l_0 = l_{0l} + l_{0r}$ be the complex length of $\lambda_0$,
 and likewise define $l_1$. 
(By complex length of $\lambda_0$
 we mean the complex distance along $\lambda_0$
  between the geodesics $\gamma_l$ and $\gamma_r$
   that it was defined to connect). 
Similarly we let $l'_0 = l'_{0l} + l'_{0r}$ be the complex length of $\lambda'_0$
 (the portion of $\lambda_{01}$ going clockwise from $e_l$ to $e_r$),
 and likewise define $l'_1$.
We let $u_r = u_{0r} + u_{1r}$ be the complex distance between $\lambda_0$ and $\lambda_1$
 along $\gamma_r$, 
 and likewise define $u_l$. 
Because $H$ is an $\epsilon, R$-good hamster wheel, we have $u_r = 2 + O(\epsilon/R)$ (and likewise for $u_l$) 
and $l_0 = R - 2 \log \sinh 1 + O(\epsilon/R)$ (and likewise for $l_1$). 

By the Hyperbolic Cosine Law for Hexagons (Lemma \ref{hex trig}),
we have
\begin{equation}
\cosh u_{l0} = \coth l_0 \coth e_l + \cosh e_r \csch l_0 \csch e_l
\end{equation}
and likewise for $0$ replaced with 1, and/or $l$ interchanged with $r$. 
Moreover we have $u_{0l} + u_{1l} = u_l$, and likewise with $l$ replaced by $r$. 

Now imagine we are given $u_l$, $u_r$ each near 2, and $\coth l_0$ and $\csch l_0$ (and $\coth l_1$ and $\csch l_1$) thought of as \emph{independent variables} with the former near 1 and the latter near 0. 
Then $u_{0l}, u_{1l}, u_{0r}, u_{1r}$ and $e_l, e_r$ are determined implicitly by the Implicit Function Theorem, 
and satisfy
% (and basic estimates on $\coth$ and $\csch$) that for $R$ large,
\begin{equation}\label{est:ue}
 u_{0l}= 1 + O(\epsilon/R), e_l = \coth^{-1} \cosh 1 + O(\epsilon/R),
\end{equation}
and likewise for $u_{1l}, u_{0r}, u_{1r}$ and $e_r$. 

\begin{comment}
        First, we prove that $e_r$ is bounded. 
        We first observe that $u_{0l}$ and $u_{0r}$ have bounded real parts.
        It then follows from the Hyperbolic Cosine Law for Hexagons that 
        \ann{how do we know that $u_{0l}$ and $u_{0r}$ are bounded below?}
        \begin{align*}
        \cosh l_0 &= \frac {\cosh l'_0+ \cosh u_{0l}\cosh u_{0r}}{\sinh u_{0l} \sinh u_{0r}} \\
        &= \frac {e^{l'_0}} {2 \sinh u_{0l} \sinh u_{0r}}  + O(1) 
        \end{align*}
        and hence, when $l_0 = R + O(1)$ is large, 
        \begin{equation} \label{l diff}
        l_0 = l'_0 - \log(\sinh u_{0l} \sinh u_{0r}) + O(e^{-R}).
        \end{equation}
        In particular $l_0 - l'_0$ is bounded when $l_0$ is large. 
        From this it follows from the Sine Law for Right-Angled Hexagons
         that $\sinh e_r/\sinh u_{0l}$ has bounded absolute value,
         and hence that $e_r$ is bounded.
        Likewise $e_l$ is bounded. 
\end{comment}

Then it follows that $\eta_0 = O(e^{-R/2})$,
because $l_0, l'_0 > R - O(1)$, 
and the endpoints of the corresponding segments are at a bounded distance from each other (in pairs). 

% We must have $l_{0l} > R/2 - O(1)$ or $l_{0r} > R/2 - O(1)$ (or both); we assume the latter.
By Lemma \ref{hex trig}, we have 
\begin{align*}
\cosh l_{0r} % &= \frac{\cosh \eta_0 \cosh u_r + \cosh \eta_1}{\sinh \eta_0 \sinh u_r}\\
&= \frac{\cosh \eta_0 \coth u_r + \cosh \eta_1 \csch u_r}{\sinh \eta_0} \\
&= \frac{\coth 1 + \csch 1 + O(\epsilon/R)}{\eta_0 (1 + O(e^{-R}))}.
\end{align*}
We have also the general estimate 
$$
\cosh l_{0r} = (1 + O(e^{-2 l_{0r}})) e^{l_{0r}}/2.
$$
Therefore
\begin{equation*}
l_{0r} = \log 2(\coth 1 + \csch 1) - \log \eta_0 + O(\epsilon/R + e^{-2 l_{0r}}) 
\end{equation*}
and likewise for $l_{0l}$.
We can then conclude in sequence that $l_{0l}, l_{0r} > R/2 + O(1)$,
 that $e^{-2 l_{0r}} = O(\epsilon/R)$,
 that $|l_{0l} - l_{0r}| <   \epsilon/R$,
 and 
 \begin{equation} \label{est:l}
 l_{0l} = R/2 - \log \sinh 1 + O(\epsilon),
\end{equation}
and likewise for $l_{0r}$, $l_{1l}$, and $l_{1r}$.
Moreover 
\begin{equation} \label{est:eta}
\eta_0 = e^{-R/2 + C_1 + O(\epsilon/R)},
\end{equation}
where 
$$
C_1 = \log 2(\coth 1 + \csch 1) + \log \sinh 1.
$$
By Lemma \ref{penta trig},
we have 
$$
\cosh{l'_{0r}} = \sinh{l_{0r}} \sinh{u_{0r}};
$$
by our previous estimates,
this implies 
\begin{equation} \label{est:lprime}
l'_{0r} = R/2 + O(\epsilon/R).
\end{equation}
Equations \eqref{est:ue}, \eqref{est:l}, \eqref{est:eta}, and \eqref{est:lprime}
 show that the $u$'s, $e$'s, $l$'s, $l'$'s, and $\log \eta$'s
  differ by $O(\epsilon/R)$ from their perfect counterparts.
\begin{comment}         
         \ldots and then from Lemma \ref{penta trig} it follows that 
         $$
         \sinh e_r \,\sinh u_{0r} = 1 + C_0 e^{-R} + O(e^{-2R})
         $$
         where $C_0$ is a universal constant. 
        Applying this equation with $u_{0r}$ replaced by $u_{1r}$,
         we find that $|u_{0r} - u_{1r}| = O(e^{-R})$. 
        Therefore $u_{0r} = \frac 12 u_r + O(e^{-R}) = 1 + O(\epsilon/R)$.
        Now we can conclude from equation \eqref{l diff}  that 
        \begin{align*}
        l'_0 &= l_0  + 2\log \sinh 1 + O(\epsilon/R) \\
         &= R + O(\epsilon),
        \end{align*}
        and hence that 
        $$ 
        l' = 2R + O(\epsilon).
        $$
        Applying Lemma \ref{penta trig} to the upper right and upper left-hand pentagons,
        we obtain
        \[
         \frac {\coth u_{0l}}{\cosh l_{0l}} = \frac {\coth u_{0r}}{\cosh l_{0r}} 
        \]
        and hence
        \begin{align*}
        l_{0l} - l_{0r}  &= \log \coth u_{0r} - \log \coth u_{0l} + O(e^{-R}) \\
        &= O(\epsilon/R).
        \end{align*}
        We can similarly prove that $l'_{0r} - l'_{1r} = O(\epsilon/R)$. 
        Hence $l'_r = R/2 + O(\epsilon/R)$,
        and likewise for $l'_l$. 
        From $l_{0r} = R/2 - \log\sinh 1 + O(\epsilon)$ and $u_{0r} = 1 + O(\epsilon)$ we derive, 
         from Lemma \ref{penta trig},
         that $\eta_0 = e^{-R/2 + C_1 + O(\epsilon)}$, 
         where $C_1 = \log(2 \cosh 1)$. 
\end{comment}

\begin{figure}[ht!]
\includegraphics[width=\linewidth]{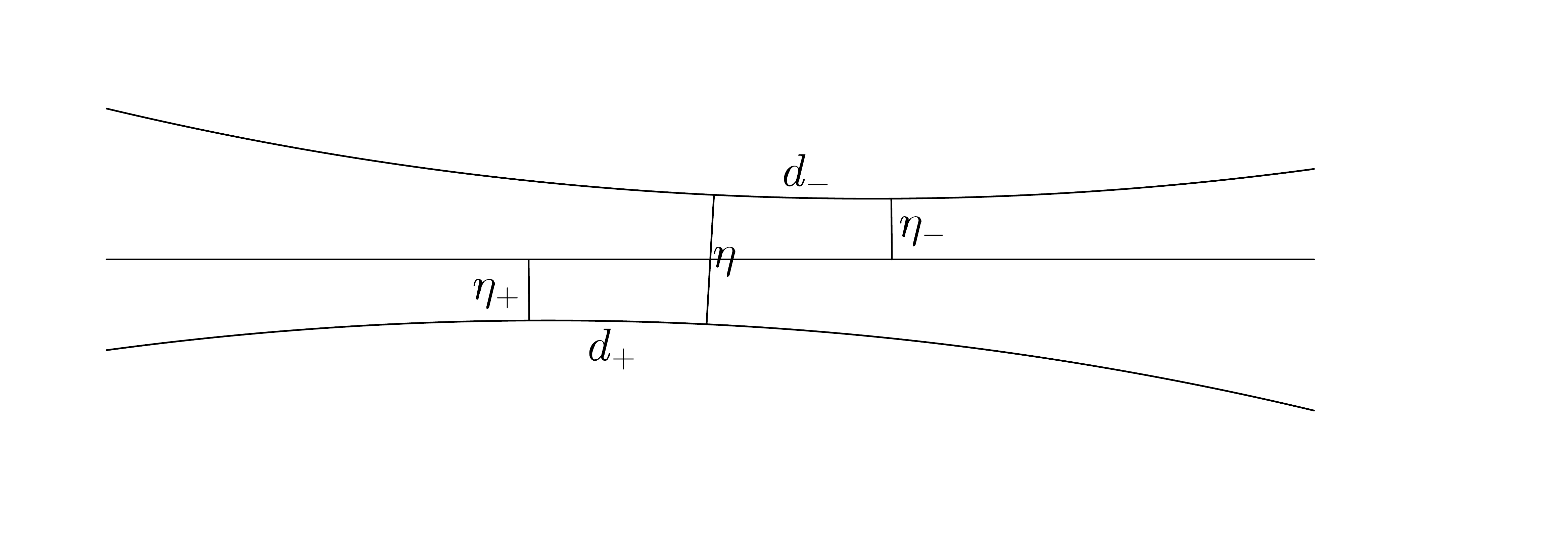}
\caption{}
\label{F:Step2}
\end{figure}

Now that we understand the geometry of the small orthogeodesics 
 from the railings of the hamster wheel to the inner boundary curves,
 our task is to relate it to the geometry of the small orthogeodesics connecting adjacent inner boundaries. 
Accordingly, we refer to Figure \ref{F:Step2} and we come up with 
$$
\cosh(d + i\pi) = 
\frac {\cosh \eta - \cosh \eta_- \cosh \eta_+}{\sinh \eta_- \sinh \eta_+},
$$
where $d = d_\C(\eta_+, \eta_-).$
We know that $d= O(\epsilon/R)$
 and $\eta_-, \eta_+ = e^{-R/2 + C_1 + O(\epsilon)}$,
% (where $C_1 = \log (2 \cosh1)$),
so we can conclude that 
\begin{equation} \label{est:eta2}
\eta = 2e^{-R/2 + C_1 +  O(\epsilon/R)}.
\end{equation}

We can then conclude (from the Hyperbolic Sine Law) that 
\begin{equation} \label{est:d}
d_-, d_+ = O(\epsilon/R).
\end{equation}
We can then verify Parts 1 through \ref{req:last} of the Theorem, as follows:
\begin{enumerate}
\item
Follows from \eqref{est:eta2}.
\item
Follows from \eqref{est:ue} applied to $e_l$ and $e_r$. 
\item
Follows from \eqref{est:d} and \eqref{est:lprime}, because these two equations provide an $O(\epsilon/R)$ estimate on the complex distance between the feet on $\gamma$ of the two short orthogeodesics to $\gamma$.
\item
Follows from \eqref{est:ue} because the feet of the \emph{rungs} lie within $O(\epsilon/R)$ of the slow and constant turning vector field, and $u_{0l}$, etc., determine the complex distance between the feet of the rungs and those of the medium orthogeodesics (on the outer cuffs). 
\item
Follows from \eqref{est:d} and \eqref{est:lprime} when $\gamma$ is inner cuff, and \eqref{est:ue} (invoked at most $2R$ times and summed) when $\gamma$ is an outer cuff. 
\end{enumerate}
\end{proof}

\new{We include for completeness the statements of the Sine and Cosine Laws for hyperbolic right angled hexagons, and the Cosine Law for right angled pentagons,
as Lemmas \ref{hex trig} and \ref{penta trig} below. }
We will use the convention that when calculating the complex length of an edge $d$ of a polygon, both of the adjacent edges are oriented to point towards $d$.
Lemma \ref{hex trig} is essentially the same as (6) and (3) in Section VI.2 of \cite{fenchel}; Lemma \ref{penta trig} is essentially the same as (2) and (5) in the same Section. 

\begin{lemma}[Cosine and sine laws for right angled hexagons] \label{hex trig}
Let $a, C, b, A, c, B$ be the complex edge lengths of a right angled hexagon in $\mathbb{H}^3$, in cyclic order as in Figure \ref{F:Hex}. Then 
$$\cosh A = \frac{\cosh b \cosh c + \cosh a}{\sinh b \sinh c},$$
and 
$$\frac{\sinh A}{\sinh a}=\frac{\sinh B}{\sinh b}=\frac{\sinh C}{\sinh c}.$$
\begin{figure}[ht!]
\includegraphics[width=0.4\linewidth]{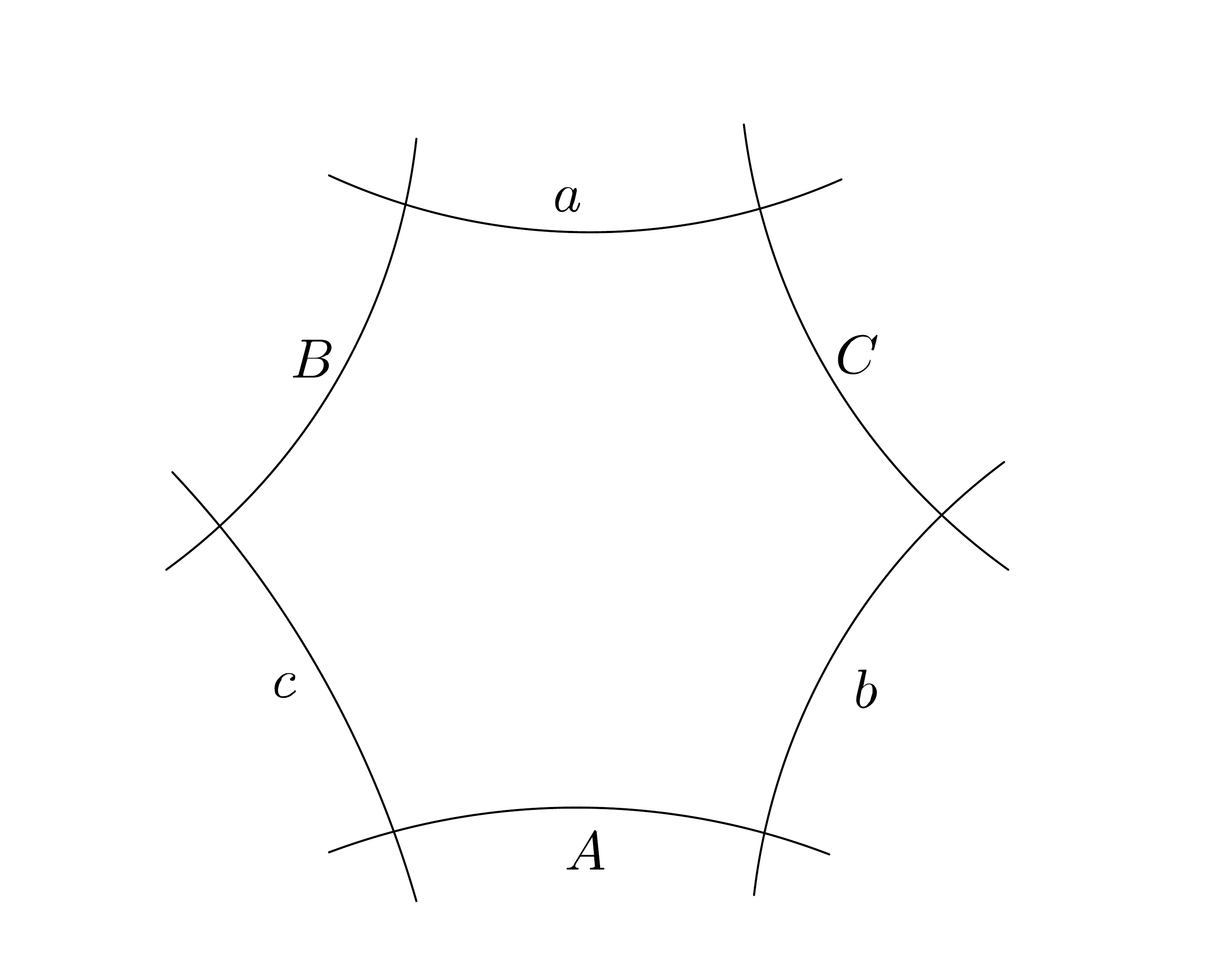}
\caption{}
\label{F:Hex}
\end{figure}
\end{lemma}

\begin{lemma}[Cosine  laws for right angled pentagons] \label{penta trig}
Let $a, b,c,d,e$ be the complex edge lengths of a right angled pentagon in $\mathbb{H}^3$, in cyclic order as in Figure \ref{F:Pent}. Then 
$$\cosh e = \sinh b \sinh c = \coth a \coth d.$$
\begin{figure}[ht!]
\includegraphics[width=0.4\linewidth]{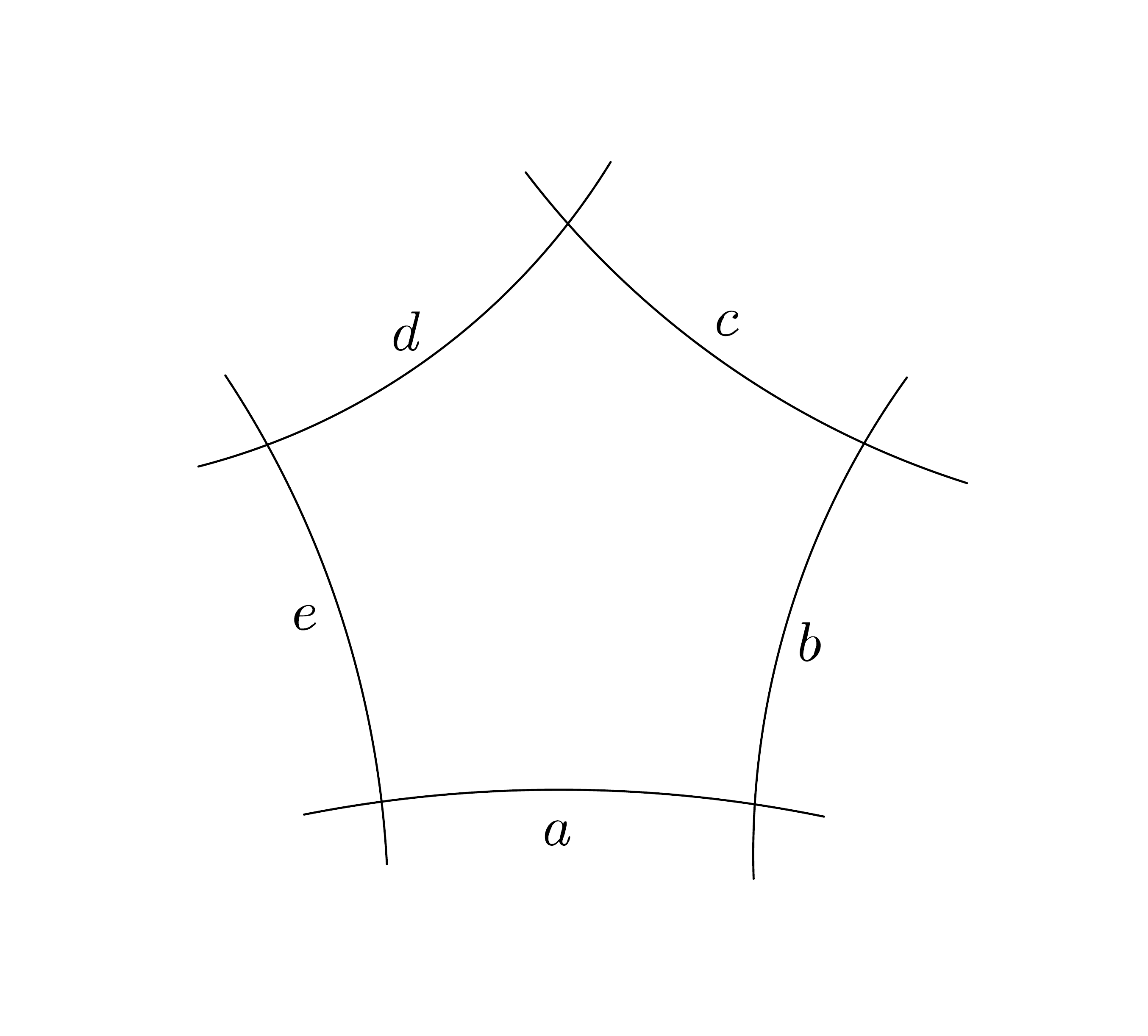}
\caption{}
\label{F:Pent}
\end{figure}
\end{lemma}

\bibliography{mybib}{}
\bibliographystyle{amsalpha}

\end{document}